\documentclass[12pt]{report}

\usepackage{amsmath,amssymb,amsfonts,amscd,array,amsthm}
\usepackage{pgf,tikz}
    \usetikzlibrary{arrows}
	\usetikzlibrary{3d}
    \usetikzlibrary{shapes}
    \usetikzlibrary{patterns}
    \usetikzlibrary{calc}
    \usetikzlibrary{decorations.markings}
\usepackage{graphicx}
\usepackage{xcolor}
\usepackage[arrow,matrix,poly,curve,cmtip,color]{xy}
	\xyoption{frame}
	\xyoption{all}
\usepackage{url}
\usepackage{mathdots}
\usepackage{enumerate}
\usepackage{enumitem}
	\setenumerate{listparindent=\parindent}
\usepackage{colortbl}
\usepackage{caption}
\usepackage[labelformat=simple,labelfont={}]{subcaption}
	
\usepackage{float}
\usepackage{fullpage}

\newtheorem{theorem}{Theorem}[section] 
\newtheorem{lemma}[theorem]{Lemma}
\newtheorem{proposition}[theorem]{Proposition}

\newtheorem{conjecture}[theorem]{Conjecture}

\theoremstyle{definition}
\newtheorem{definition}[theorem]{Definition}
\newtheorem{example}[theorem]{Example}

\newtheorem{remark}[theorem]{Remark}

\DeclareMathOperator{\C}{\mathrm{C}}
\DeclareMathOperator{\CFC}{\mathrm{CFC}}
\DeclareMathOperator{\supp}{supp}
\DeclareMathOperator{\cyclesupp}{supp_{cycle}}
\DeclareMathOperator{\FC}{FC}
\DeclareMathOperator{\Acyc}{Acyc}

\newcommand{\Z}{\mathbb{Z}}
\newcommand{\N}{\mathbb{N}}
\newcommand{\x}{\mathsf{x}}
\newcommand{\y}{\mathsf{y}}
\newcommand{\w}{\mathsf{w}}
\renewcommand{\u}{\mathsf{u}}
\newcommand{\gen}[1]{\langle #1 \rangle}

\renewcommand{\mapsto}{\longmapsto}

\newcommand{\ds}{\displaystyle}
\providecommand{\abs}[1]{\left\lvert#1\right\rvert}

\definecolor{orange}{RGB}{255,102,0}
\definecolor{ggreen}{RGB}{0,153,0}
\definecolor{darkblue}{RGB}{0,0,255}
\definecolor{purple}{RGB}{153,51,255}
\definecolor{turq}{RGB}{72,209,204}
\definecolor{gray}{RGB}{220,220,220}

\makeatletter
\newcommand{\boxbl}[1]{\textcolor{blue}{%
\tikz[baseline={([yshift=-1ex]current bounding box.center)}] \node [rectangle, minimum width=1ex,rounded corners,draw,line width=1.2pt] {\normalcolor\m@th$\displaystyle#1$};}}
\makeatother

\makeatletter
\newcommand{\boxor}[1]{\textcolor{orange}{%
\tikz[baseline={([yshift=-1ex]current bounding box.center)}] \node [rectangle, minimum width=1ex,rounded corners,draw,line width=1.2pt] {\normalcolor\m@th$\displaystyle#1$};}}
\makeatother

\makeatletter
\newcommand{\boxt}[1]{\textcolor{turq}{%
\tikz[baseline={([yshift=-1ex]current bounding box.center)}] \node [rectangle, minimum width=1ex,rounded corners,draw,line width=1.2pt] {\normalcolor\m@th$\displaystyle#1$};}}
\makeatother

\makeatletter
\newcommand{\boxgr}[1]{\textcolor{ggreen}{%
\tikz[baseline={([yshift=-1ex]current bounding box.center)}] \node [rectangle, minimum width=1ex,rounded corners,draw,line width=1.2pt] {\normalcolor\m@th$\displaystyle#1$};}}
\makeatother

\makeatletter
\newcommand{\boxp}[1]{\textcolor{purple}{%
\tikz[baseline={([yshift=-1ex]current bounding box.center)}] \node [rectangle, minimum width=1ex,rounded corners,draw,line width=1.2pt] {\normalcolor\m@th$\displaystyle#1$};}}
\makeatother

\makeatletter
\newcommand{\boxm}[1]{\textcolor{magenta}{%
\tikz[baseline={([yshift=-1ex]current bounding box.center)}] \node [rectangle, minimum width=1ex,rounded corners,draw,line width=1.2pt] {\normalcolor\m@th$\displaystyle#1$};}}
\makeatother

\makeatletter
\newcommand{\boxr}[1]{\textcolor{red}{%
\tikz[baseline={([yshift=-1ex]current bounding box.center)}] \node [rectangle, minimum width=1ex,rounded corners,draw,line width=1.2pt] {\normalcolor\m@th$\displaystyle#1$};}}
\makeatother

\makeatletter
\renewcommand{\boxed}[1]{\textcolor{black}{%
\tikz[baseline={([yshift=-1ex]current bounding box.center)}] \node [rectangle, minimum width=1ex,rounded corners,draw,line width=1.2pt] {\normalcolor\m@th$\displaystyle#1$};}}
\makeatother

\newcommand\xxaxis{0}
\newcommand\yyaxis{90}
\newcommand\sq[2]{
    \fill[fill=gray!30, draw=black, rounded corners, line width=1pt, shift={(\xxaxis:#1)}, shift={(\yyaxis:#2)}] 
    (0,0) -- (1,0) -- (1,-1) -- (0,-1) -- cycle; }

\newcommand\sqor[2]{
    \fill[draw=orange, fill=orange!10, line width=1.1pt, rounded corners, shift={(\xxaxis:#1)}, shift={(\yyaxis:#2)}]
    (0,0) -- (1,0) -- (1,-1) -- (0,-1) -- cycle; }
\newcommand\sqorhash[2]{
    \fill[pattern=north east lines, pattern color=orange!35, draw=orange, line width=1.1pt, rounded corners, shift={(\xxaxis:#1)}, shift={(\yyaxis:#2)}]
    (0,0) -- (1,0) -- (1,-1) -- (0,-1) -- cycle; }

\newcommand\sqgr[2]{
    \fill[draw=ggreen, fill=ggreen!05, line width=1.1pt, rounded corners, shift={(\xxaxis:#1)}, shift={(\yyaxis:#2)}]
    (0,0) -- (1,0) -- (1,-1) -- (0,-1) -- cycle; }
\newcommand\sqgrhash[2]{
    \fill[pattern=north east lines, pattern color=ggreen!35, draw=ggreen, line width=1.1pt, rounded corners, shift={(\xxaxis:#1)}, shift={(\yyaxis:#2)}]
    (0,0) -- (1,0) -- (1,-1) -- (0,-1) -- cycle; }
\newcommand\sqgrcheck[2]{
    \fill[pattern=checkerboard, pattern color=ggreen!06, draw=ggreen, line width=1.1pt, rounded corners, shift={(\xxaxis:#1)}, shift={(\yyaxis:#2)}]
    (0,0) -- (1,0) -- (1,-1) -- (0,-1) -- cycle; }

\newcommand\sqbl[2]{
    \fill[draw=darkblue, fill=darkblue!05, line width=1.1pt, rounded corners, shift={(\xxaxis:#1)}, shift={(\yyaxis:#2)}]
    (0,0) -- (1,0) -- (1,-1) -- (0,-1) -- cycle; }
\newcommand\sqblhash[2]{
    \fill[pattern=north east lines, pattern color=darkblue!35, draw=darkblue, line width=1.1pt, rounded corners, shift={(\xxaxis:#1)}, shift={(\yyaxis:#2)}]
    (0,0) -- (1,0) -- (1,-1) -- (0,-1) -- cycle; }
\newcommand\sqblcheck[2]{
    \fill[pattern=checkerboard, pattern color=darkblue!06, draw=darkblue, line width=1.1pt, rounded corners, shift={(\xxaxis:#1)}, shift={(\yyaxis:#2)}]
    (0,0) -- (1,0) -- (1,-1) -- (0,-1) -- cycle; }
    
\newcommand\sqm[2]{
    \fill[draw=magenta, fill=magenta!08, line width=1.1pt, rounded corners, shift={(\xxaxis:#1)}, shift={(\yyaxis:#2)}]
    (0,0) -- (1,0) -- (1,-1) -- (0,-1) -- cycle; }

\newcommand\bsq[2]{
    \fill[fill=white, dotted, draw=black, line width=0.5pt, rounded corners, shift={(\xxaxis:#1)},shift={(\yyaxis:#2)}]
    (0.05,-0.05) -- (0.95,-0.05) -- (0.95,-0.95) -- (0.05,-0.95) -- cycle; }

\begin{document}

\title{Conjugacy classes of cyclically fully commutative elements in Coxeter groups of type~$A$}
\author{Brooke Fox}


\begin{titlepage}
\ 

\vfill

\begin{center}
{\Large\textbf{Conjugacy classes of cyclically fully commutative\\
elements in Coxeter groups of type $A$}}

\vspace{.5cm}

MS Thesis, Northern Arizona University, 2014
\end{center}

\vspace{1cm}

\noindent Brooke Fox\\
Northern Arizona University\\
Department of Mathematics and Statistics\\
Northern Arizona University\\
Flagstaff, AZ 86011\\
\url{bkf23@nau.edu}\\
\\
Advisor: Dana C.~Ernst, PhD\\
Second Reader: Michael Falk, PhD\\
Third Reader: Stephen Wilson, PhD

\vfill

\end{titlepage}

\pagenumbering{roman}
\pagestyle{plain}


\chapter*{Abstract}

A fundamental result of Coxeter groups, known as Matsumoto's theorem, states that any two reduced expressions of the same element differ by a sequence of commutations and braid moves. If two elements have expressions that are cyclic shifts of each other (as words), then they are conjugate (as group elements). We say that an expression is cyclically reduced if every cyclic shift of it is reduced, and ask the following question, where an affirmative answer would be a ``cyclic version'' of Matsumoto's theorem. \emph{Do two cyclically reduced expressions of conjugate elements differ by a sequence of braid relations and cyclic shifts}? While the answer is, in general, ``no,'' understanding when the answer is ``yes'' is a central focus of a broad ongoing research project. It was recently shown to hold for all Coxeter elements.  

A Coxeter element is a special case of a fully commutative element, which is any element with the property that any two reduced expressions are equivalent by only commutations. In this thesis, we study the cyclically fully commutative elements. These are the elements for which every cyclic shift of any reduced expression is a reduced expression of a fully commutative element. In this light, the cyclically fully commutative elements are the ``cyclic version'' of the fully commutative elements. In particular, the cyclic version of Matsumoto's theorem for the cyclically fully commutative elements asks when two reduced expressions for conjugate elements are equivalent via only commutations and cyclic shifts.  

In this thesis, we study the combinatorics of cyclically fully commutative elements in Coxeter groups of type $A$ as it relates to conjugacy. In particular, we introduce the notion of cylindrical heaps and ring equivalence in order to state our main result, which says that two cyclically fully commutative elements of a Coxeter group of type $A$ are conjugate if and only if their corresponding cylindrical heaps are ring equivalent.


\tableofcontents
\listoffigures


\chapter{Preliminaries}

\pagenumbering{arabic}

\section{Introduction}
    This thesis is organized as follows.
    After necessary background material on Coxeter groups is presented in Section~\ref{sec:coxeter}, we introduce the class of fully commutative elements in Section~\ref{sec:FC}. Then, in Section~\ref{sec:heaps}, we discuss a visual representation for  elements of Coxeter groups, called heaps.
    The cyclically fully commutative elements, introduced in Section~\ref{sec:CFC}, are exactly those elements that are fully commutative when written in a circle and can be thought of as a generalization of Coxeter elements.
    In Section~\ref{sec:Sn}, we explore the connection between Coxeter groups of type $A_n$ and the symmetric group $S_{n+1}$. We also state a conjecture about the permutations corresponding of cyclically fully commutative elements (Conjecture~\ref{conjecture}).
    Finally, in Section~\ref{sec:chunks}, we introduce the notion of cylindrical heaps and ring equivalence in order to state the main result of this thesis (Theorem~\ref{thm:conjiffring}), which says that two cyclically fully commutative elements of a Coxeter group of type $A_n$ are conjugate if and only if their corresponding cylindrical heaps are ring equivalent.
    The last section states and proves several lemmas used to prove the main result.

\section{Coxeter groups}\label{sec:coxeter}
    A \emph{Coxeter system} is a pair $(W,S)$ consisting of a finite set $S$ of generating involutions and a group $W$, called a \emph{Coxeter group}, with presentation
    $$W = \gen{S \mid (st)^{m(s,t)} = e ~\text{for}~ m(s,t) < \infty },$$
    where $e$ is the identity, $m(s,t) = 1$ if and only if $s = t$, and $m(s,t) = m(t,s)$.
    It follows that the elements of $S$ are distinct as group elements and that $m(s,t)$ is the order of $st$~\cite{Humphreys1990}.
    We call $m(s,t)$ the \emph{bond strength} of $s$ and $t$.
    Coxeter groups are generalizations of reflection groups. Each generator $s \in S$ can be thought of as a reflection. Recall that the composition of two reflections is a rotation by twice the angle between the corresponding hyperplanes. So, if $s,t \in S$, we can think of $st$ as a rotation, where $m(s,t)$ is the order of the rotation.

    Since elements of $S$ have order two, the relation $(st)^{m(s,t)} = e$ can be written as
\begin{equation}\label{braid} \underbrace{sts \cdots}_{m(s,t)} = \underbrace{tst \cdots}_{m(s,t)} \end{equation}
    with $m(s,t) \geq 2$ factors.
    If $m(s,t) = 2$, then $st = ts$ is called a \emph{commutation relation} since $s$ and $t$ commute. If $m(s,t) \geq 3$, then the relation in \eqref{braid} is called a \emph{braid relation}.
    We will write $\gen{st}_{m(s,t)}$ to denote the word $sts \cdots$ consisting of $m(s,t)$ factors.
    Replacing $\gen{st}_{m(s,t)}$ with $\gen{ts}_{m(s,t)}$ will be referred to as a \emph{commutation} if $m(s,t) = 2$ and a \emph{braid move} if $m(s,t) \geq 3$.
    
    We can represent the Coxeter system $(W,S)$ with a unique \emph{Coxeter graph} $\Gamma$ having
\begin{enumerate}[leftmargin=0.75in,label=(\alph*)]
    \item vertex set $S = \{s_1, \ldots, s_n\}$ and
    \item edges $\{s_i,s_j\}$ for each $m(s_i,s_j) \geq 3$.
\end{enumerate} 
    Each edge $\{s_i,s_j\}$ is labeled with its corresponding bond strength $m(s_i,s_j)$. Since bond strength 3 is the most common, we typically omit the labels of 3 on those edges.
    
    There is a one-to-one correspondence between Coxeter systems and Coxeter graphs.
    Given a Coxeter graph $\Gamma$, we can construct the corresponding Coxeter system $(W,S)$.
    In this case, we say that $(W,S)$, or just $W$, is of type $\Gamma$. If $(W,S)$ is of type $\Gamma$, for emphasis, we may write $(W,S)$ as $(W(\Gamma),S(\Gamma))$.
    Note that generators $s_i$ and $s_j$ are connected by an edge in the Coxeter graph $\Gamma$ if and only if $s_i$ and $s_j$ do not commute~\cite{Humphreys1990}.
    Also, if $\Gamma$ is connected, then we say that $\Gamma$, or $W(\Gamma)$, is \emph{irreducible}.

    The Coxeter system of type $A_n$ is given by the Coxeter graph in Figure~\ref{fig:A}. We can construct $(W(A_n),S(A_n))$ having the generating set $S(A_n) = \{s_1, s_2, \ldots, s_n\}$ and defining relations
\begin{enumerate}[leftmargin=0.75in, label=(\alph*)]
    \item $s_is_i = e$ for all $i$;
    \item $s_is_j = s_js_i$ when $\abs{i-j} > 1$;
    \item $s_is_js_i = s_js_is_j$ when $\abs{i-j} = 1$.
\end{enumerate}
    The Coxeter group $W(A_n)$ is isomorphic to the symmetric group $S_{n+1}$ under the mapping that sends $s_i$ to the adjacent transposition $(i~i+1)$.
    This thesis focuses specifically on Coxeter systems of type $A_n$.

\begin{definition} Let $S^*$ denote the free monoid over $S$. If a word $\w=s_{x_1}s_{x_2}\cdots s_{x_m}\in S^*$ is equal to $w$ when considered as an element of $W$, we say that $\w$ is an \emph{expression} for $w$.
    (Expressions will be written in {\sf sans serif} font for clarity.) Furthermore, if $m$ is minimal among all possible expressions for $w$, we say that $\w$ is a \emph{reduced expression} for $w$, and we call $m$ the \emph{length} of $w$, denoted $\ell(w)$.
\end{definition}

    Each element $w \in W$ can have several different reduced expressions that represent it.
    The following theorem is called Matsumoto's Theorem.

\begin{theorem}[Matsumoto,~\cite{Boothby2012}] \label{thm:matsumoto} In a Coxeter group $W$, any two reduced expressions for the same group element differ by a sequence of commutations and braid moves. \qed
\end{theorem}

    It follows from Matsumoto's Theorem that all reduced expressions for $w \in W$ have the same number of generators appearing in the expression.
    Let $w \in W$ and let $\w$ be a reduced expression for $w$. Then the \emph{support} of $\w$, denoted $\supp(\w)$, is the set of generators that appear in $\w$.
    Also from Matsumoto's Theorem we have that $s$ appears in a reduced expression for $w$ if and only if $s$ appears in every reduced expression for $w$, so we can define the support of a group element.
    Define $\supp(w)$ to be the set of generators appearing in any reduced expression for $w$.
    If $\supp(w) = S$, we say that $w$ has \emph{full support}.
    
    Given a reduced expression $\w$ for $w \in W$, we define a \emph{subexpression} of $\w$ to be any subsequence of $\w$. We will refer to a subexpression consisting of a string of consecutive symbols from $\w$ as a \emph{subword} of $\w$.

\begin{example}\label{ex:subword} Let $w \in W(A_6)$ and let $\w = s_1 s_2 s_4 s_5 s_2 s_6 s_5$ be an expression for $w$. Then we have $$s_1 \textcolor{magenta}{s_2 s_4} s_5 s_2 s_6 s_5
    = s_1 s_4 \textcolor{magenta}{s_2 s_5} s_2 s_6 s_5
    = s_1 s_4 s_5 \textcolor{ggreen}{s_2 s_2} s_6 s_5
    = s_1 s_4 s_5 s_6 s_5,$$
    where the \textcolor{magenta}{pink} subword denotes applying a commutation to the corresponding generators to obtain the next expression and the \textcolor{ggreen}{green} subword denotes canceling two adjacent occurrences of the same generator.
    So, $\w$ is not reduced. It turns out that $s_1 s_4 s_5 s_6 s_5$ is a reduced expression for $w$ and $\supp(w) = \{s_1,s_2,s_4,s_5,s_6\}$. Hence $\ell(w) = 5$.
\end{example}

\begin{example} Let $W$ be the Coxeter group of type $A_4$, and let $w \in W$ have reduced expression $\w = s_1s_2s_3s_4s_2$. Then the set of all the reduced expressions for $w$ is
    $$\{s_1s_2s_3\textcolor{magenta}{s_4s_2}, s_1\textcolor{blue}{s_2s_3s_2}s_4, \textcolor{magenta}{s_1s_3}s_2s_3s_4, s_3s_1s_2s_3s_4\},$$
    where the \textcolor{magenta}{pink} subword denotes applying a commutation and the \textcolor{blue}{blue} subword denotes applying a braid relation to get to the next reduced expression. Then $\ell(w) = 5$ and $w$ has full support.
\end{example}
    
\begin{definition} A \emph{Coxeter element} is an element $w \in W$ for which every generator appears exactly once in each reduced expression for $w$. 
\end{definition}

    Note that $\supp(w) = S$ for a Coxeter element $w$. The set of Coxeter elements of $W$ is denoted by $\C(W)$.
    
\begin{example} \label{ex:Coxelt} Consider the Coxeter group of type $A_4$. Let $w_1, w_2, w_3, w_4 \in W(A_4)$ have reduced expressions $s_1s_2s_4s_3$, $s_2s_1s_3s_4$, $s_2s_4$, and $s_1s_2s_3s_4s_1s_2$, respectively.
    Then $w_1$ and $w_2$ are Coxeter elements because each has exactly one occurrence of each generator $s_1,s_2,s_3,s_4$ in its reduced expression.
    On the other hand, $w_3$ is not a Coxeter element because it does not have full support. Also, $w_4$ is not a Coxeter element because it has generators repeated; there are two occurrences each of $s_1$ and $s_2$ in its reduced expression.
\end{example}

\begin{example} Let $W$ be the Coxeter group of type $A_4$. Then the Coxeter elements of $W$ and their corresponding reduced expressions are shown in Figure~\ref{fig:coxeltsinA4}, where each column contains the reduced expressions for a single Coxeter element.
    There are $4! = 24$ reduced expressions for Coxeter elements in $W$, but there are only 8 Coxeter elements because some reduced expressions determine the same group element by commutation.
\begin{figure}[h!] \centering
$$\begin{array}{llllllll}
    s_1s_2s_3s_4 & s_4s_3s_2s_1 & s_1s_2s_4s_3 & s_2s_1s_3s_4 & s_3s_4s_2s_1 & s_4s_3s_1s_2 & s_1s_3s_2s_4 & s_2s_1s_4s_3 \\
         &      & s_1s_4s_2s_3 & s_2s_3s_1s_4 & s_3s_2s_4s_1 & s_4s_1s_3s_2 & s_3s_1s_2s_4 & s_2s_4s_1s_3 \\
         &      & s_4s_1s_2s_3 & s_2s_3s_4s_1 & s_3s_2s_1s_4 & s_1s_4s_3s_2 & s_1s_3s_4s_2 & s_2s_4s_3s_1 \\
         &      &      &      &      &      & s_3s_1s_4s_2 & s_4s_2s_1s_3 \\
         &      &      &      &      &      & s_3s_4s_1s_2 & s_4s_2s_3s_1
\end{array}$$
\caption{Coxeter elements and their reduced expressions in $W(A_4)$.} \label{fig:coxeltsinA4}
\end{figure}
\end{example}

\section{Fully commutative elements}\label{sec:FC}
    Let $(W,S)$ be a Coxeter system of type $\Gamma$ and let $w \in W$. Following~\cite{Stembridge1996}, we define a relation $\sim$ on the set of reduced expressions for $w$. Let $\w$ and $\w'$ be two reduced expressions for $w$. 
    We define $\w \sim \w'$ if we can obtain $\w'$ from $\w$ by applying a single commutation move of the form $s_is_j \mapsto s_js_i$, where $m(s_i,s_j) = 2$.
    Now, define the equivalence relation $\approx$ by taking the reflexive transitive closure of $\sim$. Each equivalence class under $\approx$ is called a \emph{commutation class}.
    Two reduced expressions are said to be \emph{commutation equivalent} if they are in the same commutation class.

\begin{example} \label{ex:comm_eq} Let $W$ be the Coxeter group of type $A_5$ and consider the reduced expressions $\w = s_1s_3s_2s_5s_4$ and $\w' = s_5s_1s_3s_4s_2$. Then $\w$ and $\w'$ are reduced expressions for the same element $w \in W$ and are commutation equivalent since
    $$s_1s_3\textcolor{magenta}{s_2s_5}s_4 = s_1s_3s_5\textcolor{magenta}{s_2s_4} = s_1\textcolor{magenta}{s_3s_5}s_4s_2 = \textcolor{magenta}{s_1s_5}s_3s_4s_2 = s_5s_1s_3s_4s_2,$$
    where the \textcolor{magenta}{pink} subwords denote applying a commutation to the corresponding generators to obtain the next expression.
\end{example}

\begin{example}\label{ex:nonFC} Let $W$ be the Coxeter group of type $A_4$ and let $w \in W$ have reduced expressions $\w = s_1 s_2 s_3 s_2 s_4$ and $\w' = s_1 s_3 s_2 s_3 s_4$.
    Then it is easily seen that $\w$ and $\w'$ are not commutation equivalent, so $w$ has more than one commutation class. Specifically, the commutation classes are $$\{s_1s_2s_3s_2s_4, s_1s_2s_3s_4s_2\} ~\text{and}~ \{s_1s_3s_2s_3s_4, s_3s_1s_2s_3s_4\}.$$
\end{example}

\begin{example}\label{ex:FC} Let $W$ be the Coxeter group of type $A_3$ and let $w \in W$ have reduced expression $\w = s_2s_1s_3s_2$. Then, by applying the commutation $s_1s_3 \mapsto s_3s_1$, $\w' = s_2s_3s_1s_2$ is also a reduced expression for $w$.
    There are no other reduced expressions for $w$ because we cannot apply any other commutations or braid moves. Therefore there is exactly one commutation class---namely, $\{s_2s_1s_3s_2, s_2s_3s_1s_2\}$.
\end{example}
    
    If $w$ has exactly one commutation class, then we say that $w$ is \emph{fully commutative}, or just FC.
    The set of all fully commutative elements of $W$ is denoted by $\FC(\Gamma)$, where $\Gamma$ is the corresponding Coxeter graph, or $\FC(W)$.
    For consistency, we say that a reduced expression $\w$ is FC if it is a reduced expression for some $w \in\FC(\Gamma)$.
    Note that the element in Example~\ref{ex:nonFC} is not FC since there are two commutation classes, while the element in Example~\ref{ex:FC} is FC since there is only one commutation class.

    Given some $w\in\FC(\Gamma)$ and a starting reduced expression for $w$, observe that the definition of fully commutative states that one only needs to perform commutations to obtain all the reduced expression for $w$, but the following theorem states that, when $w$ is FC, performing commutations is the only possible way to obtain another reduced expression for $w$.

\begin{theorem}[Stembridge,~\cite{Stembridge1996}] \label{thm:stem} An element $w \in W$ is FC if and only if no reduced expression for $w$ contains $\gen{s_i,s_j}_{m(s_i,s_j)}$ as a subword for all $s_i \neq s_j$ when $m(s_i,s_j) \geq 3$. \qed
\end{theorem}

    This theorem states that an element is FC if and only if there is no opportunity to apply a braid move.
    Notice that Coxeter elements are FC since there will never be opportunity to apply braid moves as, by definition, there is exactly one appearance of each generator.

\begin{example} \label{ex:FC2} Let $W$ be the Coxeter group of type $A_5$. Let $w \in W$ have reduced expression $\w = s_1 s_4 s_3 s_5 s_2 s_1 s_3 s_4$. Then we have 
    $$s_1 s_4 \textcolor{magenta}{s_3 s_5} s_2 s_1 s_3 s_4 = s_1 s_4 s_5 s_3 s_2 \textcolor{magenta}{s_1 s_3} s_4 = s_1 s_4 s_5 \textcolor{blue}{s_3 s_2 s_3} s_1 s_4,$$
    where the \textcolor{magenta}{pink} subword denotes applying a commutation to the corresponding generators to obtain the next expression.
    So, $w$ is not FC because there is opportunity to apply a braid move, highlighted in \textcolor{blue}{blue}.
\end{example}

    Stembridge classified the irreducible Coxeter groups that contain only finitely many fully commutative elements, called the \emph{FC-finite Coxeter groups}.
    This thesis is mainly concerned with $W(A_n)$, which is a finite group, so it has finitely many FC elements. However, there exist infinite Coxeter groups that contain only finitely many FC elements. 
    For example, Coxeter groups of type $E_n$ with $n \geq 9$ as shown in Figure~\ref{fig:E} are infinite, but they have only finitely many FC elements.

\begin{theorem}[Stembridge,~\cite{Stembridge1996}] \label{thm:FCfinite} The FC-finite irreducible Coxeter groups are of type $A_n$ with $n \geq 1$, $B_n$ with $n \geq 2$, $D_n$ with $n \geq 4$, $E_n$ with $n \geq 6$, $F_n$ with $n \geq 4$, $H_n$ with $n \geq 3$, and $I_2(m)$ with $5 \leq m < \infty$. The corresponding Coxeter graphs are shown in Figure~\ref{fig:coxgraphs}. \qed
\end{theorem}

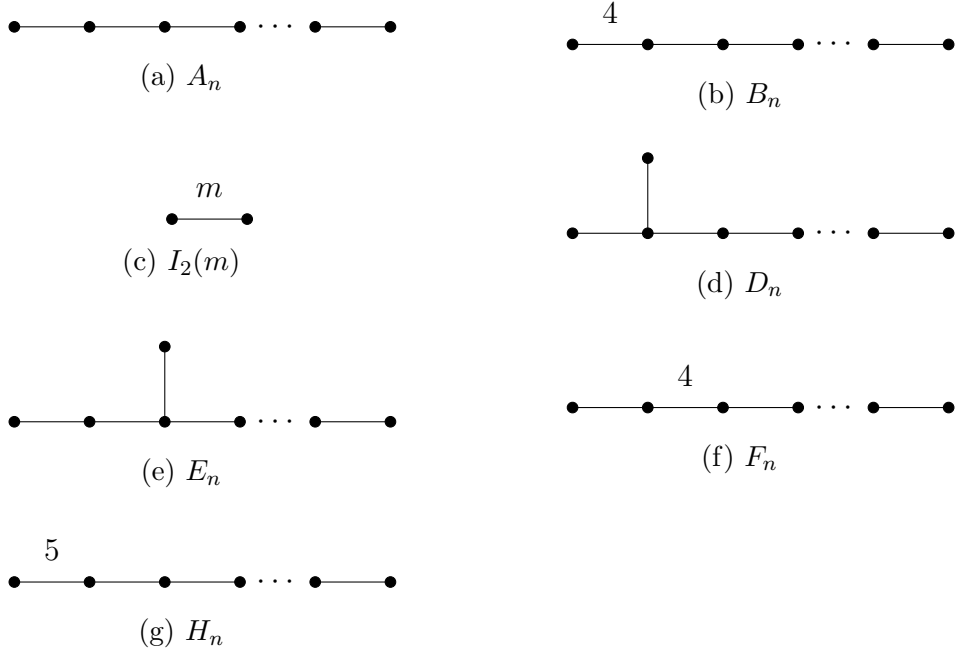
\begin{figure}[h!]
\begin{tabular}{m{7cm} m{7cm}}
\begin{subfigure}{0.5\textwidth} \centering
\begin{tikzpicture}[scale=1.0]
\draw[fill=black] \foreach \x in {1,2,...,6} {(\x,10) circle (2pt)};
\draw {(.5,10) node{}
(4.5,10) node{$\cdots$}
[-] (1,10) -- (4,10)
[-] (5,10) -- (6,10)
(1,10) node{}}; 
\end{tikzpicture}
\caption{$A_{n}$} \label{fig:A}
\end{subfigure} &

\begin{subfigure}{0.5\textwidth} \centering
\begin{tikzpicture}[scale=1.0]
\draw [fill=black] \foreach \x in {1,2,...,6} {(\x,8.5) circle (2pt)};
\draw {(.5,8.5) node{}
(1.5,8.5) node[label=above:$4$]{}
(4.5,8.5) node{$\cdots$}
[-] (1,8.5) -- (4,8.5)
[-] (5,8.5) -- (6,8.5)
(2,8.5) node{}}; 
\end{tikzpicture}
\caption{$B_{n}$} \label{fig:B}
\end{subfigure} \\

    & \\ 

\begin{subfigure}{0.5\textwidth} \centering
\begin{tikzpicture}[scale=1.0]
\draw[fill=black] \foreach \x in {1,2} {(\x,0) circle (2pt)};
\draw {(.25,0) node{}
(1.5,0) node[label=above:$m$]{}
[-] (1,0) -- (2,0)
(2,0) node{}};
\end{tikzpicture}
\caption{$I_{2}(m)$} \label{fig:I}
\end{subfigure} &

\begin{subfigure}{0.5\textwidth} \centering
\begin{tikzpicture}[scale=1.0]
\draw[fill=black] \foreach \x in {1,2,...,6} {(\x,6.5) circle (2pt)};
\draw[fill=black] (2,7.5) circle (2pt);
\draw {(.5,6.5) node{}
(4.5,6.5) node{$\cdots$}
[-] (1,6.5) -- (4,6.5)
[-] (5,6.5) -- (6,6.5)
[-] (2,6.5) -- (2,7.5)
(2,6.5) node{}};
\end{tikzpicture}
\caption{$D_{n}$} \label{fig:D}
\end{subfigure} \\

    & \\ 
    
\begin{subfigure}{0.5\textwidth} \centering
\begin{tikzpicture}[scale=1.0]
\draw[fill=black] \foreach \x in {1,2,...,6} {(\x,4.5) circle (2pt)};
\draw[fill=black] (3,5.5) circle (2pt);
\draw {(.5,4.5) node{}
(4.5,4.5) node{$\cdots$}
[-] (1,4.5) -- (4,4.5)
[-] (5,4.5) -- (6,4.5)
[-] (3,4.5) -- (3,5.5)
(3,4.5) node{}};
\end{tikzpicture}
\caption{$E_{n}$} \label{fig:E}
\end{subfigure} &

\begin{subfigure}{0.5\textwidth} \centering
\begin{tikzpicture}[scale=1.0]
\draw[fill=black] \foreach \x in {1,2,...,6} {(\x,3) circle (2pt)};
\draw {(.5,3) node{}
(2.5,3) node[label=above:$4$]{}
(4.5,3) node{$\cdots$}
[-] (1,3) -- (4,3)
[-] (5,3) -- (6,3)
(3,3) node{}};
\end{tikzpicture}
\caption{$F_{n}$} \label{fig:F}
\end{subfigure} \\

    & \\ 

\begin{subfigure}{0.5\textwidth} \centering
\begin{tikzpicture}[scale=1.0]
\draw[fill=black] \foreach \x in {1,2,...,6} {(\x,1.5) circle (2pt)};
\draw {(.5,1.5) node{}
(1.5,1.5) node[label=above:$5$]{}
(4.5,1.5) node{$\cdots$}
[-] (1,1.5) -- (4,1.5)
[-] (5,1.5) -- (6,1.5)
(2,1.5) node{}}; 
\end{tikzpicture}
\caption{$H_{n}$} \label{fig:H}
\end{subfigure}
\end{tabular}
\caption{Coxeter graphs corresponding to the irreducible FC-finite Coxeter groups.}
\label{fig:coxgraphs}
\end{figure}

    It is well known that the number of FC elements in $W(A_{n-1})$ is given by the Catalan number $C_n = \frac{1}{n+1} \binom{2n}{n}$.

\section{Heaps}\label{sec:heaps}
    We can now discuss another representation of elements of Coxeter groups. Each reduced expression is associated with a labeled partially ordered set called a heap. 
    We follow the development in~\cite{Ernst2010} and~\cite{Stembridge1996}.

\begin{definition} \label{def:heap} Let $(W,S)$ be a Coxeter system. Suppose $\w = s_{x_1} s_{x_2} \cdots s_{x_k}$ is a reduced expression for $w \in W$, and as in~\cite{Stembridge1996}, define a partial ordering $\prec$ on the indices $\{1,\ldots,k\}$ by the transitive closure of the relation $j \prec i$ if $i < j$ and $s_{x_i}$ and $s_{x_j}$ do not commute.
    In particular, $j \prec i$ if $i < j$ and $s_{x_i} = s_{x_j}$ by transitivity and the fact that $\w$ is reduced.
    This partial order with $i$ labeled $s_{x_i}$ is called the \emph{heap} of $\w$.
\end{definition}

    Note that for simplicity, we are omitting the labels of the underlying poset but retaining the labels of the corresponding generators.

\begin{example}\label{ex:firstheap} Let $\w = s_2 s_1 s_3 s_2 s_4 s_5$ be a reduced expression for $w \in W(A_5)$.  We see that $\w$ is indexed by $\{1, 2, 3, 4, 5, 6\}$ because $\ell(w)=6$. We see that $4 \prec 3$ since $3 < 4$ and $s_4$ and $s_3$ do not commute.
    The labeled Hasse diagram for the heap poset of $\w$ is shown in Figure~\ref{fig:hasse}.
\begin{center} 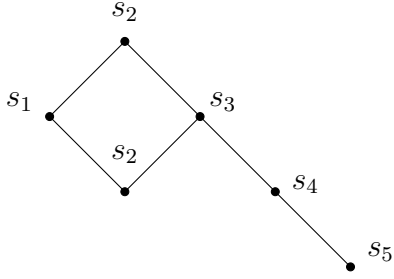
\begin{figure}[h!] \centering
\begin{tikzpicture}
\draw (0,2)--(-1,1); \draw (0,2)--(1,1); \draw (-1,1)--(0,0); \draw (0,0)--(1,1); \draw (1,1)--(2,0); \draw (2,0)--(3,-1);
\draw [fill=black] (0,2) circle (1.5pt); \draw[color=black] (0,2.4) node {$s_2$};
\draw [fill=black] (-1,1) circle (1.5pt); \draw[color=black] (-1.4,1.2) node {$s_1$};
\draw [fill=black] (1,1) circle (1.5pt); \draw[color=black] (1.3,1.2) node {$s_3$};
\draw [fill=black] (0,0) circle (1.5pt); \draw[color=black] (0,0.5) node {$s_2$};
\draw [fill=black] (2,0) circle (1.5pt); \draw[color=black] (2.4,0.1) node {$s_4$};
\draw [fill=black] (3,-1) circle (1.5pt); \draw[color=black] (3.4,-0.8) node {$s_5$};
\end{tikzpicture}
\caption{The labeled Hasse diagram for the heap poset of $w = s_2 s_1 s_3 s_2 s_4 s_5$.} \label{fig:hasse}
\end{figure} \end{center}
\end{example}

    Let $\w$ be a fixed reduced expression for $w \in W(A_n)$. As in~\cite{Billey2007} and~\cite{Ernst2010}, we represent a heap for $\w$ as a set of lattice points embedded in $\{1,\ldots,n\} \times \N$.
    To do so, we assign (not necessarily unique) coordinates $(x,y) \in \{1,\ldots,n\} \times \N$ to each entry of the labeled Hasse diagram for the heap of $\w$ in such a way that
\begin{enumerate}[leftmargin=0.75in, label=(\alph*)]
    \item An entry with coordinates $(x,y)$ is labeled $s_i$ (or $i$) in the heap if and only if $x = i$; 
    \item An entry with coordinates $(x,y)$ is greater than an entry with coordinates $(x',y')$ in the heap if and only if $y > y'$.
\end{enumerate}
    It follows from the definition that there is an edge in the Hasse diagram from $(x,y)$ to $(x',y')$ if and only if $x = x' \pm 1$, $y > y'$, and there are no entries $(x'', y'')$ such that $x'' \in \{x, x'\}$ and $y'< y'' < y$.
    This implies that we can completely reconstruct the edges of the Hasse diagram and the corresponding heap poset from a lattice point representation.
    The lattice point representation of a heap allows us to visualize potentially cumbersome arguments. Note that our heaps are upside-down versions of the heaps that appear in~\cite{Billey2007} and several other papers. That is, in this thesis entries on top of a heap correspond to generators occurring to the left, as opposed to the right, in the corresponding reduced expression.
    One can form similar lattice point representations for heaps when $\Gamma$ is a straight line Coxeter graph.
    
    Let $\w = s_{x_1} \cdots s_{x_n}$ be any reduced expression for $w \in W(A_{n})$. We let $H(\w)$ denote a lattice representation of the heap poset in $\{1,\ldots,n\} \times \N$ described in the paragraph above.
    There are many possible coordinate assignments for the entries of $H(\w)$, yet the $x$-coordinates for each entry will be fixed. If $s_{x_i}$ and $s_{x_j}$ are adjacent generators in the Coxeter graph with $i<j$, then we must place the point labeled by $s_{x_i}$ at a level that is \emph{above} the level of the point labeled by $s_{x_j}$.
    In particular, two entries labeled by the same generator may only differ by the amount of vertical space between them while maintaining their relative vertical position to adjacent entries in the heap.
    
    Because generators that are not adjacent in the Coxeter graph commute, points whose $x$-coordinates differ by more than one can slide past each other or land at the same level.
    To visualize the labeled heap poset of a lattice representation we will enclose each entry of the heap in a block in such a way that if one entry covers another, the blocks overlap halfway.

\begin{remark} It follows from Proposition 2.2 in~\cite{Stembridge1996} that heaps are well-defined up to commutation class.
    That is, if $\w$ and $\w'$ are two reduced expressions for $w \in W$ that are in the same commutation class, then the labeled heaps of $\w$ and $\w'$ are equal.
    In particular, if $w$ is FC, then it has a single commutativity class, and so there is a unique heap associated to $w$.
    In this case, if $w$ is FC, then we may write $H(w)$ to denote the heap of any reduced expression for $w$.
\end{remark}

    There are potentially many different ways to represent a heap, each differing by the vertical placement of blocks. For example, we can place blocks in vertical positions that are as high as possible, as low as possible, or some combination of high/low. In this thesis, we choose what we view to be the best representation of the heap for each example.

\begin{example}\label{ex:heap} Let $W$ be the Coxeter group of type $A_5$ and $\w = s_1s_2s_3s_1s_2s_4s_5$ be a reduced expression for $w \in \FC(W)$.
    We will construct one possible lattice point representation for $H(w)$.
    Starting from the right hand side, the first generator is $s_5$, so we place a block, labeled with a 5, in position $(5,1)$. Observe that the $x$-coordinate is forced to be 5 since the block corresponds to the generator $s_5$, but we have a choice for the $y$-coordinate. We choose to place it as low as possible to get
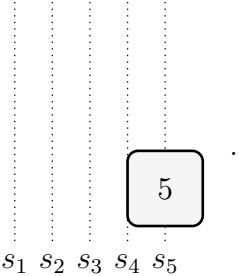
\begin{center} \begin{tabular}{m{2.75cm} m{0.5cm}} \begin{tikzpicture}
    \node at (0.5,-1.5) {$s_1$}; \node at (1,-1.5) {$s_2$}; \node at (1.5,-1.5) {$s_3$}; \node at (2,-1.5) {$s_4$}; \node at (2.5,-1.5) {$s_5$};
    \draw[dotted, line width=0.5pt] (0.5,-1.2) -- (0.5,2);
    \draw[dotted, line width=0.5pt] (1,-1.2)   -- (1,2);
    \draw[dotted, line width=0.5pt] (1.5,-1.2) -- (1.5,2);
    \draw[dotted, line width=0.5pt] (2,-1.2)   -- (2,2);
    \draw[dotted, line width=0.5pt] (2.5,-1.2) -- (2.5,2);
    \sq{2}{0}; \node at (2.5,-0.5) {$5$};
\end{tikzpicture} & . 
\end{tabular} \end{center}

    \noindent Now, moving right to left, the next generator is $s_4$, so, similarly, we place a block, labeled with a 4, in position $(4,2)$. The $x$-coordinate must be 4, and we must place it at $y \geq 2$ because $s_4$ and $s_5$ do not commute, and so the $s_4$ block will be on a level above $s_5$, overlapping it halfway. We choose to place it as low as possible to get
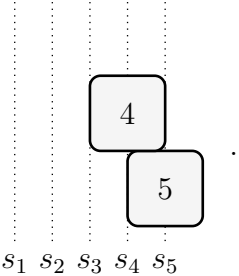
\begin{center} \begin{tabular}{m{2.75cm} m{0.5cm}} \begin{tikzpicture}
    \node at (0.5,-1.5) {$s_1$}; \node at (1,-1.5) {$s_2$}; \node at (1.5,-1.5) {$s_3$}; \node at (2,-1.5) {$s_4$}; \node at (2.5,-1.5) {$s_5$};
    \draw[dotted, line width=0.5pt] (0.5,-1.2) -- (0.5,2);
    \draw[dotted, line width=0.5pt] (1,-1.2)   -- (1,2);
    \draw[dotted, line width=0.5pt] (1.5,-1.2) -- (1.5,2);
    \draw[dotted, line width=0.5pt] (2,-1.2)   -- (2,2);
    \draw[dotted, line width=0.5pt] (2.5,-1.2) -- (2.5,2);
    \sq{2}{0};   \node at (2.5,-0.5) {$5$};
    \sq{1.5}{1}; \node at (2,0.5)    {$4$};
\end{tikzpicture} & . 
\end{tabular} \end{center}

    \noindent The next two generators, moving to the left, are $s_2$ and $s_1$.  Since $s_2$ and $s_4$ commute, we place the corresponding blocks on the same level as each other (or with the same $y$-coordinate). We could have placed the $s_2$ block lower in the heap since there is nothing blocking it, but we choose to place it on the same level as the $s_4$ block because they commute.
    Since the $s_4$ block has 2 as its $y$-coordinate, we place the $s_2$ block in position $(2,2)$.
    Since $s_2$ and $s_1$ do not commute, we must place the $s_1$ block above the $s_2$ block. We choose to place the $s_1$ block in position $(1,3)$.
    We get
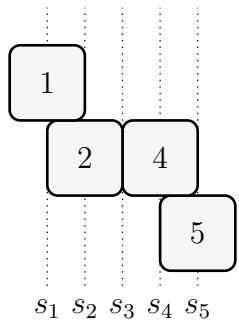
\begin{center} \begin{tabular}{m{2.9cm} m{0.5cm}} 
\begin{tikzpicture}
    \node at (0.5,-1.5) {$s_1$}; \node at (1,-1.5) {$s_2$}; \node at (1.5,-1.5) {$s_3$}; \node at (2,-1.5) {$s_4$}; \node at (2.5,-1.5) {$s_5$};
    \draw[dotted, line width=0.5pt] (0.5,-1.2) -- (0.5,2.5);
    \draw[dotted, line width=0.5pt] (1,-1.2)   -- (1,2.5);
    \draw[dotted, line width=0.5pt] (1.5,-1.2) -- (1.5,2.5);
    \draw[dotted, line width=0.5pt] (2,-1.2)   -- (2,2.5);
    \draw[dotted, line width=0.5pt] (2.5,-1.2) -- (2.5,2.5);
    \sq{2}{0};   \node at (2.5,-0.5) {$5$};
    \sq{1.5}{1}; \node at (2,0.5)    {$4$};
    \sq{0.5}{1};   \node at(1,0.5)   {$2$};
    \sq{0}{2};   \node at(0.5,1.5)   {$1$};
\end{tikzpicture} & .
\end{tabular} \end{center}

    \noindent Continuing to place blocks in the same manner, a heap representation that corresponds to $\w$ is
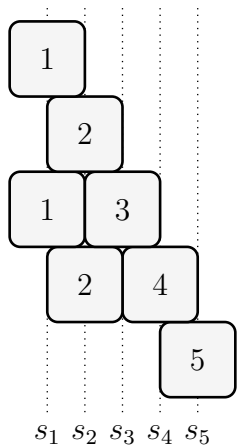
\begin{center} \begin{tabular}{m{2.9cm} m{0.5cm}}  \begin{tikzpicture}
    \node at (0.5,-1.5) {$s_1$}; \node at (1,-1.5) {$s_2$}; \node at (1.5,-1.5) {$s_3$}; \node at (2,-1.5) {$s_4$}; \node at (2.5,-1.5) {$s_5$};
    \draw[dotted, line width=0.5pt] (0.5,-1.2) -- (0.5,4.2);
    \draw[dotted, line width=0.5pt] (1,-1.2)   -- (1,4.2);
    \draw[dotted, line width=0.5pt] (1.5,-1.2) -- (1.5,4.2);
    \draw[dotted, line width=0.5pt] (2,-1.2)   -- (2,4.2);
    \draw[dotted, line width=0.5pt] (2.5,-1.2) -- (2.5,4.2);
    \sq{2}{0};   \node at (2.5,-0.5) {$5$};
    \sq{1.5}{1}; \node at (2,0.5)    {$4$};
    \sq{0.5}{1}; \node at (1,0.5)    {$2$};
    \sq{0}{2};   \node at (0.5,1.5)  {$1$};
    \sq{1}{2};   \node at (1.5,1.5)  {$3$};
    \sq{0.5}{3}; \node at (1,2.5)    {$2$};
    \sq{0}{4};   \node at (0.5,3.5)  {$1$};
\end{tikzpicture} & .
\end{tabular} \end{center}
\end{example}
    
    Conversely, given a heap, we can write an expression for the group element.
    By starting on the top and moving left to right and down, we write the corresponding generators. We get an expression that is commutation equivalent to any expression to which the heap corresponds.
    
\begin{example} Given the heap in Figure~\ref{fig:heapToRedExp}, we obtain the reduced expression $s_2 s_3 s_5 s_4$, which is commutation equivalent to $s_2s_5s_3s_4$ and $s_5s_2s_3s_4$, all of which yield the same heap. In fact, all reduced expressions yield the same heap since this particular element is FC.

\begin{center} \begin{figure}[H] \centering
\begin{tikzpicture}
    \sq{2}{2};   \node at (2.5,1.5)  {$5$};
    \sq{1.5}{1}; \node at (2,0.5)    {$4$};
    \sq{0.5}{3}; \node at (1,2.5)    {$2$};
    \sq{1}{2};   \node at (1.5,1.5)  {$3$};
\end{tikzpicture}
\caption{The heap for an FC element.}\label{fig:heapToRedExp}
\end{figure}
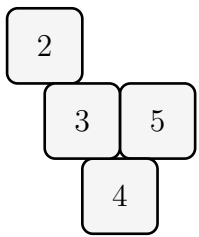 \end{center}
\end{example}
   
\begin{example} We return to Example~\ref{ex:heap}. Note that $\w = s_1s_2s_3s_1s_2s_4s_5$ is not fully commutative because there is opportunity to apply a braid relation.
    We have
    \begin{equation} \label{eq:nonuniqueheapex} 
        s_1s_2\textcolor{magenta}{s_3s_1}s_2s_4s_5 = \textcolor{blue}{s_1s_2s_1}
        s_3s_2s_4s_5 = s_2s_1s_2s_3s_2s_4s_5,
    \end{equation}
where the \textcolor{magenta}{pink} subword denotes applying a commutation to obtain the next expression and the \textcolor{blue}{blue} subword denotes applying a braid relation to obtain the next expression.
    Since $w$ is not FC, we can represent $w$ with a different heap using the last reduced expression $s_2s_1s_2s_3s_2s_4s_5$ in (\ref{eq:nonuniqueheapex}).
    We get the heap shown in Figure~\ref{fig:nonuniqueheapex} as another representation of $w$. Note that we can see the braid relation $s_2s_1s_2 = s_1s_2s_1$ in the heap, highlighted in \textcolor{blue}{blue}, that we applied in (\ref{eq:nonuniqueheapex}).
\begin{center} \begin{figure}[H] \centering
\begin{tikzpicture}
    \sq{2}{0};     \node at (2.5,-0.5) {$5$};
    \sq{1.5}{1};   \node at (2,0.5)    {$4$};
    \sq{0.5}{1};   \node at (1,0.5)    {$2$};
    \sqbl{0}{4};   \node at (0.5,3.5)  {$1$};
    \sq{1}{2};     \node at (1.5,1.5)  {$3$};
    \sqbl{0.5}{3}; \node at (1,2.5)    {$2$};
    \sqbl{0.5}{5}; \node at (1,4.5)    {$2$};
\end{tikzpicture}
\caption{A second heap for the element in Example~\ref{ex:heap}.} \label{fig:nonuniqueheapex}
\end{figure}
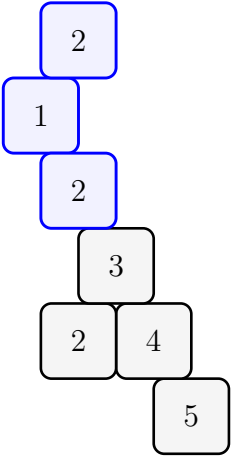 \end{center}
\end{example}

\begin{definition}\label{def:subheap} Let $\w = s_{x_{1}} \cdots s_{x_{n}}$ be a reduced expression for $w \in W$.  We define a heap $H'$ to be a \emph{subheap} of $H(\w)$ if $H' = H(\w')$, where $\w' = s_{y_1}s_{y_2} \cdots s_{y_k}$ is a subexpression of $\w$. We emphasize that the subexpression need not be a subword.
\end{definition} 

    We say that a subposet $Q$ of a poset $P$ is \emph{convex} if $y \in Q$ whenever $x < y < z$ in $P$ and $x, z \in Q$. We will refer to a subheap as a \emph{convex subheap} if the underlying subposet is convex.  

\begin{example} Let $w \in W(A_6)$ have reduced expression $\w = s_2 s_3 \textcolor{blue}{s_5 s_4} s_6 \textcolor{blue}{s_5}$. Since there is no opportunity to apply a braid relation in any reduced expression for $w$, $w$ is FC, and so there is a unique heap.
    The heap $H(w)$ is shown in Figure~\ref{fig:FCheapex}. Notice that we chose to place all the blocks in this heap as low as possible.

\begin{center} 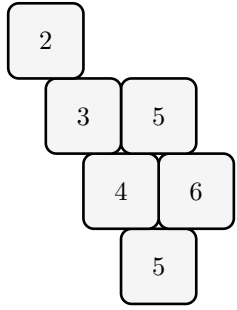
\begin{figure}[H] \centering
\begin{tikzpicture}
    \sq{0.5}{0}; \node at (1,-0.5)  {\footnotesize $5$};
    \sq{1}{1};   \node at (1.5,0.5) {\footnotesize $6$};
    \sq{0}{1};   \node at (0.5,0.5) {\footnotesize $4$};
    \sq{0.5}{2}; \node at (1,1.5)   {\footnotesize $5$};
    \sq{-0.5}{2};\node at (0,1.5)   {\footnotesize $3$};
    \sq{-1}{3};  \node at (-0.5,2.5){\footnotesize $2$};
\end{tikzpicture}
\caption{The heap of an FC element of $W(A_6)$.}\label{fig:FCheapex}
\end{figure} \end{center}
   
    Now, let $\w' = s_5 s_4 s_5$ be the subexpression of $\w$ that results from picking the third, fourth, and last generators of $\w$, highlighted in \textcolor{blue}{blue}. Then $H(\w')$ is shown in Figure~\ref{fig:nonconvexsubheapex} and is a subheap of $H(w)$, but $H(\w')$ is not convex since there is a block in $H(w)$ corresponding to the generator $s_6$ that occurs between the two occurrences of $s_5$ but does not have a block representing it in $H(\w')$.
    However, if we include the generator $s_6$, we get $\w'' = s_5 s_4 s_6 s_5$, and $H(\w'')$ is a convex subheap of $H(w)$, as shown in Figure~\ref{fig:convexsubheapex}.

\begin{center} \begin{figure}[H] \centering
\begin{subfigure}{0.4\textwidth} \centering
\begin{tikzpicture}
    \sq{0.5}{2}; \node at (1,1.5)   {\footnotesize $5$};
    \sq{0}{1};   \node at (0.5,0.5) {\footnotesize $4$};
    \sq{0.5}{0}; \node at (1,-0.5)  {\footnotesize $5$};
\end{tikzpicture}
\caption{}\label{fig:nonconvexsubheapex}
\end{subfigure}
\begin{subfigure}{0.4\textwidth} \centering
\begin{tikzpicture}
    \sq{0.5}{2}; \node at (1,1.5)   {\footnotesize $5$};
    \sq{0}{1};   \node at (0.5,0.5) {\footnotesize $4$};
    \sq{0.5}{0}; \node at (1,-0.5)  {\footnotesize $5$};
    \sq{1}{1};   \node at (1.5,0.5) {\footnotesize $6$};
\end{tikzpicture}
\caption{}\label{fig:convexsubheapex}
\end{subfigure}
\caption{Non-convex and convex subheaps of the heap in Figure~\ref{fig:FCheapex}.}
\end{figure}
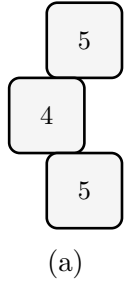
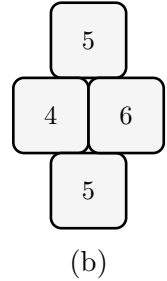
\end{center}
    
    Notice that if we remove the block labeled by 6 from the original heap, then the heap in Figure~\ref{fig:nonconvexsubheapex} is a convex subheap of $H(w)$.
\end{example}

    From this point on, if there will be no confusion, we will not specify the exact subexpression from which a subheap arises.
    The following proposition follows from the proof of Proposition 3.3 in~\cite{Stembridge1996}.

\begin{proposition}\label{prop:convexsubheapiff} Let $w \in \FC(W)$.  Then $H'$ is a convex subheap of $H(w)$ if and only if $H'$ is the heap for some subword of some reduced expression for $w$. \qed
\end{proposition}

    The following proposition follows from Lemma 2.4.5 in~\cite{Ernst2010}. It will help us identify when a heap corresponds to a fully commutative element in $W(A_n)$.

\begin{proposition} \label{prop:convexsubheap} Let $w \in \FC(A_n)$. Then $H(w)$ cannot contain either of the convex subheaps shown in Figure~\ref{fig:convexnotinFCheaps}, where $1 \leq i \leq n-1$ and \begin{tabular}{m{0.5cm}} \begin{tikzpicture}[scale=0.6] \bsq{0}{0}; \end{tikzpicture} \end{tabular} is used to emphasize the absence of a block in the corresponding position in $H(w)$. \qed
\end{proposition}
\begin{center} \begin{figure}[H] \centering
\begin{subfigure}{0.3\textwidth} \centering
\begin{tikzpicture}
    \sq{0.5}{2}; \node at (1,1.5) {\footnotesize $i+1$};
    \sq{0}{1}; \node at (0.5,0.5) {\footnotesize $i$};
    \sq{0.5}{0}; \node at (1,-0.5) {\footnotesize $i+1$};
    \bsq{1}{1};
\end{tikzpicture}\caption{}
\end{subfigure}
\begin{subfigure}{0.3\textwidth} \centering
\begin{tikzpicture}
    \sq{0}{2}; \node at (0.5,1.5) {\footnotesize $i$};
    \sq{0.5}{1}; \node at (1,0.5) {\footnotesize $i+1$};
    \sq{0}{0}; \node at (0.5,-0.5) {\footnotesize $i$};
    \bsq{-0.5}{1};
\end{tikzpicture}\caption{}
\end{subfigure}
\caption{The convex subheaps not contained in heaps of FC elements.}\label{fig:convexnotinFCheaps}
\end{figure}
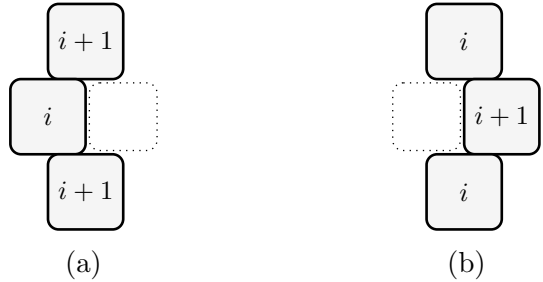 \end{center}

    The following lemma will become useful in proofs in later sections.

\begin{lemma}\label{lem:stst}
If $m(s_i,s_j) = 3$, then $s_is_js_is_j = s_js_i$.
\end{lemma}

\begin{proof} Consider the expression $s_is_js_is_j$. Applying a braid relation to the first three generators and simplifying, we get $s_is_js_is_j = s_js_is_js_j = s_js_i$.
\end{proof}

    Recall that we defined heaps for reduced expressions. However, it will be useful for us to extend the stacked blocks representation of a heap to non-reduced expressions, which we do in the obvious way. If $\w$ and $\w'$ are two expressions, not necessarily reduced, for $w \in W$, then we will write $H(\w) \equiv H(\w')$.

    Using this idea, we can write Lemma~\ref{lem:stst} in terms of heaps. If $m(s_i,s_j) = 3$, then we have the equivalent heaps shown in Figure~\ref{fig:lemststheaps}.
\begin{center} 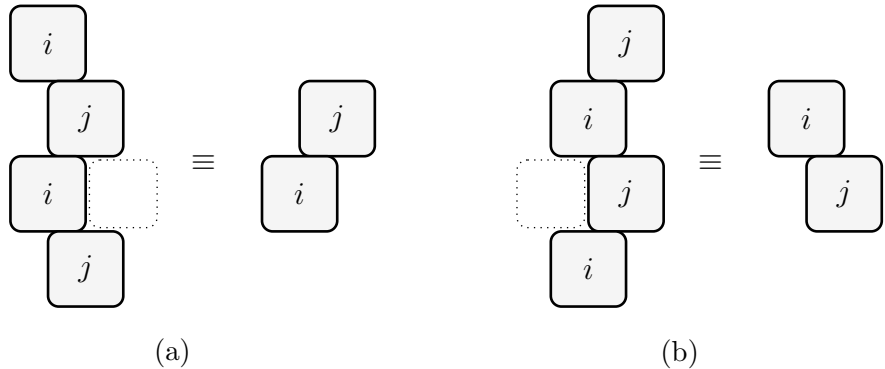
\begin{figure}[H] \centering
\begin{subfigure}{0.4\textwidth} \centering
\begin{tabular}{m{2cm}m{0.5cm}m{1cm}}
\begin{tikzpicture}
    \sq{0}{2};    \node at (0.5,1.5) {$i$};
    \sq{0.5}{1};  \node at (1,0.5)   {$j$};
    \sq{0}{0};    \node at (0.5,-0.5){$i$};
    \sq{0.5}{-1}; \node at (1,-1.5)  {$j$};
    \bsq{1}{0};
\end{tikzpicture} & $\equiv$ &
\begin{tikzpicture}
    \sq{0.5}{1};  \node at (1,0.5)    {$j$};
    \sq{0}{0};    \node at (0.5,-0.5) {$i$};
\end{tikzpicture}
\end{tabular} \caption{}
\end{subfigure}
\begin{subfigure}{0.4\textwidth} \centering
\begin{tabular}{m{2cm}m{0.5cm}m{1cm}}
\begin{tikzpicture}
    \sq{0.5}{2};  \node at (1,1.5)   {$j$};
    \sq{0}{1};    \node at (0.5,0.5) {$i$};
    \sq{0.5}{0};  \node at (1,-0.5)  {$j$};
    \sq{0}{-1};   \node at (0.5,-1.5){$i$};
    \bsq{-0.5}{0};
\end{tikzpicture} & $\equiv$ &
\begin{tikzpicture}
    \sq{0}{1};    \node at (0.5,0.5) {$i$};
    \sq{0.5}{0};  \node at (1,-0.5)  {$j$};
\end{tikzpicture}
\end{tabular} \caption{}
\end{subfigure}
\caption{Lemma~\ref{lem:stst} in terms of heaps.}\label{fig:lemststheaps}
\end{figure} \end{center}

    Let $w \in W(A_n)$ with expression $\w$. If $s_is_js_is_j$ is a subword of $\w$ with $m(s_i,s_j) = 3$, then we refer to $s_is_js_is_j$ as an \emph{extra long $s_is_j$-chain}.

\begin{remark} From now on, for brevity, we may write $i$ in place of the generator $s_i$. For example, we may write $123$ in place of $s_1s_2s_3$.
\end{remark}


\chapter{Cyclically fully commutative elements}

\section{Cyclically reduced elements}
    Recall that $s^{-1} = s$ for all $s \in S$, so $sws^{-1} = sws$.
    Given a word $\w = s_{x_1} s_{x_2} \cdots s_{x_k}$ for $w \in W$, a \emph{cyclic shift} of $\w$ is defined to be the natural expression that arises by conjugating $w$ by $s_{x_1}$.
    That is, $$s_{x_1} s_{x_2} \cdots s_{x_k} ~{\mapsto}~ s_{x_2} \cdots s_{x_k} s_{x_1}$$ since $s_{x_1}(s_{x_1} s_{x_2} \cdots s_{x_k})s_{x_1} = s_{x_2} \cdots s_{x_k} s_{x_1}$.

\begin{definition}\label{def:cycred}
    Let $(W,S)$ be a Coxeter system and let $\w$ be a reduced expression for some $w \in W$. If every cyclic shift of $\w$ is a reduced expression for some element in $W$, then we say that $\w$ is \emph{cyclically reduced}. A group element $w \in W$ is \emph{cyclically reduced} if every reduced expression for $w$ is cyclically reduced.
\end{definition}

    If $w$ is cyclically reduced, we can write every reduced expression for $w$ in a circle without creating any collapse in length.
    
\begin{example}\label{ex:kappa} We now consider a couple of examples.
\begin{enumerate}[label=(\alph*), leftmargin=0.75in]
\item Consider the Coxeter group of type $A_5$. Let $w \in W$ have reduced expression $\w = 31245$.
    The cyclic version of $w$ is shown in Figure~\ref{fig:circles}. In this case, $w$ is clearly cyclically reduced since there are no repeat generators. That is, we never have two adjacent occurrences of the same generator after commutations or braid moves.

\begin{center} 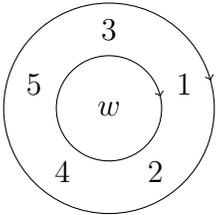
\begin{figure}[H] \centering
\begin{tikzpicture}[scale=0.7]
    \draw[decoration={markings, mark=at position 0.05 with {\arrow{<}}},postaction={decorate}]
        (0,0) circle (2cm);
    \draw[decoration={markings, mark=at position 0.05 with {\arrow{<}}},postaction={decorate}]
        (0,0) circle (1cm);
    \draw (0,0)     node {$w$};
    \draw (90:1.5)  node {3};
    \draw (18:1.5)  node {1};
    \draw (306:1.5) node {2};
    \draw (234:1.5) node {4};
    \draw (162:1.5) node {5};
\end{tikzpicture}
\caption{Cyclically reduced group element of $W(A_4)$ written in a circle.} \label{fig:circles}
\end{figure} \end{center}

\item Consider the Coxeter group of type $A_4$. Let $w \in W$ have reduced expression $\w = 342132$. Then $w$ is not cyclically reduced, as shown in Figure~\ref{fig:notcycred}.
\begin{center} \begin{figure}[H] \centering
\begin{tabular}{m{3cm} m{0.2cm} m{3cm} m{0.2cm} m{3cm} m{0.2cm} m{3cm}}
\begin{tikzpicture}[scale=0.7]
    \draw[decoration={markings, mark=at position 0.05 with {\arrow{<}}},postaction={decorate}]
        (0,0) circle (2cm);
    \draw[decoration={markings, mark=at position 0.05 with {\arrow{<}}},postaction={decorate}]
        (0,0) circle (1cm);
    \draw (0,0)        node {$w$};
    \draw (90:1.5)     node {\textcolor{blue}{3}};
    \draw (30:1.5)     node {4};
    \draw (330:1.5)    node {2};
    \draw (270.71:1.5) node {1};
    \draw (210.28:1.5) node {\textcolor{blue}{3}};
    \draw (150.85:1.5) node {\textcolor{blue}{2}};
\end{tikzpicture} & = &
\begin{tikzpicture}[scale=0.7]
    \draw[decoration={markings, mark=at position 0.05 with {\arrow{<}}},postaction={decorate}]
        (0,0) circle (2cm);
    \draw[decoration={markings, mark=at position 0.05 with {\arrow{<}}},postaction={decorate}]
        (0,0) circle (1cm);
    \draw (0,0)        node {$w$};
    \draw (90:1.5)     node {2};
    \draw (30:1.5)     node {\textcolor{magenta}{4}};
    \draw (330:1.5)    node {\textcolor{magenta}{2}};
    \draw (270:1.5)    node {1};
    \draw (210:1.5)    node {2};
    \draw (150:1.5)    node {3};
\end{tikzpicture} & = &
\begin{tikzpicture}[scale=0.7]
    \draw[decoration={markings, mark=at position 0.05 with {\arrow{<}}},postaction={decorate}]
        (0,0) circle (2cm);
    \draw[decoration={markings, mark=at position 0.05 with {\arrow{<}}},postaction={decorate}]
        (0,0) circle (1cm);
    \draw (0,0)        node {$w$};
    \draw (90:1.5)     node {\textcolor{turq}{2}};
    \draw (30:1.5)     node {\textcolor{turq}{2}};
    \draw (330:1.5)    node {4};
    \draw (270:1.5)    node {1};
    \draw (210:1.5)    node {2};
    \draw (150:1.5)    node {3};
\end{tikzpicture} & = &
\begin{tikzpicture}[scale=0.7]
    \draw[decoration={markings, mark=at position 0.05 with {\arrow{<}}},postaction={decorate}]
        (0,0) circle (2cm);
    \draw[decoration={markings, mark=at position 0.05 with {\arrow{<}}},postaction={decorate}]
        (0,0) circle (1cm);
    \draw (0,0)        node {$w$};
    \draw (90:1.5)     node {3};
    \draw (0:1.5)      node {4};
    \draw (270.71:1.5) node {1};
    \draw (180.28:1.5) node {2};
\end{tikzpicture}
\end{tabular}
\caption{Not cyclically reduced group element of $W(A_4)$ written in a circle}\label{fig:notcycred}
\end{figure}
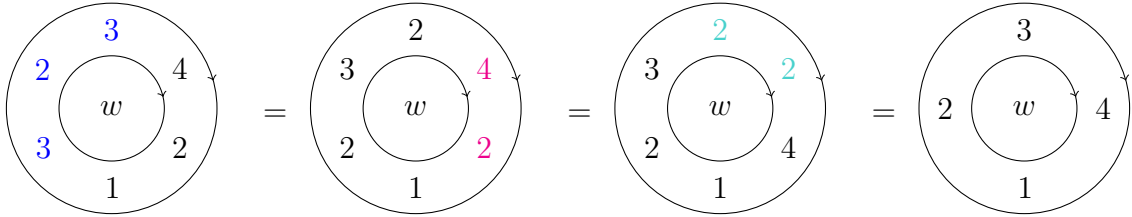 \end{center}
\end{enumerate}
\end{example}

    Now it is natural to ask the question: Do two cyclically reduced expressions for conjugate group elements differ by a sequence of commutations, braid moves, and cyclic shifts?
    
    Unfortunately the answer is no in general, but it is often true. One of the goals of the authors of~\cite{Boothby2012} is to understand when the answer is yes.
    This motivates the following definition.

\begin{definition}\label{def:CVMT} Let $W$ be a Coxeter group. We say that a conjugacy class $C$ satisfies the \emph{cyclic version of Matsumoto's Theorem}, or CVMT, if any two cyclically reduced expressions of elements in $C$ differ by commutations, braid moves, and cyclic shifts.
\end{definition}

    We can easily find an example where the CVMT fails.

\begin{example} Let $W$ be the Coxeter group of type $A_2$. Then $1$ and $2$ (or any two distinct generators in $W(A_n)$) are conjugate since $(12)1(21) = 12121 = 21221 = 211 = 2$, but $1$ and $2$ clearly do not differ by a sequence of commutations, braid moves, and cyclic shifts.
\end{example}

    It is well known that if $s_i \in S$, then $\ell(s_i w) = \ell(w) \pm 1$, and so $\ell(w^k) \leq k \cdot \ell(w)$. If equality holds for all $k \in \N$, we say that $w$ is \emph{logarithmic}.
    If every connected component of $\Gamma_{\supp(w)}$ (that is, the subgraph of $\Gamma$ induced by the generators which appear in $w$) describes an infinite Coxeter group, then we say that $w$ is \emph{torsion-free}.

\begin{proposition}[Boothby, et al.,~\cite{Boothby2012}] \label{prop:logarithmic} Let $W$ be a Coxeter group. If $w\in W$ is logarithmic, then $w$ is cyclically reduced and torsion-free. \qed
\end{proposition}

    It follows from a result in~\cite{Speyer2009} together with the fact that Coxeter elements are trivially cyclically reduced that the converse of Proposition~\ref{prop:logarithmic} holds for Coxeter elements.

\begin{theorem}[Speyer,~\cite{Speyer2009}]\label{thm:speyer} In any Coxeter group, a Coxeter element is logarithmic if and only if it is torsion-free. \qed
\end{theorem}

    The proof of Theorem~\ref{thm:speyer} is combinatorial and relies on a natural bijection between the set $\C(W)$ of Coxeter elements and the set $\Acyc(\Gamma)$ of acyclic orientations of the Coxeter graph.
    Specifically, if $c \in \C(W)$, let $(\Gamma,c)$ denote the digraph where, if $m(s_i,s_j) \geq 3$, the edge $\{s_i,s_j\}$ in the Coxeter graph $\Gamma$ is oriented as $(s_i,s_j)$ if $s_i$ appears before $s_j$ in $c$.
    The vertex $s_{x_i}$ is a source (respectively, sink) of $(\Gamma,c)$ if and only if $s_{x_i}$ is initial (respectively, terminal) in some reduced expression for $c$. 
    Conjugating a Coxeter element $c = s_{x_1} \cdots s_{x_n}$ by $s_{x_1}$ cyclically shifts the word $s_{x_1} \cdots s_{x_n}$ to $s_{x_2} \cdots s_{x_n}s_{x_1}$ since
\begin{equation} s_{x_1}(s_{x_1}s_{x_2}\cdots s_{x_n})s_{x_1} = s_{x_2} \cdots s_{x_n}s_{x_1},\end{equation}
    and, on the level of acyclic orientations, this corresponds to converting the source vertex $s_{x_1}$ of $(\Gamma,c)$ into a sink, which takes the orientation $(\Gamma,c)$ to $(\Gamma,s_{x_1}cs_{x_1})$.
    This generates an equivalence relation $\sim_\kappa$ on $\Acyc(\Gamma)$ and on $\C(W)$.
    Two acyclic orientations $(\Gamma,c)$ and $(\Gamma,c')$ are \emph{$\kappa$-equivalent} if and only if there is a sequence $x_1,\dots,x_k$ such that $c' = s_{x_k}\cdots s_{x_1} c s_{x_1}\cdots s_{x_k}$ and $s_{x_{i+1}}$ is a source vertex of $(\Gamma,s_{x_i}\cdots s_{x_1}cs_{x_1}\cdots s_{x_i})$ for each $i=1,\dots,k-1$.
    
    Thus, two Coxeter elements $c,c' \in \C(W)$ are \emph{$\kappa$-equivalent} if they differ by a sequence of length-preserving conjugations. That is, $c \sim_\kappa c'$ if they are conjugate by $s_{x_1}\cdots s_{x_k}$ such that
    $$\ell(c) = \ell(s_{x_i} \cdots s_{x_1} c s_{x_1} \cdots s_{x_i})$$ holds for each $i = 1,\ldots, k$.

    Performing a cyclic shift of a reduced expression of an arbitrary element $w \in W$ yields an element that is conjugate to $w$, but an element conjugate to $w$ is not necessarily a cyclic shift of $w$.
    The following result by H.~Eriksson and K.~Eriksson shows that conjugation and cyclic shifts are the same for Coxeter elements.

\begin{theorem}[Eriksson--Eriksson,~\cite{Eriksson2009}] \label{thm:e2} Let $W$ be a Coxeter group and let $c, c' \in \C(W)$. Then $c$ and $c'$ are conjugate if and only if $c$ and $c'$ are $\kappa$-equivalent. \qed
\end{theorem}

\section{Cyclically fully commutative elements}\label{sec:CFC}
    Note that the Erikssons' result is the CVMT applied to Coxeter elements. Despite the fact that the CVMT does not hold in general, we wish to gain understanding about when it does.

    The proof of Theorem~\ref{thm:e2} depends on torsion-free Coxeter elements being logarithmic, and the proof of this involves combinatorial properties of the acyclic orientation construction and source-to-sink equivalence relation.
    Thus, we are motivated to extend these properties to a larger class of elements. In fact, the acyclic orientation construction above generalizes to the FC elements.
    If $w \in \FC(W)$, then $(\Gamma,w)$ is the graph whose vertices are the disjoint union of generators in any reduced expression of $w$, and a directed edge is present for each pair of noncommuting generators, with the orientation denoting which comes first in $w$.
    Since $w \in \FC(W)$, i.e., $w$ has no opportunity for braid moves, the graph $(\Gamma,w)$ is well-defined. Though the acyclic orientation construction extends from $\C(W)$ to $\FC(W)$, the source-to-sink operation does not because a cyclic shift of a reduced expression for an FC element need not be FC.
    
\begin{example} Let $w \in W(A_4)$ have reduced expression $\w = 213243$. Then $w$ is FC because there is no opportunity to apply a braid move in any reduced expression for $w$, but a cyclic shift of $\w$ is commutation equivalent to a word containing a \textcolor{blue}{blue} $\gen{23}_3$ subword since
    $$213243 \overset{2}{\mapsto} 13\textcolor{magenta}{24}32 = 134\textcolor{blue}{232}$$ after applying a commutation to the \textcolor{magenta}{pink} subword, where $\overset{i}{\longmapsto}$ indicates a cyclic shift by $i$.
\end{example}

    The previous example motivates the following definition.

\begin{definition}\label{def:CFC} An element $w \in W$ is \emph{cyclically fully commutative}, or CFC, if every cyclic shift of every reduced expression for $w$ is a reduced expression for an FC element.
\end{definition}

    We denote the set of CFC elements of $W$ by $\CFC(\Gamma)$, where $\Gamma$ is the Coxeter graph corresponding to $W$, or $\CFC(W)$.
    CFC elements are exactly the elements whose reduced expressions, when written in a circle, avoid $\gen{s,t}_{m(s,t)}$ subwords for $m(s,t) \geq 3$, and hence they are the elements for which the source-to-sink operation extends in a well-defined manner.
    The remainder of this thesis considers CFC elements.
    
\begin{example} Let $W$ be the Coxeter group of type $A_4$ and let $w, y \in W$ have reduced expressions $\w = 1243$ and $\y = 21324$, respectively. Then both $\w$ and $\y$ are FC, but, when we write each reduced expression in a circle, we have the diagrams shown in Figure~\ref{fig:circlesexample}, so $w$ is CFC because there are no opportunities for braid moves or collapses created in the circle, but $y$ is not CFC since the two adjacent occurrences of 2 collapse after commuting 2 and 4.

\begin{center} 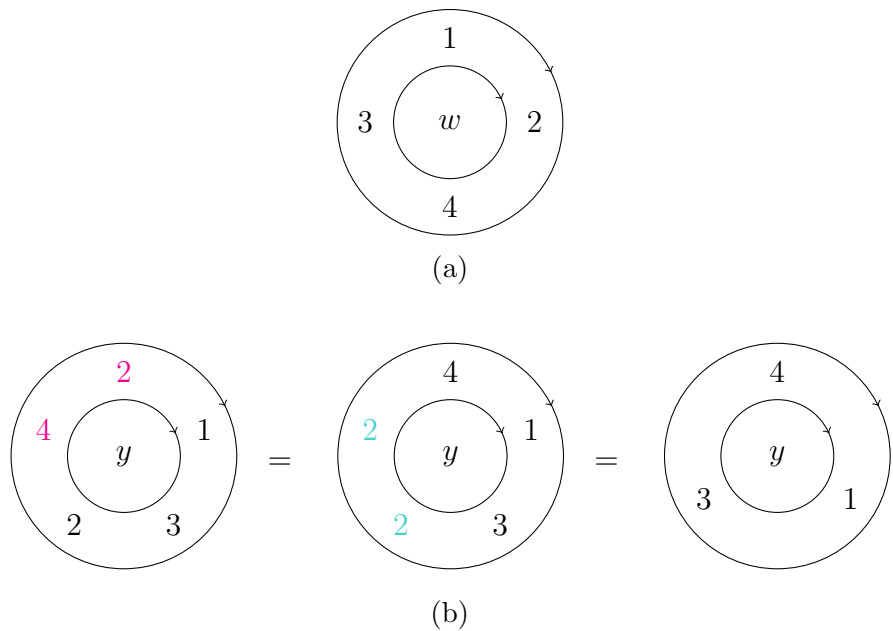
\begin{figure}[H] \centering
\begin{subfigure}{0.4\textwidth} \centering
\begin{tikzpicture}[scale=0.75]
    \draw[decoration={markings, mark=at position 0.08 with {\arrow{<}}},postaction={decorate}]
        (0,0) circle (2cm);
    \draw[decoration={markings, mark=at position 0.08 with {\arrow{<}}},postaction={decorate}]
        (0,0) circle (1cm);
    \draw (0,0)    node {$w$};
    \draw (0,1.5)  node {1};
    \draw (1.5,0)  node {2};
    \draw (0,-1.5) node {4};
    \draw (-1.5,0) node {3};
\end{tikzpicture}
\caption{}
\end{subfigure}
\begin{subfigure}{\textwidth} \centering \vspace{20pt}
\begin{tabular}{m{3cm} m{0.5cm} m{3cm} m{0.5cm} m{3cm}}
\begin{tikzpicture}[scale=0.75]
    \draw[decoration={markings, mark=at position 0.08 with {\arrow{<}}},postaction={decorate}]
        (0,0) circle (2cm);
    \draw[decoration={markings, mark=at position 0.08 with {\arrow{<}}},postaction={decorate}]
        (0,0) circle (1cm);
    \draw (0,0)     node {$y$};
    \draw (90:1.5)  node {\textcolor{magenta}{2}};
    \draw (18:1.5)  node {1};
    \draw (234:1.5) node {2};
    \draw (306:1.5) node {3};
    \draw (162:1.5) node {\textcolor{magenta}{4}};
\end{tikzpicture} & $=$ &
\begin{tikzpicture}[scale=0.75]
    \draw[decoration={markings, mark=at position 0.08 with {\arrow{<}}},postaction={decorate}]
        (0,0) circle (2cm);
    \draw[decoration={markings, mark=at position 0.08 with {\arrow{<}}},postaction={decorate}]
        (0,0) circle (1cm);
    \draw (0,0)     node {$y$};
    \draw (90:1.5)  node {4};
    \draw (18:1.5)  node {1};
    \draw (234:1.5) node {\textcolor{turq}{2}};
    \draw (306:1.5) node {3};
    \draw (162:1.5) node {\textcolor{turq}{2}};
\end{tikzpicture} & $=$ &
\begin{tikzpicture}[scale=0.75]
    \draw[decoration={markings, mark=at position 0.08 with {\arrow{<}}},postaction={decorate}]
        (0,0) circle (2cm);
    \draw[decoration={markings, mark=at position 0.08 with {\arrow{<}}},postaction={decorate}]
        (0,0) circle (1cm);
    \draw (0,0)     node {$y$};
    \draw (90:1.5)  node {4};
    \draw (330:1.5) node {1};
    \draw (210:1.5) node {3};
\end{tikzpicture} \end{tabular}
\caption{}
\end{subfigure}
\caption{Two FC elements of $W(A_4)$ written in a circle.}\label{fig:circlesexample}
\end{figure} \end{center}
\end{example}
    
\begin{remark}\label{rem:CoxCFC}
Coxeter elements are CFC since Coxeter elements are FC and any cyclic shift of a Coxeter element is still a Coxeter element.
\end{remark}
    
\begin{proposition}\label{prop:CFCsubexps} Elements corresponding to subexpressions of Coxeter elements are CFC.
\end{proposition}
\begin{proof} Let $(W,S)$ be a Coxeter system and let $w \in \C(W)$ with reduced expression $\w$. Then $w$ is FC. Let $\w'$ be a subexpression of $\w$. Generators from $S$ appear at most once in $\w'$, so $\w'$ is reduced and FC, as well.
    Hence every cyclic shift of $\w'$ has at most one appearance of each generator, so no cyclic shift of $\w'$ will have $\gen{st}_{m(s,t)}$ as a subword for all $s,t \in \supp(\w')$. Thus, the group element corresponding to $\w'$ is CFC.
\end{proof}

    The following classification of CFC elements in Coxeter groups of type $A_n$ is Proposition 5.4 in~\cite{Boothby2012}.
    
\begin{proposition}[Boothby, et al.,~\cite{Boothby2012}] \label{prop:CFCiffatmostonce} Let $w \in W(A_n)$. Then $w$ is CFC if and only if each generator in $\supp(w)$ appears exactly once. \qed
\end{proposition}

    In other words, the CFC elements in $W(A_n)$ are precisely the elements that correspond to reduced subexpressions of the Coxeter elements.

\begin{example} Let $W$ be the Coxeter group of type $A_3$. The set of CFC elements of $W$ is
	$$\CFC(A_3) = \{e, 1, 2, 3, 13, 12, 21, 23, 32, 123, 321, 132, 231\}.$$
\end{example}

\section{Cylindrical heaps}
    Let $w \in W(A_n)$ have reduced expression $\w$ and suppose $\w$ is commutation equivalent to a reduced expression that begins with $s_i$. Then a block labeled by $i$ occurs at the top of the heap $H(\w)$.
    A \emph{cyclic shift of $H(\w)$ with respect to $i$} is the heap that results from removing the block labeled by $i$ from the top of the heap and appending it to the bottom.
    In other words, if $\w$ is commutation equivalent to $s_i\u$, then a cyclic shift of $H(\w)$ with respect to $i$ is the heap $H(\u s_i)$.
    Note that $H(\u s_i)$ may not be the heap for a reduced expression. 
    However, since $\CFC(A_n) \subseteq \FC(A_n)$, any $w \in \CFC(A_n)$ has a unique heap and cyclic shifts of reduced expressions of CFC elements are reduced, so if $w$ is CFC, then $H(\u s_i)$ is the unique heap obtained by performing a cyclic shift on $w$.
    
    Consider the equivalence relation $\approx_\kappa$ generated by cyclic shifts of heaps.
    It is clear that $w$ is CFC if and only if all heaps in the equivalence class for $H(w)$ are heaps for reduced expressions of FC elements.
    Let $w, w' \in \CFC(A_n)$. Then $H(w)$ and $H(w')$ are \emph{cyclically equivalent} if $H(w)$ and $H(w')$ differ by a sequence of cyclic shifts.
    We emphasize that cyclically equivalent is only defined for heaps corresponding to CFC elements.

\begin{example} \label{ex:CFC} Let $W$ be the Coxeter group of type $A_7$.
\begin{enumerate}[leftmargin=0.75in, label=(\alph*)]
\item The group element corresponding to the heap in Figure~\ref{fig:cylheapsex1} is CFC since every sequence of cyclic shifts of the heap corresponds to a reduced expression for an FC element.
\begin{center} \begin{figure}[H] \centering \begin{tikzpicture}[scale=0.85]
    \sq{0}{2};   \node at (0.5,1.5)  {\footnotesize $4$};
    \sq{0.5}{1}; \node at (1,0.5)    {\footnotesize $5$};
    \sq{1}{0};   \node at (1.5,-0.5) {\footnotesize $6$};
    \sq{1.5}{-1};\node at (2,-1.5)   {\footnotesize $7$};
\end{tikzpicture}
\caption{The heap of a CFC element in $W(A_7)$.}\label{fig:cylheapsex1} \end{figure} \end{center}

\item The group element corresponding to the heap in Figure~\ref{fig:cylheapsex2.1} is FC because there is no opportunity to apply a braid move, but it is not CFC since the blocks labeled 2 collapse after a cyclic shift, where $\overset{2}{\mapsto}$ denotes a cyclic shift with respect to 2.

\begin{center} 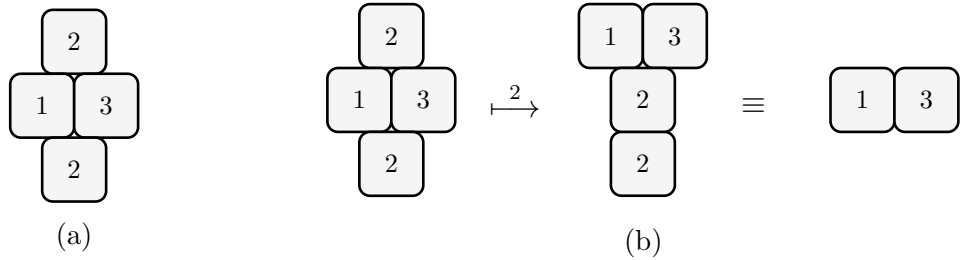
\begin{figure}[H] \centering
\begin{subfigure}{0.35\textwidth} \centering
\begin{tikzpicture}[scale=0.85]
    \sq{0.5}{2}; \node at (1,1.5)   {\footnotesize $2$};
    \sq{0}{1};   \node at (0.5,0.5) {\footnotesize $1$};
    \sq{1}{1};   \node at (1.5,0.5) {\footnotesize $3$};
    \sq{0.5}{0}; \node at (1,-0.5)  {\footnotesize $2$};
\end{tikzpicture} \caption{}\label{fig:cylheapsex2.1}
\end{subfigure}
\begin{subfigure}{0.55\textwidth} \centering
\begin{tabular}{m{1.75cm} m{0.75cm} m{1.75cm} m{0.75cm} m{1.75cm}}
\begin{tikzpicture}[scale=0.85]
    \sq{0.5}{2}; \node at (1,1.5)   {\footnotesize $2$};
    \sq{0}{1};   \node at (0.5,0.5) {\footnotesize $1$};
    \sq{1}{1};   \node at (1.5,0.5) {\footnotesize $3$};
    \sq{0.5}{0}; \node at (1,-0.5)  {\footnotesize $2$};
\end{tikzpicture} & $\overset{2}{\longmapsto}$ &
\begin{tikzpicture}[scale=0.85]
    \sq{0.5}{-1}; \node at (1,-1.5)  {\footnotesize $2$};
    \sq{0}{1};    \node at (0.5,0.5) {\footnotesize $1$};
    \sq{1}{1};    \node at (1.5,0.5) {\footnotesize $3$};
    \sq{0.5}{0};  \node at (1,-0.5)  {\footnotesize $2$};
\end{tikzpicture} & $\equiv$ &
\begin{tikzpicture}[scale=0.85]
    \sq{0}{1}; \node at (0.5,0.5) {\footnotesize $1$};
    \sq{1}{1}; \node at (1.5,0.5) {\footnotesize $3$};
\end{tikzpicture}
\end{tabular} \caption{}\label{fig:cylheapsex2.2}
\end{subfigure}
\caption{The heap of an FC (but not CFC) element in $W(A_7)$.}\label{fig:cylheapsex2}
\end{figure} \end{center}

\item The group element corresponding to the heap in Figure~\ref{fig:cylheapsex3} is not CFC since it is not even FC by Proposition~\ref{prop:convexsubheap}; $w$ has a reduced expression with $\gen{23}_3$ as a subword, highlighted in \textcolor{blue}{blue} in the heap.
\begin{center} 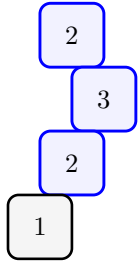
\begin{figure}[H] \centering \begin{tikzpicture}[scale=0.85]
    \sqbl{0.5}{3};  \node at (1,2.5)   {\footnotesize $2$};
    \sqbl{1}{2};    \node at (1.5,1.5) {\footnotesize $3$};
    \sqbl{0.5}{1};  \node at (1,0.5)   {\footnotesize $2$};
    \sq{0}{0};      \node at (0.5,-0.5){\footnotesize $1$};
\end{tikzpicture} \caption{The heap of a non-FC element in $W(A_7)$.}\label{fig:cylheapsex3}
\end{figure} \end{center}
\end{enumerate}
\end{example}
    
    We let $\hat{H}(w)$ represent the equivalence class of CFC heaps cyclically equivalent to $H(w)$, which we visualize by wrapping representatives on a cylinder.
    We call $\hat{H}(w)$ a \emph{cylindrical heap}. Note that our notion of a cylindrical heap coincides with the definition of a cylindric transformation of a heap given in~\cite{Petreolle2014}.

\begin{example}
Let $w \in \CFC(A_4)$ have reduced expression $1324$. Then $\hat{H}(w)$ can be represented by the cylindrical heap shown in Figure~\ref{fig:cylheap1324}, where we identify the edges of the north and south faces so that the arrows match direction. The elements of $\hat{H}(w)$ are shown in Figure~\ref{fig:cylheapelements}.
\end{example}

\begin{center}
\begin{figure}[H]
\centering
\begin{tikzpicture}
\draw[line width=1.5pt,->] (-0.5,1)--(2.75,1); \draw[line width=1.5pt,->] (-0.5,-1)--(2.75,-1);
    \sq{0}{1};   \node at (0.5,0.5) {\footnotesize $1$};
    \sq{1}{1};   \node at (1.5,0.5) {\footnotesize $3$};
    \sq{0.5}{0}; \node at (1,-0.5)  {\footnotesize $2$};
    \sq{1.5}{0}; \node at (2,-0.5)  {\footnotesize $4$};
\end{tikzpicture}
\caption{The cylindrical heap for a CFC element in $W(A_4)$.}\label{fig:cylheap1324}
\end{figure}
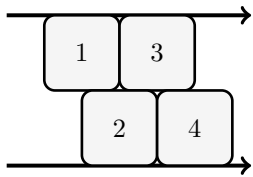
\end{center}

\begin{center} \begin{figure}[htpb] \centering
\begin{tabular}{m{2cm} m{0.4cm} m{2cm} m{0.4cm} m{2cm} m{0.4cm} m{2cm} m{0.4cm}}
\begin{tikzpicture}[scale=0.85]
    \sq{0}{1};   \node at (0.5,0.5) {\footnotesize $1$};
    \sq{1}{1};   \node at (1.5,0.5) {\footnotesize $3$};
    \sq{0.5}{0}; \node at (1,-0.5)  {\footnotesize $2$};
    \sq{1.5}{0}; \node at (2,-0.5)  {\footnotesize $4$};
\end{tikzpicture} & , &
\begin{tikzpicture}[scale=0.85]
    \sq{0}{1};   \node at (0.5,0.5) {\footnotesize $1$};
    \sq{1}{-1};  \node at (1.5,-1.5){\footnotesize $3$};
    \sq{0.5}{0}; \node at (1,-0.5)  {\footnotesize $2$};
    \sq{1.5}{0}; \node at (2,-0.5)  {\footnotesize $4$};
\end{tikzpicture} & , &
\begin{tikzpicture}[scale=0.85]
    \sq{0}{-1};  \node at (0.5,-1.5){\footnotesize $1$};
    \sq{1}{-1};  \node at (1.5,-1.5){\footnotesize $3$};
    \sq{0.5}{0}; \node at (1,-0.5)  {\footnotesize $2$};
    \sq{1.5}{0}; \node at (2,-0.5)  {\footnotesize $4$};
\end{tikzpicture} & , &
\begin{tikzpicture}[scale=0.85]
    \sq{0}{-1};  \node at (0.5,-1.5){\footnotesize $1$};
    \sq{1}{-1};  \node at (1.5,-1.5){\footnotesize $3$};
    \sq{0.5}{0}; \node at (1,-0.5)  {\footnotesize $2$};
    \sq{1.5}{-2};\node at (2,-2.5)  {\footnotesize $4$};
\end{tikzpicture} & , \\ &&&&&&& \\
\begin{tikzpicture}[scale=0.85]
    \sq{0}{2};   \node at (0.5,1.5) {\footnotesize $1$};
    \sq{1}{0};   \node at (1.5,-0.5){\footnotesize $3$};
    \sq{0.5}{1}; \node at (1,0.5)   {\footnotesize $2$};
    \sq{1.5}{-1};\node at (2,-1.5)  {\footnotesize $4$};
\end{tikzpicture} & , &
\begin{tikzpicture}[scale=0.85]
    \sq{1.5}{3}; \node at (2,2.5)   {\footnotesize $4$};
    \sq{0}{2};   \node at (0.5,1.5) {\footnotesize $1$};
    \sq{1}{2};   \node at (1.5,1.5) {\footnotesize $3$};
    \sq{0.5}{1}; \node at (1,0.5)   {\footnotesize $2$};
\end{tikzpicture} & , &
\begin{tikzpicture}[scale=0.85]
    \sq{1}{3};   \node at (1.5,2.5) {\footnotesize $3$};
    \sq{0.5}{2}; \node at (1,1.5)   {\footnotesize $2$};
    \sq{1.5}{2}; \node at (2,1.5)   {\footnotesize $4$};
    \sq{0}{1};   \node at (0.5,0.5) {\footnotesize $1$};
\end{tikzpicture} & , &
\begin{tikzpicture}[scale=0.85]
    \sq{0}{1};   \node at (0.5,0.5) {\footnotesize $1$};
    \sq{0.5}{2}; \node at (1,1.5)   {\footnotesize $2$};
    \sq{1}{3};   \node at (1.5,2.5) {\footnotesize $3$};
    \sq{1.5}{4}; \node at (2,3.5)   {\footnotesize $4$};
\end{tikzpicture} &
\end{tabular}
\caption{The elements of the equivalence class of CFC heaps cyclically equivalent the heap of some $w \in \CFC(A_4)$.}\label{fig:cylheapelements}
\end{figure}
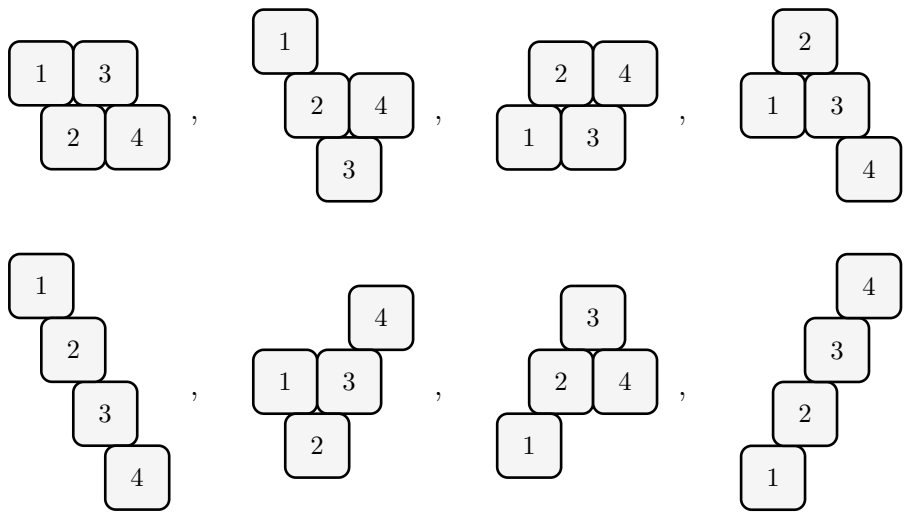
\end{center}

\begin{remark}\label{rem:CFCcylinderwrap}
Even though we lose the underlying poset structure when we wrap a heap on a cylinder, a convex subheap retains its natural meaning on the cylinder.
    An element $w \in W(A_n)$ is CFC if and only if the cylindrical heap of $w$ does not contain any representatives having any convex subheaps shown in Figure~\ref{fig:convexsubheapsnotinCFC}.
    
\begin{center} 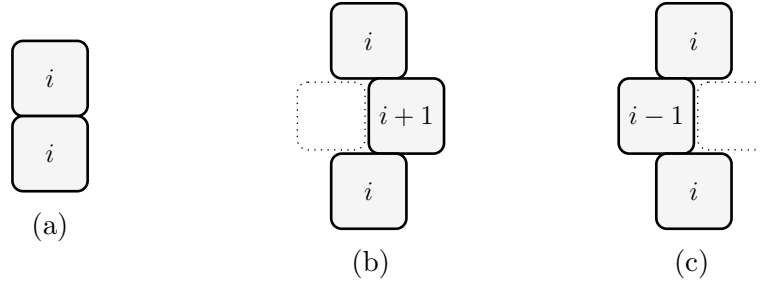
\begin{figure}[htpb] \centering
\begin{subfigure}{0.25\textwidth} \centering
\begin{tikzpicture}
    \sq{0}{2};    \node at (0.5,1.5) {\scalebox{0.85}{$i$}};
    \sq{0}{1};    \node at (0.5,0.5) {\scalebox{0.85}{$i$}};
\end{tikzpicture}
\caption{}\label{}
\end{subfigure}
\begin{subfigure}{0.25\textwidth} \centering
\begin{tikzpicture}
    \sq{0}{3};    \node at (0.5,2.5) {\scalebox{0.85}{$i$}};
    \sq{0.5}{2};  \node at (1,1.5)   {\scalebox{0.85}{$i+1$}};
    \sq{0}{1};    \node at (0.5,0.5) {\scalebox{0.85}{$i$}};
    \bsq{-0.5}{2};
\end{tikzpicture}
\caption{}\label{}
\end{subfigure}
\begin{subfigure}{0.25\textwidth} \centering
\begin{tikzpicture}
    \sq{0}{3};    \node at (0.5,2.5) {\scalebox{0.85}{$i$}};
    \sq{-0.5}{2}; \node at (0,1.5)   {\scalebox{0.85}{$i-1$}};
    \sq{0}{1};    \node at (0.5,0.5) {\scalebox{0.85}{$i$}};
    \bsq{0.5}{2};
\end{tikzpicture}
\caption{}\label{}
\end{subfigure}
\caption{The convex subheaps not allowed in the heaps for CFC elements.}\label{fig:convexsubheapsnotinCFC}
\end{figure} \end{center}
\end{remark}

\begin{example}\label{ex:A4boxes} Consider the Coxeter group of type $A_4$. Recall that generators appear at most once in CFC elements, by Proposition~\ref{prop:CFCiffatmostonce}.
    The collection of boxes in Figure~\ref{fig:A4boxes} contain all reduced expressions for CFC elements in $W(A_4)$.
    The reduced expressions are grouped into boxes that contain reduced expressions for CFC elements that differ by commutations and cyclic shifts. We clearly have no opportunity for braid moves because we only consider CFC elements.
    If two reduced expressions are listed in the same column in a box, then they are reduced expressions for the same CFC element. Alternatively, a column in a particular box corresponds to a commutation class of reduced expressions.
    
    The boxes are colored based on conjugacy. That is, if two reduced expressions are in boxes of the same color (or the same box), then the corresponding CFC elements are conjugate.
    If two reduced expressions are in the same box, then the corresponding CFC elements differ by a sequence of cyclic shifts, i.e., the reduced expressions look the same, up to commutation, when written in a circle.
    
    Note that all of the CFC elements in the \textcolor{blue}{blue} box are conjugate by Theorem \ref{thm:e2} since they are Coxeter elements.
    Also, the conjugacy class for $123$ is the set of all CFC elements, i.e., columns, in the \textcolor{magenta}{pink} boxes. Its conjugacy class is partitioned into two subsets, or \emph{cyclic classes}, each of which corresponds to a cylindrical heap.
    That is, the heaps of all the reduced expressions in the box with $123$ are cyclically equivalent, and so the cylindrical heap for this box is shown in Figure~\ref{fig:cylheap123}.
    The heaps of all the reduced expressions in the box with $234$ yield the cylindrical heap shown in Figure~\ref{fig:cylheap234}.

\begin{center} 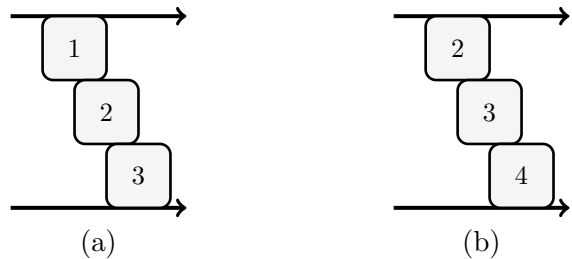
\begin{figure}[H] \centering
\begin{subfigure}{0.3\textwidth} \centering
\begin{tikzpicture}[scale=0.85]
\draw[line width=1.5pt,->] (-0.5,2)--(2.25,2); \draw[line width=1.5pt,->] (-0.5,-1)--(2.25,-1);
    \sq{0}{2};   \node at (0.5,1.5)  {\footnotesize $1$};
    \sq{0.5}{1}; \node at (1,0.5)    {\footnotesize $2$};
    \sq{1}{0};   \node at (1.5,-0.5) {\footnotesize $3$};
\end{tikzpicture}
\caption{}\label{fig:cylheap123}
\end{subfigure}
\begin{subfigure}{0.3\textwidth} \centering
\begin{tikzpicture}[scale=0.85]
\draw[line width=1.5pt,->] (-0.5,2)--(2.25,2); \draw[line width=1.5pt,->] (-0.5,-1)--(2.25,-1);
    \sq{0}{2};   \node at (0.5,1.5)  {\footnotesize $2$};
    \sq{0.5}{1}; \node at (1,0.5)    {\footnotesize $3$};
    \sq{1}{0};   \node at (1.5,-0.5) {\footnotesize $4$};
\end{tikzpicture}
\caption{}\label{fig:cylheap234}
\end{subfigure}
\caption{The cylindrical heaps corresponding to cyclic classes of CFC elements of $W(A_4)$.}\label{fig:cylheaps123and234}
\end{figure} \end{center}
    
    We are able to move between the cyclic classes since the elements are all conjugate. It is disappointing that we cannot say everything conjugate to $123$ is conjugate by cyclic shifts (a generalization of Theorem~\ref{thm:e2}). Thus, we need another way to move between the cyclic classes.
    It turns out that we can just ``slide" the cylindrical heaps to move from the cyclic class containing $123$ to the cyclic class containing $234$. We will discuss this in more detail in Section~\ref{sec:chunks}.

\begin{center} \begin{figure}[ht!]
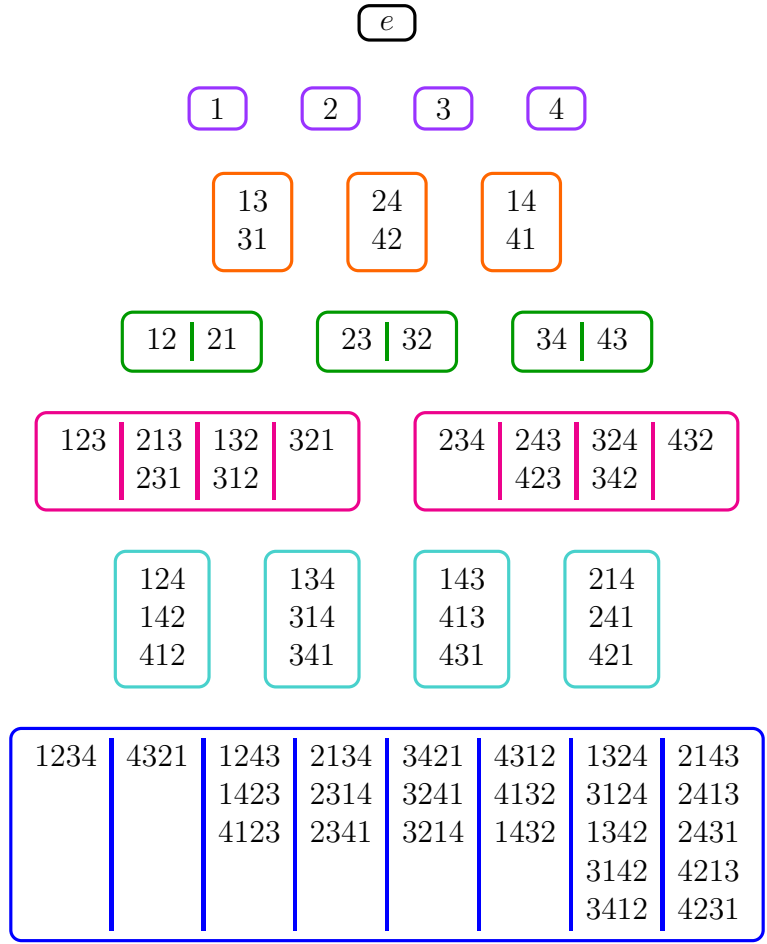
 \centering \begin{tabular}{c}
$$\boxed{~e~}$$ \\ \\
$\ds{\begin{array}{ccccccc}
\boxp{~1~} && \boxp{~2~} && \boxp{~3~} && \boxp{~4~}
\end{array}}$ \\ \\
$\ds{\begin{array}{ccccc}
\boxor{\begin{array}{c} 13 \\ 31 \end{array}} &&
\boxor{\begin{array}{c} 24 \\ 42 \end{array}} &&
\boxor{\begin{array}{c} 14 \\ 41 \end{array}}
\end{array}}$ \\ \\
${\ds 
\setlength{\arrayrulewidth}{1.25pt}
\begin{array}{ccccc}
\boxgr{
\begin{array}{c|c}\arrayrulecolor{ggreen}
12 & 21 \end{array}} 
&&
\boxgr{
\begin{array}{c|c}\arrayrulecolor{ggreen}
23 & 32 \end{array}}
&&
\boxgr{ 
\begin{array}{c|c}\arrayrulecolor{ggreen}
34 & 43 \end{array}}
\end{array}}$ \\ \\
${\ds
\setlength{\arrayrulewidth}{1.25pt}
\begin{array}{ccc}
\boxm{
\begin{array}{c|c|c|c}\arrayrulecolor{magenta}
123 & 213 & 132 & 321 \\
    & 231 & 312 &
\end{array}} 
& &
\boxm{
\begin{array}{c|c|c|c}\arrayrulecolor{magenta}
234 & 243 & 324 & 432 \\
    & 423 & 342 &
\end{array}}
\end{array}}$ \\ \\
${\ds
\setlength{\arrayrulewidth}{1.25pt}
\begin{array}{ccccccc}
\boxt{\begin{array}{c} 124 \\ 142 \\ 412 \end{array}} &&
\boxt{\begin{array}{c} 134 \\ 314 \\ 341 \end{array}} &&
\boxt{\begin{array}{c} 143 \\ 413 \\ 431 \end{array}} &&
\boxt{\begin{array}{c} 214 \\ 241 \\ 421 \end{array}}
\end{array}}$ \\ \\
$$\boxbl{
\setlength{\arrayrulewidth}{1.25pt}
\begin{array}{l|l|l|l|l|l|l|l}\arrayrulecolor{darkblue}
    1234 & 4321 & 1243 & 2134 & 3421 & 4312 & 1324 & 2143 \\
         &      & 1423 & 2314 & 3241 & 4132 & 3124 & 2413 \\
         &      & 4123 & 2341 & 3214 & 1432 & 1342 & 2431 \\
         &      &      &      &      &      & 3142 & 4213 \\
         &      &      &      &      &      & 3412 & 4231
\end{array}}$$
\end{tabular}
\caption{The conjugacy, cyclic, and commutation classes of CFC elements in $W(A_4)$.} \label{fig:A4boxes} 
\end{figure} \end{center}
\end{example}

    In~\cite{Stembridge1996}, Stembridge classified the Coxeter groups that contain finitely many FC elements (Theorem~\ref{thm:FCfinite}). Similarly, the \emph{CFC-finite groups} can be defined as the Coxeter groups that contain only finitely many CFC elements.

\begin{theorem}[Boothby, et al.,~\cite{Boothby2012}] \label{thm:CFCfinite}
    The irreducible CFC-finite Coxeter groups are $A_n$  with $n\geq 1$, $B_n$ with $n\geq 2$ , $D_n$ with $n\geq 4$, $E_n$ with $n\geq 6$, $F_n$ with $n\geq 4$, $H_n$ with $n\geq 3$, and $I_2(m)$ with $5 \leq m < \infty$. Thus, a Coxeter group is CFC-finite if and only if it is FC-finite. The graphs of FC- and CFC-finite Coxeter groups are shown in Figure~\ref{fig:coxgraphs}. \qed
\end{theorem}

\section{Pattern avoidance for CFC elements}\label{sec:Sn}
    In this section, $W$ refers to the Coxeter group of type $A_n$. Recall that $W$ is isomorphic to the symmetric group $S_{n+1}$ via the mapping that sends $s_i$ to the adjacent transposition $(i~i+1)$.
    Also recall that every permutation can be written uniquely (up to commutation) as a product of disjoint cycles.
    We will not make a distinction between an element from $W(A_n)$ and the corresponding permutation in $S_{n+1}$.

    As a convention, we will multiply (compose) permutations right to left.
    Recall that if $w \in S_n$, then $[w(1)~w(2) \cdots w(n)]$ is the \emph{one-line notation} corresponding to $w$. Note the use of brackets.

\begin{example} Let $W$ be the Coxeter graph of type $A_4$. Let $w \in W$ have reduced expression $\w = 12342$.
    Then the corresponding permutation in $S_5$ is $$(12)(23)(34)(45)(23) = (1245).$$
    Then, in one-line notation, we have $$(1245) = [24351]$$ since 1 is sent to 2, 2 is sent to 4, 3 is sent to itself, 4 is sent to 5, and 5 is sent back to 1.
\end{example}

    We can depict the one-line notation of a permutation $w$ as a graph to see its shape. A \emph{permutation line graph} has line segments joining $(i,w(i))$ to $(i+1,w(i+1))$ for each $1 \leq i \leq n-1$.

\begin{example}\label{ex:linegraphs} We consider some permutation line graphs.
\begin{enumerate}[label=(\alph*),leftmargin=0.75in]
\item Consider the permutation $w = [2413]$. Then the permutation line graph is shown in Figure~\ref{fig:permlinegraphs1}.

\item Consider the permutation $w = [315462]$. Then the permutation line graph is shown in Figure~\ref{fig:permlinegraphs2}.
\end{enumerate}

\begin{center} \begin{figure}[ht!] \centering
\begin{subfigure}{0.4\textwidth} \centering \vspace{35pt}
\begin{tikzpicture}[scale=0.95] \begin{scriptsize}
\draw (0,0)--(4.5,0); \draw (0,0)--(0,4.5);
\foreach \x in {1,2,3,4} \draw[shift={(\x,0)},color=black] (0pt,2pt)--(0pt,-2pt);
\foreach \y in {1,2,3,4} \draw[shift={(0,\y)},color=black] (2pt,0pt)--(-2pt,0pt);
    \draw (1,-0.25) node {$1$};
    \draw (2,-0.25) node {$2$};
    \draw (3,-0.25) node {$3$};
    \draw (4,-0.25) node {$4$};
    \draw (-0.25,1)  node {$1$};
    \draw (-0.25,2)  node {$2$};
    \draw (-0.25,3)  node {$3$};
    \draw (-0.25,4)  node {$4$};
    \draw[fill=black] (1,2) circle (1pt);
    \draw[fill=black] (2,4) circle (1pt);
    \draw[fill=black] (3,1) circle (1pt);
    \draw[fill=black] (4,3) circle (1pt);
    \draw (1,2)--(2,4)--(3,1)--(4,3);
\end{scriptsize} \end{tikzpicture}
\caption{}
\label{fig:permlinegraphs1}
\end{subfigure}
\begin{subfigure}{0.4\textwidth} \centering
\begin{tikzpicture}[scale=0.85] \begin{scriptsize}
\draw (0,0)--(6.5,0); \draw (0,0)--(0,6.5);
\foreach \x in {1,2,3,4,5,6} \draw[shift={(\x,0)},color=black] (0pt,2pt)--(0pt,-2pt);
\foreach \y in {1,2,3,4,5,6} \draw[shift={(0,\y)},color=black] (2pt,0pt)--(-2pt,0pt);
    \draw (1,-0.25) node {$1$};
    \draw (2,-0.25) node {$2$};
    \draw (3,-0.25) node {$3$};
    \draw (4,-0.25) node {$4$};
    \draw (5,-0.25) node {$5$};
    \draw (6,-0.25) node {$6$};
    \draw (-0.25,1) node {$1$};
    \draw (-0.25,2) node {$2$};
    \draw (-0.25,3) node {$3$};
    \draw (-0.25,4) node {$4$};
    \draw (-0.25,5) node {$5$};
    \draw (-0.25,6) node {$6$};
    \draw[fill=black]   (1,3) circle (1pt);
    \draw[fill=black]   (2,1) circle (1pt);
    \draw[fill=black]   (3,5) circle (1pt);
    \draw[fill=black]   (4,4) circle (1pt);
    \draw[fill=black]   (5,6) circle (1pt);
    \draw[fill=black]   (6,2) circle (1pt);
    \draw (1,3)--(2,1)--(3,5)--(4,4)--(5,6)--(6,2);
\end{scriptsize} \end{tikzpicture}
\caption{}
\label{fig:permlinegraphs2}
\end{subfigure}
\caption{Permutation line graphs.}\label{fig:permlinegraphsexample}
\end{figure}
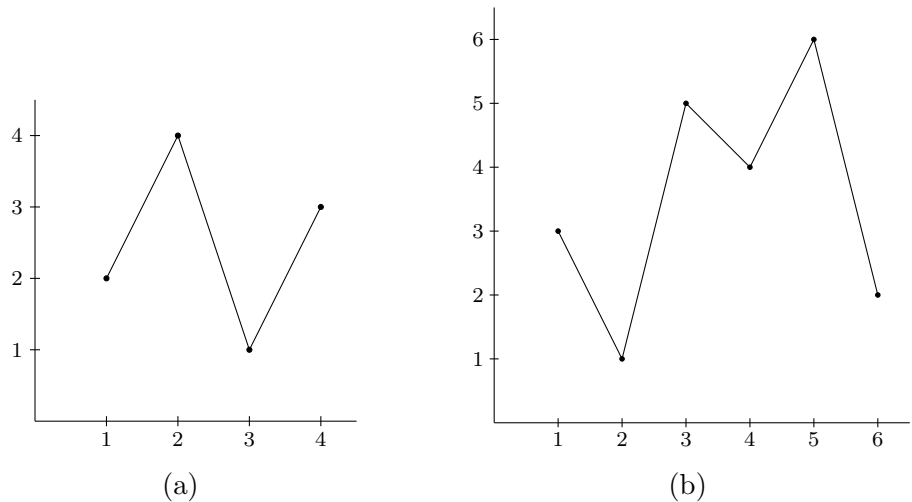
\end{center}
\end{example}

    Using the notion of \emph{pattern avoidance}, we can determine whether a permutation is FC or CFC by inspecting its one-line notation.
    If $w \in W(A_n)$, then $w$ avoids the pattern $321$ if there is no subset $\{i,j,k\} \subseteq \{1,\ldots,n+1\}$ with $i < j < k$ and $w(k) < w(j) < w(i)$.
    Similarly, a permutation $w$ avoids the pattern $3412$ if there is no subset $\{i,j,k,\ell\} \subseteq \{1,\ldots,n+1\}$ with $i < j < k < \ell$ and $w(k) < w(\ell) < w(i) < w(j)$.
    
    Note that the elements that constitute the 321 and 3412 patterns need not be consecutive. To have a 321 pattern, the one-line notation must have a strictly descending subsequence of three elements.
    Portions of the permutation line graphs corresponding to the patterns 321 and 3412 are shown in Figure~\ref{fig:patternmountains}.

\begin{center} 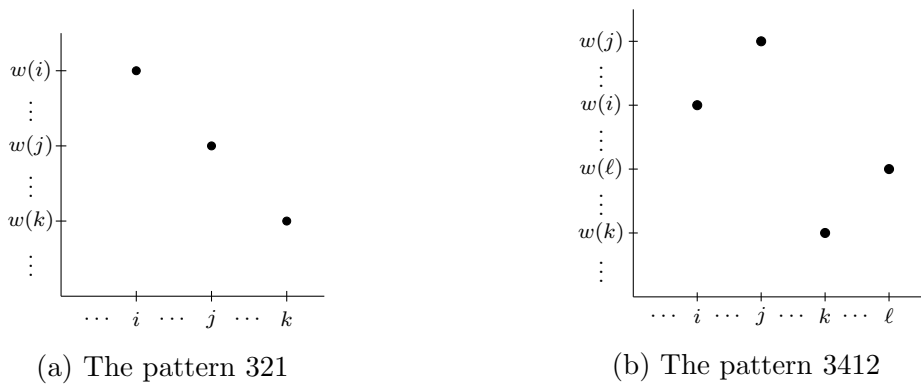
\begin{figure}[h!]
\begin{subfigure}{3in} \centering \vspace{10pt}
\begin{tikzpicture}
\draw (0,0)--(3.5,0); \draw (0,0)--(0,3.5);
\foreach \x in {1,2,3} \draw[shift={(\x,0)},color=black] (0pt,2pt)--(0pt,-2pt);
\foreach \y in {1,2,3} \draw[shift={(0,\y)},color=black] (2pt,0pt)--(-2pt,0pt);
\begin{scriptsize}
    \draw (1,-0.3) node {$i$};
    \draw (2,-0.3) node {$j$};
    \draw (3,-0.3) node {$k$};
    \draw (-0.4,1) node {$w(k)$};
    \draw (-0.4,2) node {$w(j)$};
    \draw (-0.4,3) node {$w(i)$};
    \draw[fill=black] (1,3) circle (1.5pt);
    \draw[fill=black] (2,2) circle (1.5pt);
    \draw[fill=black] (3,1) circle (1.5pt);
    \draw (0.5,-0.3) node {$\cdots$};
    \draw (1.5,-0.3) node {$\cdots$};
    \draw (2.5,-0.3) node {$\cdots$};
    \draw (-0.4,0.5) node {$\vdots$};
    \draw (-0.4,1.55) node {$\vdots$};
    \draw (-0.4,2.55) node {$\vdots$};
\end{scriptsize}\end{tikzpicture}
\caption{The pattern 321}\label{fig:321}
\end{subfigure}
\begin{subfigure}{3in} \centering
\begin{tikzpicture}[scale=0.85]
\draw (0,0)--(4.5,0); \draw (0,0)--(0,4.5);
\foreach \x in {1,2,3,4} \draw[shift={(\x,0)},color=black] (0pt,2pt)--(0pt,-2pt);
\foreach \y in {1,2,3,4} \draw[shift={(0,\y)},color=black] (2pt,0pt)--(-2pt,0pt);
\begin{scriptsize}
    \draw (1,-0.3) node {$i$};
    \draw (2,-0.3) node {$j$};
    \draw (3,-0.3) node {$k$};
    \draw (4,-0.3) node {$\ell$};
    \draw (-0.5,1) node {$w(k)$};
    \draw (-0.5,2) node {$w(\ell)$};
    \draw (-0.5,3) node {$w(i)$};
    \draw (-0.5,4) node {$w(j)$};
    \draw[fill=black] (1,3) circle (2pt);
    \draw[fill=black] (2,4) circle (2pt);
    \draw[fill=black] (3,1) circle (2pt);
    \draw[fill=black] (4,2) circle (2pt);
    \draw (0.5,-0.3) node {$\cdots$};
    \draw (1.5,-0.3) node {$\cdots$};
    \draw (2.5,-0.3) node {$\cdots$};
    \draw (3.5,-0.3) node {$\cdots$};
    \draw (-0.5,0.5) node {$\vdots$};
    \draw (-0.5,1.55) node {$\vdots$};
    \draw (-0.5,2.55) node {$\vdots$};
    \draw (-0.5,3.55) node {$\vdots$};
\end{scriptsize} \end{tikzpicture}
\caption{The pattern 3412}\label{fig:3412}
\end{subfigure}
\caption{The permutation line graphs of the 321 and 3412 patterns.} \label{fig:patternmountains}
\end{figure} \end{center}

\begin{example}\label{ex:patterns} Let $w \in W(A_5)$ have reduced expression $\w = 234513$. Then $w$ corresponds to $$(23)(34)(45)(56)(12)(34) = (13562)$$ in $S_6$. The one-line notation for $w$ is $[31\textcolor{ggreen}{54}6\textcolor{ggreen}{2}]$. 
    There is a 321 pattern in the one-line notation, highlighted in \textcolor{ggreen}{green}, but there is no 3412 pattern. The permutation line graph for $[315462]$ is shown in Figure~\ref{fig:permlinegraphs2}.
\end{example}

\begin{proposition}[Billey,~\cite{Billey2007}]\label{prop:321FC} An element $w \in W(A_n)$ is FC if and only if $w$ is 321-avoiding. \qed
\end{proposition}

    As a consequence of Proposition~\ref{prop:321FC}, the element from Example~\ref{ex:patterns} is not FC.
    
    The following proposition about pattern avoidance is from~\cite{Boothby2012}.
    
\begin{proposition}[Boothby, et al.,~\cite{Boothby2012}]\label{prop:patterns} An element $w \in W(A_n)$ is CFC if and only if $w$ is $321$- and $3412$-avoiding. \qed
\end{proposition}

\begin{example}\label{ex:permscyclic} We will now explore a few examples. \begin{enumerate}[leftmargin=0.75in, label=(\alph*)]
\item Let $W$ be the Coxeter group of type $A_3$. Then $W \cong S_4$. Let $w \in W$ have reduced expression $\w = 132$. Then $w$ is CFC since $w$ is a Coxeter element. In this case, its heap is shown in Figure~\ref{fig:cycheapex1}

\begin{center} \begin{figure}[H] \centering
\begin{tikzpicture}
    \sq{0}{1};   \node at (0.5,0.5) {$1$};
    \sq{1}{1};   \node at (1.5,0.5) {$3$};
    \sq{0.5}{0}; \node at (1,-0.5)  {$2$};
\end{tikzpicture}
\caption{The heap for a CFC element in $W(A_3)$.}\label{fig:cycheapex1}
\end{figure}
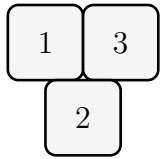 \end{center}
    We see that $w$ corresponds to the permutation $(12)(34)(23) = (1243) = [2413]$ in $S_4$ in cycle notation and one-line notation.
    The permutation line graph is shown in part (a) of Example~\ref{ex:linegraphs}.
    Since there are only four elements and is not $[3412]$ exactly, $w$ is clearly 3412-avoiding. It is also 321-avoiding because there is not a strictly decreasing subsequence of three elements in the one-line notation. We can also see this in the permutation line graph.
    These conclusions agree with Proposition~\ref{prop:patterns}.
	
\item Let $W$ be the Coxeter group of type $A_3$. Then $W \cong S_4$. Let $w \in W$ have reduced expression $\w = 3213$. Then, in cycle and one-line notations, we have that $w$ corresponds to $$(34)(23)(12)(34) = (124) = [2431]$$ in $S_4$.
    Since $431$ in the one-line notation is a 321 pattern, $[2431]$ is not 321-avoiding.
    We can see this in the permutation line graph, shown in Figure~\ref{fig:permline2431}, where the circled points correspond to the elements that constitute the 321 pattern.
    Then, by Proposition~\ref{prop:321FC}, $w$ is not FC, and hence not CFC. In this case, the heaps for $w$ are shown in Figures~\ref{fig:heapfor3213} and~\ref{fig:heapfor2321}.
    It is clear from the heaps that $w$ is not FC because each heap contains a convex subheap corresponding to the braid relation $323 = 232$.

\begin{center} 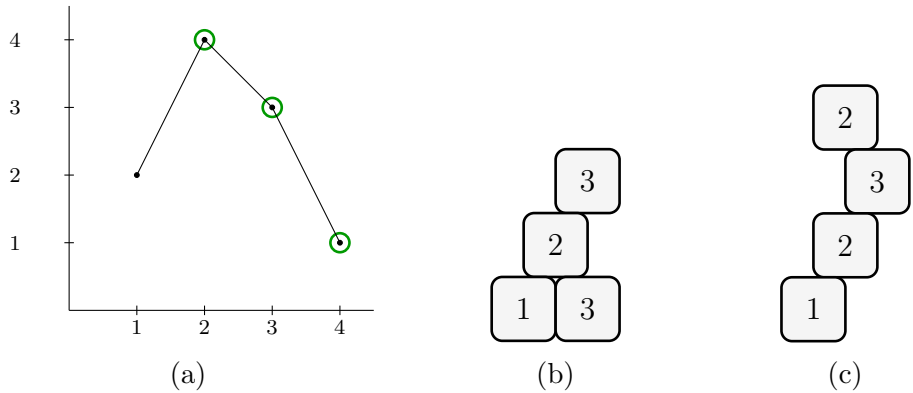
\begin{figure}[H] \centering
\begin{subfigure}{0.35\textwidth} \centering
\begin{tikzpicture}[scale=0.9]
\begin{scriptsize}
\draw (0,0)--(4.5,0); \draw (0,0)--(0,4.5);
\foreach \x in {1,2,3,4} \draw[shift={(\x,0)},color=black] (0pt,2pt)--(0pt,-2pt);
\foreach \y in {1,2,3,4} \draw[shift={(0,\y)},color=black] (2pt,0pt)--(-2pt,0pt);
    \draw (1,-0.25) node {$1$};
    \draw (2,-0.25) node {$2$};
    \draw (3,-0.25) node {$3$};
    \draw (4,-0.25) node {$4$};
    \draw (-0.8,1)  node {$1$};
    \draw (-0.8,2)  node {$2$};
    \draw (-0.8,3)  node {$3$};
    \draw (-0.8,4)  node {$4$};
    \draw[fill=black] (1,2) circle (1pt);
    \draw[fill=black] (2,4) circle (1pt); \draw[color=ggreen,line width=1pt] (2,4) circle (4pt);
    \draw[fill=black] (3,3) circle (1pt); \draw[color=ggreen,line width=1pt] (3,3) circle (4pt);
    \draw[fill=black] (4,1) circle (1pt); \draw[color=ggreen,line width=1pt] (4,1) circle (4pt);
    \draw (1,2)--(2,4)--(3,3)--(4,1);
\end{scriptsize} \end{tikzpicture}
\caption{}\label{fig:permline2431}
\end{subfigure}
\begin{subfigure}{0.225\textwidth} \centering \vspace{54pt}
\begin{tikzpicture}[scale=0.85]
    \sq{1}{2};   \node at (1.5,1.5) {$3$};
    \sq{0.5}{1}; \node at (1,0.5)   {$2$};
    \sq{0}{0};   \node at (0.5,-0.5){$1$};
    \sq{1}{0};   \node at (1.5,-0.5){$3$};
\end{tikzpicture}
\caption{}\label{fig:heapfor3213}
\end{subfigure}
\begin{subfigure}{0.225\textwidth} \centering \vspace{30pt}
\begin{tikzpicture}[scale=0.85]
    \sq{0.5}{3}; \node at (1,2.5)   {$2$};
    \sq{1}{2};   \node at (1.5,1.5) {$3$};
    \sq{0.5}{1}; \node at (1,0.5)   {$2$};
    \sq{0}{0};   \node at (0.5,-0.5){$1$};
\end{tikzpicture}
\caption{}\label{fig:heapfor2321}
\end{subfigure}
\caption{The permutation line graph and two heaps corresponding to some $w\in W(A_4)$.}\label{}
\end{figure} \end{center}

\item Let $W$ be the Coxeter group of type $A_n$ where $n$ is at least 7. Let $w \in W$ correspond to the element $[\cdots \textcolor{magenta}{5}1\textcolor{magenta}{83}2\textcolor{magenta}{4} \cdots]$ of $S_{n+1}$.
    Note that this element is not 3412-avoiding because the pink elements create a 3412 pattern. So, $w$ is not CFC.
    Moreover, $w$ is not even FC by Proposition~\ref{prop:321FC} because 832 exhibits a 321 pattern.
	
\item Let $W$ be the Coxeter group of type $A_3$. Then $W \cong S_4$. Let $w \in W$ have reduced expression $\w = 2132$. Then $w$ is FC and the heap of $w$ is shown in Figure~\ref{fig:heap2132}. Inspecting the heap makes it clear that $w$ is not CFC. Moreover, we see that the corresponding permutation is $$(23)(12)(34)(23) = (13)(24) = [3412],$$ which is obviously not 3412-avoiding but is 321-avoiding.

\begin{center} 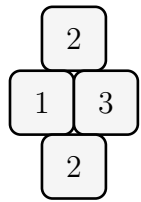
\begin{figure}[H] \centering
\begin{tikzpicture}[scale=0.85]
    \sq{0.5}{1}; \node at (1,0.5)   {$2$};
    \sq{0}{0};   \node at (0.5,-0.5){$1$};
    \sq{1}{0};   \node at (1.5,-0.5){$3$};
    \sq{0.5}{-1};\node at (1,-1.5)  {$2$};
\end{tikzpicture}
\caption{The heap of an FC element of $W(A_3)$.}\label{fig:heap2132}
\end{figure} \end{center}

\item \label{ex:cyclicshifts} Let $w \in W(A_4)$ have reduced expression $\w = 1234$. Then $w$ is FC and the heap of $w$ is shown in Figure~\ref{fig:heap1234}. Then $w$ corresponds to $(12)(23)(34)(45) = (12345)$.
    All possible sequences of cyclic shifts of $H(w)$ and their corresponding permutations in $S_5$ are shown in Figure~\ref{fig:cyclicshiftsof1234}. We shift the \textcolor{magenta}{pink} blocks to obtain the heap that follows. Since $w$ is a Coxeter element, it is CFC, and every cyclic shift of $w$ is also CFC by Remark~\ref{rem:CoxCFC}.
\begin{center} 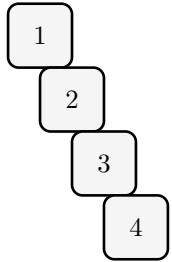
\begin{figure}[H] \centering
\begin{tikzpicture}[scale=0.85]
    \sq{0}{2};    \node at (0.5,1.5)  {\footnotesize $1$};
    \sq{0.5}{1};  \node at (1,0.5)    {\footnotesize $2$};
    \sq{1}{0};    \node at (1.5,-0.5) {\footnotesize $3$};
    \sq{1.5}{-1}; \node at (2,-1.5)   {\footnotesize $4$};
\end{tikzpicture} \caption{The heap of a CFC element of $W(A_4)$.}\label{fig:heap1234}
\end{figure} \end{center}
\end{enumerate}
\end{example}

\begin{center} \begin{figure}[h!] \centering
\begin{tabular}{@{}m{2.7cm} @{}m{0.9cm} @{}m{2.3cm} @{}m{2.7cm} @{}m{0.9cm} @{}m{1cm}}
\begin{tikzpicture}[scale=0.85]
    \sq{0}{0};    \node at (0.5,-0.5) {\footnotesize $1$};
    \sqm{0.5}{1}; \node at (1,0.5)    {\footnotesize $2$};
    \sq{1}{0};    \node at (1.5,-0.5) {\footnotesize $3$};
    \sq{1.5}{-1}; \node at (2,-1.5)   {\footnotesize $4$};
\end{tikzpicture} & $\mapsto$ &
    $(13452)$ &
\begin{tikzpicture}[scale=0.85]
    \sq{0}{0};    \node at (0.5,-0.5) {\footnotesize $1$};
    \sq{0.5}{-1}; \node at (1,-1.5)   {\footnotesize $2$};
    \sqm{1}{0};   \node at (1.5,-0.5) {\footnotesize $3$};
    \sq{1.5}{-1}; \node at (2,-1.5)   {\footnotesize $4$};
\end{tikzpicture} & $\mapsto$ &
    $(12453)$ \\ && \\
\begin{tikzpicture}[scale=0.85]
    \sqm{0}{0};   \node at (0.5,-0.5) {\footnotesize $1$};
    \sq{0.5}{-1}; \node at (1,-1.5)   {\footnotesize $2$};
    \sq{1}{-2};   \node at (1.5,-2.5) {\footnotesize $3$};
    \sq{1.5}{-1}; \node at (2,-1.5)   {\footnotesize $4$};
\end{tikzpicture} & $\mapsto$ &
    $(12354)$ &
\begin{tikzpicture}[scale=0.85]
    \sq{0}{0};    \node at (0.5,-0.5) {\footnotesize $1$};
    \sqm{0.5}{1}; \node at (1,0.5)    {\footnotesize $2$};
    \sq{1}{0};    \node at (1.5,-0.5) {\footnotesize $3$};
    \sq{1.5}{1};  \node at (2,0.5)    {\footnotesize $4$};
\end{tikzpicture} & $\mapsto$ &
    $(13542)$ \\ && \\
\begin{tikzpicture}[scale=0.85]
    \sqm{0}{0};   \node at (0.5,-0.5) {\footnotesize $1$};
    \sq{0.5}{-1}; \node at (1,-1.5)   {\footnotesize $2$};
    \sq{1}{0};    \node at (1.5,-0.5) {\footnotesize $3$};
    \sq{1.5}{1};  \node at (2,0.5)    {\footnotesize $4$};
\end{tikzpicture} & $\mapsto$ &
    $(14532)$ &
\begin{tikzpicture}[scale=0.85]
    \sq{0}{-2};   \node at (0.5,-2.5) {\footnotesize $1$};
    \sq{0.5}{-1}; \node at (1,-1.5)   {\footnotesize $2$};
    \sq{1}{0};    \node at (1.5,-0.5) {\footnotesize $3$};
    \sqm{1.5}{1}; \node at (2,0.5)    {\footnotesize $4$};
\end{tikzpicture} & $\mapsto$ &
    $(15432)$
\end{tabular}
\begin{center} \begin{tabular}{@{}m{2.7cm} @{}m{0.9cm} @{}m{1.5cm}} \vspace{14pt} \begin{tikzpicture}[scale=0.85]
    \sq{0}{0};    \node at (0.5,-0.5) {\footnotesize $1$};
    \sq{0.5}{1};  \node at (1,0.5)    {\footnotesize $2$};
    \sq{1}{2};    \node at (1.5,1.5)  {\footnotesize $3$};
    \sq{1.5}{1};  \node at (2,0.5)    {\footnotesize $4$};
\end{tikzpicture} & $\mapsto$ &
    $(14532)$
\end{tabular} \end{center}
\caption{Cyclic shifts of the heap of a CFC element of $W(A_4)$ and the corresponding permutations.}\label{fig:cyclicshiftsof1234}
\end{figure}
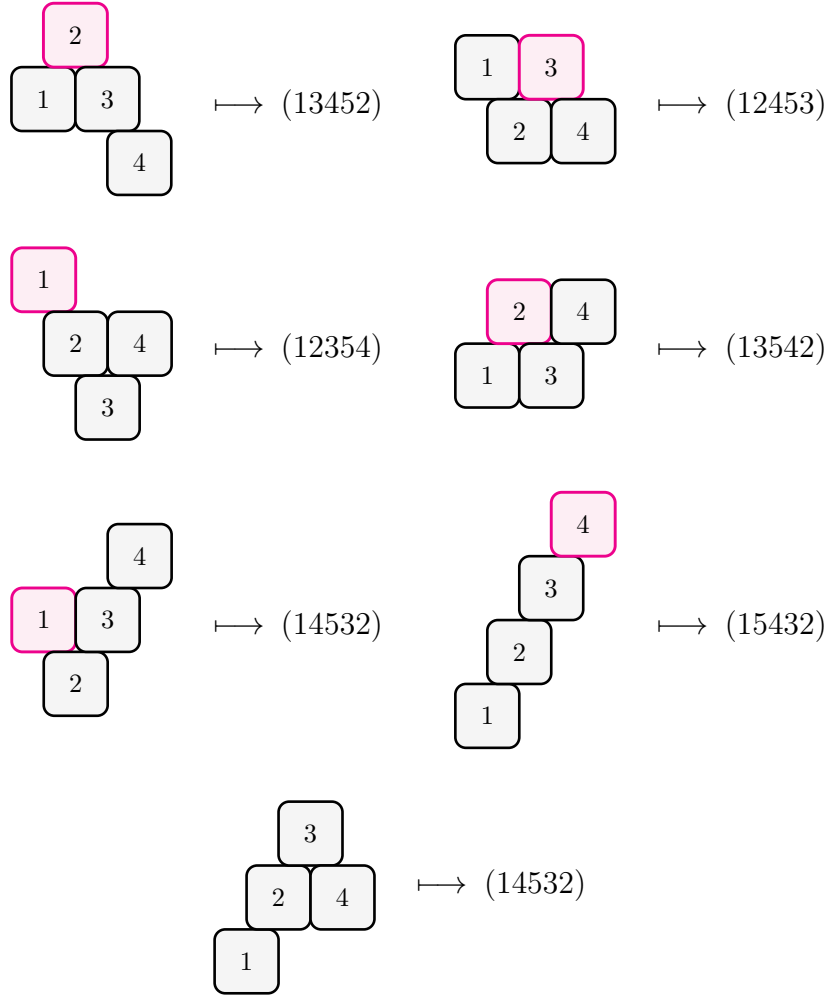 \end{center}

    Recall that two permutations are conjugate if and only if they have the same cycle type.
    Since Coxeter elements in $W(A_n)$ correspond to $(n+1)$-cycles in $S_{n+1}$, all Coxeter elements are conjugate as they have the same cycle type. So, all Coxeter elements are cyclically equivalent by Theorem~\ref{thm:e2} and as seen in Figure~\ref{fig:A4boxes}, but not all $(n+1)$-cycles correspond to Coxeter elements.
     
    For example, $1234 = (12)(23)(34)(45) = (12345)$. There are seven other Coxeter elements conjugate to $1234$, as shown in Figure~\ref{fig:coxeltsinA4}, but there are 24 distinct 5-cycles in $S_5$.

    Given a product of disjoint cycles, we want to be able to determine if the group element corresponding to the permutation is a CFC element.
    For example, which 4-cycles in $S_4$ correspond to CFC elements in $A_3$?
    In order to attempt to answer this question, we need a couple definitions.

    Let $(\cdots i~w(i)~w^2(i) \cdots)$ be a cycle in the permutation corresponding to $w \in W(A_n)$, assuming the smallest element appearing in the cycle is written first. Then there is a \emph{direction change at $w(i)$} if
\begin{enumerate}[label=(\alph*),leftmargin=0.75in]
    \item $i < w(i)$ and $w(i) > w^2(i)$ or
    \item $i > w(i)$ and $w(i) < w^2(i)$.
\end{enumerate}

\begin{example} Consider the symmetric group $S_6$. Let $w = (12435)$ and $y = (135)(246)$ in $S_6$. Then neither cycle for $y$ has a direction change, but there is a direction change at 4 in $w$ since $i = 2 < w(i) = 4$ and $4 = w(i) > w^2(i) = 3$. There is also a direction change at 3 in $w$.
\end{example}

    We define the \emph{support of a cycle $c$ of $w$} to be the set of numbers appearing in the cycle, denoted by $\cyclesupp(c)$.
    Note that $\cyclesupp(c)$ is not the same set as $\supp(w)$, even in the case when $w$ corresponds to a single cycle.
    We say a cycle $c$ has \emph{connected support} if the support of $c$ is a set of consecutive numbers.

\begin{example} The support of the permutation $(1357)$ is $\{1,3,5,7\}$, so $(1357)$ does not have connected support. However, the permutation $(234)$ does have connected support, namely $\cyclesupp((234)) = \{2,3,4\}$.
\end{example}

    Utilizing Sage~\cite{sage}, we witnessed evidence of the following conjecture, which we believe is true in general.

\begin{conjecture}\label{conjecture}
Let $w \in W(A_n)$ correspond to a permutation with disjoint cycles $c_1, c_2, \ldots, c_k$ in $S_{n+1}$. Assume each $c_j$ is written with the smallest number first.
    Then $w \in \CFC(A_n)$ if and only if each $c_j$ has connected support and has at most one direction change.
\end{conjecture}

\begin{example} We return to part (e) of Example~\ref{ex:permscyclic}. In that example, we have a 5-cycle corresponding to each cyclic shift of the heap of $1234$.
    Note that each of the 5-cycles satisfies both conditions of Conjecture~\ref{conjecture}.
    However, the cycle $w = (14352)$ has three direction changes, namely, at 4, 3, and 5. 
    One heap of $w$ corresponding to the reduced expression $324134$ is shown in Figure~\ref{fig:dirchangeheap}. Then $w$ is not CFC since $w$ is not FC, due to the appearance of a 434 subword, satisfying Conjecture~\ref{conjecture}.
\begin{center} 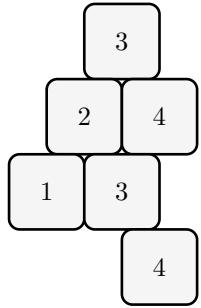
\begin{figure}[H] \centering
\begin{tikzpicture}
    \sq{1}{4};   \node at (1.5,3.5) {\footnotesize $3$};
    \sq{0.5}{3}; \node at (1,2.5)   {\footnotesize $2$};
    \sq{1.5}{3}; \node at (2,2.5)   {\footnotesize $4$};
    \sq{0}{2};   \node at (0.5,1.5) {\footnotesize $1$};
    \sq{1}{2};   \node at (1.5,1.5) {\footnotesize $3$};
    \sq{1.5}{1}; \node at (2,0.5)   {\footnotesize $4$};
\end{tikzpicture}
\caption{A heap for an element of $W(A_4)$ that corresponds to a permutation in $S_5$ with more than one direction change.}\label{fig:dirchangeheap}
\end{figure} \end{center}
\end{example}
    
    Cycle type provides insight into the structure of the sets of conjugate CFC elements. However, our ultimate goal is to generalize to other types of Coxeter groups, where cycle type is not available.


\chapter{Conjugacy classes of CFC elements in Coxeter groups of type $A_n$}

    In this chapter, we focus exclusively on Coxeter systems of type $A_n$.

\section{Chunks and rings}\label{sec:chunks}
    In order to formalize the ``sliding" of cylindrical heaps as first mentioned in Example~\ref{ex:A4boxes}, we develop the notion of chunks and rings.    
    Recall from Section~\ref{sec:CFC} that the CFC elements in Coxeter groups of type $A_n$ are those elements that correspond to subexpressions of Coxeter elements.
    
\begin{definition} Let $w \in \CFC(A_n)$. Then we refer to a \emph{diagonal heap (or diagonal subheap) of size $m$} as shown in Figure~\ref{fig:diagonalheap} where $k'=k+m-1$, $1 \leq k \leq n$, and $1 \leq m \leq n-k+1$.
\begin{center} 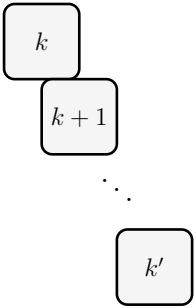
\begin{figure}[H] \centering
\begin{tikzpicture}
    \sq{0}{2};    \node at (0.5,1.5)   {\scalebox{0.8}{$k$}};
    \sq{0.5}{1};  \node at (1,0.5)     {\scalebox{0.8}{$k+1$}};
                  \node at (1.5,-0.35) {$\ddots$};
    \sq{1.5}{-1}; \node at (2,-1.5)    {\scalebox{0.8}{$k'$}};
\end{tikzpicture} \caption{A diagonal heap.}\label{fig:diagonalheap}
\end{figure} \end{center}
\end{definition}

\begin{definition} Let $w \in \CFC(A_n)$. We call a convex subheap of the heap of $w$ consisting of $m$ blocks a \emph{chunk of size $m$} if it corresponds to a maximal connected component of the underlying Hasse diagram for the heap of $w$.
\end{definition}

    Notice that diagonal heaps are examples of heaps consisting of exactly one chunk.
    
\begin{example}\label{ex:chunks} Let $w \in W(A_6)$ have reduced expression $\w = 12356$. Then $w$ is CFC, and its heap is shown in Figure~\ref{fig:chunkex1}. The underlying Hasse diagram is shown in Figure~\ref{fig:chunkex2}.
    Note that there are two connected components of the Hasse diagram, and hence two chunks in $H(w)$.
    We can see the chunks in the original heap, as well. The vertical line in Figure~\ref{fig:chunkex1} shows the separation of the two chunks. The \textcolor{magenta}{pink} and \textcolor{blue}{blue} chunks can move independently in the vertical direction.

\begin{center} 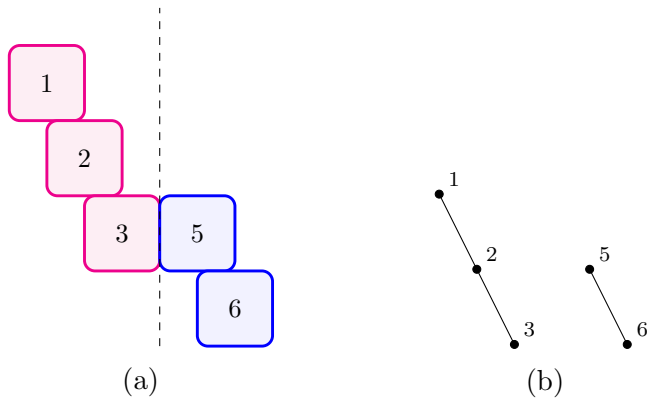
\begin{figure}[H] \centering
\begin{tabular}[b]{cc}
\begin{subfigure}{0.3\textwidth} \centering
\begin{tikzpicture}
    \sqm{0}{3};    \node at (0.5,2.5) {\footnotesize $1$};
    \sqm{0.5}{2};  \node at (1,1.5)   {\footnotesize $2$};
    \sqm{1}{1};    \node at (1.5,0.5) {\footnotesize $3$};
    \sqbl{2}{1};   \node at (2.5,0.5) {\footnotesize $5$};
    \sqbl{2.5}{0}; \node at (3,-0.5)  {\footnotesize $6$};
    \draw[dashed] (2,-1)--(2,3.5);
\end{tikzpicture}
\caption{}\label{fig:chunkex1}
\end{subfigure} &
\begin{subfigure}{0.3\textwidth} \centering \vspace{60pt}
\begin{tikzpicture}
\draw (0.5,2)--(1,1)--(1.5,0); \draw (2.5,1)--(3,0);
\begin{scriptsize}
    \draw [fill=black] (0.5,2) circle (1.5pt);
    \draw (0.7,2.2) node {1};
    \draw [fill=black] (1,1) circle (1.5pt);
    \draw (1.2,1.2) node {2};
    \draw [fill=black] (1.5,0) circle (1.5pt);
    \draw (1.7,0.2) node {3};
    \draw [fill=black] (2.5,1) circle (1.5pt);
    \draw (2.7,1.2) node {5};
    \draw [fill=black] (3,0) circle (1.5pt);
    \draw (3.2,0.2) node {6};
\end{scriptsize}
\end{tikzpicture}
\caption{}\label{fig:chunkex2}
\end{subfigure} \\ & \end{tabular}
\caption{The heap for a CFC element with its chunks colored \textcolor{magenta}{pink} and \textcolor{blue}{blue} together with its underlying Hasse diagram.}\label{fig:chunkex}
\end{figure} \end{center}
\end{example}

    Suppose $w,w' \in \CFC(A_n)$ such that $\hat{H}(w) = \hat{H}(w')$. Suppose $H(w)$ consists of a single chunk of size $k$. Then $H(w')$ also consists of a single chunk of size $k$ having the same blocks as $H(w)$, by Theorem~\ref{thm:e2}.
    
    Now, suppose $w,w' \in \CFC(A_n)$ such that $H(w)$ and $H(w')$ each have chunks $C$ and $C'$, respectively, with exactly the same blocks but not necessarily in the same configuration. Then it follows from Theorem~\ref{thm:e2} that we can obtain $C'$ by applying cyclic shifts to $C$.
    In this case, we say that $C$ and $C'$ are \emph{chunk equivalent}.
    Observe that two chunks are chunk equivalent if and only if they differ by a sequence of cyclic shifts. This generates an equivalence relation on the set of possible chunks of heaps for CFC elements in $W(A_n)$.

    Define $\hat{C}$ to be the equivalence class of the chunk $C$, which we visualize by wrapping representatives on a cylinder. It is easy to show that every chunk equivalence class contains a diagonal representative.
    We call the equivalence class of a chunk a \emph{ring}. That is, a ring is a chunk wrapped on a cylinder.
    
    For each $w \in \CFC(A_n)$, $\hat{H}(w)$ can be thought of as a disjoint union of rings. We can obtain a new cylindrical heap by \emph{sliding} rings by adding some $j \in \Z$ to the label of each block in the ring.
    We say two rings are \emph{slide equivalent} if we can slide one ring $j \in \Z$ spaces to obtain the other ring.
    Note that two rings that are slide equivalent have the same number of blocks.
    Two cylindrical heaps are \emph{slide equivalent} if we can slide the rings of one cylindrical heap to obtain the other cylindrical heap.
    Note that the corresponding rings are in the same order.

\begin{example} Let $w,y \in \CFC(A_9)$ have reduced expressions $\w = 124567$ and $\y = 236789$, respectively.
    Then the cylindrical heaps corresponding to $H(w)$ and $H(y)$ are shown in Figure~\ref{fig:slideexample1.1} and~\ref{fig:slideexample1.2}, respectively. Each of the cylindrical heaps consists of two rings.
    We can see that $\hat{H}(w)$ and $\hat{H}(y)$ are slide equivalent since we can add one to each of the labels of the blocks in the first ring of $\hat{H}(w)$ to obtain the first ring of $\hat{H}(y)$ and add two to each of the blocks in the second ring of $\hat{H}(w)$ to obtain the second ring of $\hat{H}(y)$.
\begin{center} \begin{figure}[H] \centering
\begin{subfigure}{0.4\textwidth} \centering
\begin{tikzpicture}[scale=0.8]
\draw[line width=1.5pt,->] (-0.5,3)--(4.5,3);
    \sq{0}{3};   \node at (0.5,2.5) {\footnotesize $1$};
    \sq{0.5}{2}; \node at (1,1.5)   {\footnotesize $2$};
    \sq{1.5}{2}; \node at (2,1.5)   {\footnotesize $4$};
    \sq{2}{1};   \node at (2.5,0.5) {\footnotesize $5$};
    \sq{2.5}{0}; \node at (3,-0.5)  {\footnotesize $6$};
    \sq{3}{-1};  \node at (3.5,-1.5){\footnotesize $7$};
\draw[line width=1.5pt,->] (-0.5,-2)--(4.5,-2);
\end{tikzpicture}
\caption{}\label{fig:slideexample1.1}
\end{subfigure}
\begin{subfigure}{0.4\textwidth} \centering
\begin{tikzpicture}[scale=0.8]
\draw[line width=1.5pt,->] (-0.5,3)--(5,3);
    \sq{0}{3};   \node at (0.5,2.5) {\footnotesize $2$};
    \sq{0.5}{2}; \node at (1,1.5)   {\footnotesize $3$};
    \sq{2}{2};   \node at (2.5,1.5) {\footnotesize $6$};
    \sq{2.5}{1}; \node at (3,0.5)   {\footnotesize $7$};
    \sq{3}{0};   \node at (3.5,-0.5){\footnotesize $8$};
    \sq{3.5}{-1};\node at (4,-1.5)  {\footnotesize $9$};
\draw[line width=1.5pt,->] (-0.5,-2)--(5,-2);
\end{tikzpicture}
\caption{}\label{fig:slideexample1.2}
\end{subfigure}
\caption{Two slide equivalent cylindrical heaps.}\label{}
\end{figure}
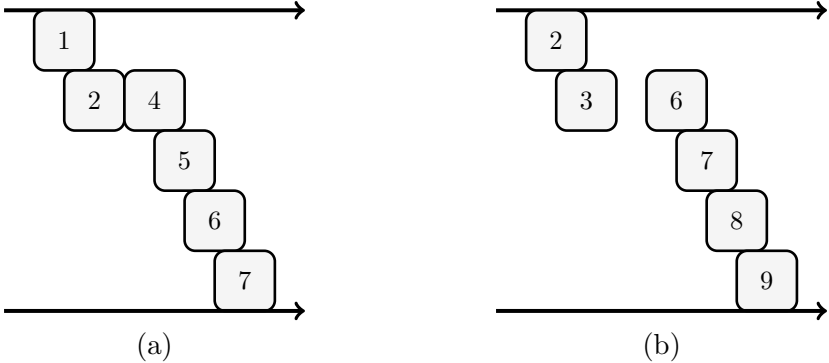 \end{center}
\end{example}
    
\begin{definition}
Suppose $w,w' \in \CFC(A_n)$ are such that $\hat{H}(w)$ and $\hat{H}(w')$ consist exactly of the rings $R_1,\ldots,R_k$ and $R_1',\ldots,R_k'$, respectively.
    Then $\hat{H}(w)$ and $\hat{H}(w')$ are \emph{ring equivalent} if there exists a bijection $R_i \longleftrightarrow R_i'$ such that $R_i$ and $R_i'$ consist of the same number of blocks (not necessarily with the same labels).
\end{definition}
    
\begin{example} Let $w \in W(A_6)$ have reduced expression $\w = 12356$. Then $w$ is CFC, and its heap is shown in Figure~\ref{fig:chunkex1}. Recall that the heap consists of two chunks.
    The cylindrical heap $\hat{H}(w)$ is shown in Figure~\ref{fig:cylheapex1} and the rings are the chunks wrapped on a cylinder, i.e., the rings are as shown in Figure~\ref{fig:cylheapex2}.
    We see that the cylindrical heap in Figure~\ref{fig:cylheapex3} is ring equivalent to $\hat{H}(w)$ because there are exactly two rings, one of size three and one of size two, in each cylindrical heap.

\begin{center} 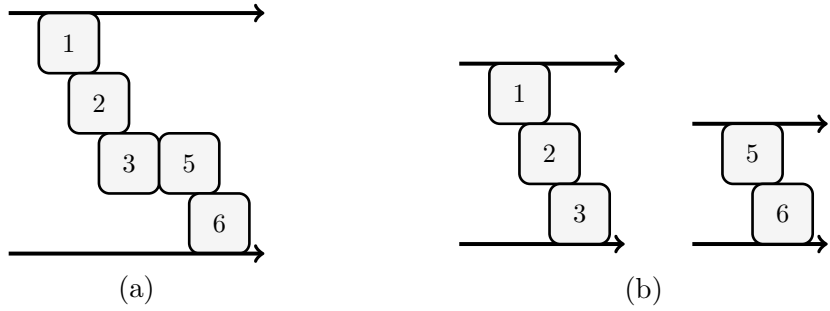
\begin{figure}[H] \centering
\begin{subfigure}{0.4\textwidth} \centering
\begin{tikzpicture}[scale=0.8]
\draw[line width=1.5pt,->] (-0.5,3)--(3.75,3);
    \sq{0}{3};   \node at (0.5,2.5) {\footnotesize $1$};
    \sq{0.5}{2}; \node at (1,1.5)   {\footnotesize $2$};
    \sq{1}{1};   \node at (1.5,0.5) {\footnotesize $3$};
    \sq{2}{1};   \node at (2.5,0.5) {\footnotesize $5$};
    \sq{2.5}{0}; \node at (3,-0.5)  {\footnotesize $6$};
\draw[line width=1.5pt,->] (-0.5,-1)--(3.75,-1);
\end{tikzpicture}
\caption{}\label{fig:cylheapex1}
\end{subfigure}
\begin{subfigure}{0.4\textwidth} \centering \vspace{20pt}
\begin{tabular}{ccc}
\begin{tikzpicture}[scale=0.8]
\draw[line width=1.5pt,->] (-0.5,3)--(2.25,3); \draw[line width=1.5pt,->] (-0.5,0)--(2.25,0);
    \sq{0}{3};   \node at (0.5,2.5) {\footnotesize $1$};
    \sq{0.5}{2}; \node at (1,1.5)   {\footnotesize $2$};
    \sq{1}{1};   \node at (1.5,0.5) {\footnotesize $3$};
\end{tikzpicture} &&
\begin{tikzpicture}[scale=0.8]
\draw[line width=1.5pt,->] (-0.5,1)--(1.75,1); \draw[line width=1.5pt,->] (-0.5,-1)--(1.75,-1);
    \sq{0}{1};   \node at (0.5,0.5) {\footnotesize $5$};
    \sq{0.5}{0}; \node at (1,-0.5)  {\footnotesize $6$};
\end{tikzpicture} \end{tabular}
\caption{}\label{fig:cylheapex2}
\end{subfigure}
\caption{The cylindrical heap and rings for some CFC element.\label{fig:cylheap}}
\end{figure}
\end{center}

\begin{figure}[H] \centering \vspace{-30pt}
\begin{subfigure}{0.4\textwidth} \centering
\begin{tikzpicture}[scale=0.8]
\draw[line width=1.5pt,->] (-0.5,3)--(4.25,3);
    \sq{0}{3};   \node at (0.5,2.5) {\footnotesize $3$};
    \sq{0.5}{2}; \node at (1,1.5)   {\footnotesize $4$};
    \sq{2}{2};   \node at (2.5,1.5) {\footnotesize $7$};
    \sq{2.5}{1}; \node at (3,0.5)   {\footnotesize $8$};
    \sq{3}{0};   \node at (3.5,-0.5){\footnotesize $9$};
\draw[line width=1.5pt,->] (-0.5,-1)--(4.25,-1);
\end{tikzpicture}
\caption{}\label{cylheapex3.1}
\end{subfigure}
\begin{subfigure}{0.4\textwidth} \centering \vspace{18pt}
\begin{tabular}{ccc}
\begin{tikzpicture}[scale=0.8]
\draw[line width=1.5pt,->] (-0.5,1)--(1.75,1); \draw[line width=1.5pt,->] (-0.5,-1)--(1.75,-1);
    \sq{0}{1};   \node at (0.5,0.5) {\footnotesize $3$};
    \sq{0.5}{0}; \node at (1,-0.5)  {\footnotesize $4$};
\end{tikzpicture} &&
\begin{tikzpicture}[scale=0.8]
\draw[line width=1.5pt,->] (-0.5,3)--(2.25,3); \draw[line width=1.5pt,->] (-0.5,0)--(2.25,0);
    \sq{0}{3};   \node at (0.5,2.5) {\footnotesize $7$};
    \sq{0.5}{2}; \node at (1,1.5)   {\footnotesize $8$};
    \sq{1}{1};   \node at (1.5,0.5) {\footnotesize $9$};
\end{tikzpicture} \end{tabular}
\caption{}\label{fig:cylheapex3.2}
\end{subfigure}
\caption{A cylindrical heap ring equivalent to the one in Figure~\ref{fig:cylheapex1} together with its rings.}\label{fig:cylheapex3}
\end{figure}
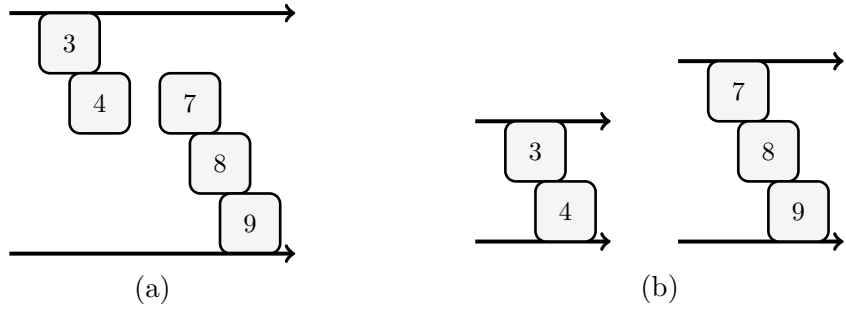
\end{example}

\begin{definition}
We refer to a heap as \emph{simple} if it is as shown in Figure~\ref{fig:simple}, where each chunk is diagonal, the leftmost chunk starts at 1, and the chunks are as close to each other as possible.
\begin{center} 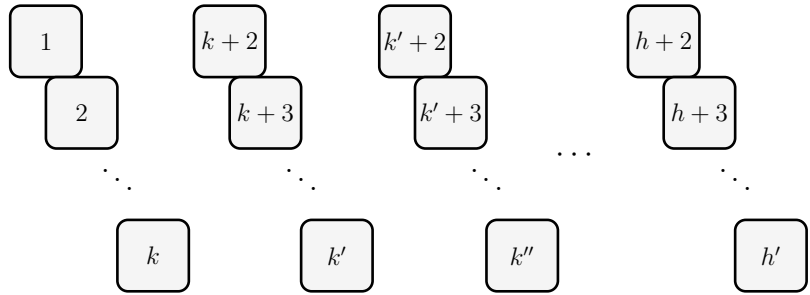
\begin{figure}[H] \centering
\begin{tabular}{m{2cm} m{2cm} m{2cm} m{0.5cm} m{2cm}}
\begin{tikzpicture}[scale=0.95]
    \sq{0}{3};    \node at (0.5,2.5) {\scalebox{0.8}{$1$}};
    \sq{0.5}{2};  \node at (1,1.5)   {\scalebox{0.8}{$2$}};
                  \node at (1.5,0.7) {$\ddots$};
    \sq{1.5}{0};  \node at (2,-0.5)  {\scalebox{0.8}{$k$}};
\end{tikzpicture} &
\begin{tikzpicture}[scale=0.95]
    \sq{2.5}{3};  \node at (3,2.5)   {\scalebox{0.8}{$k+2$}};
    \sq{3}{2};    \node at (3.5,1.5) {\scalebox{0.8}{$k+3$}};
                  \node at (4,0.7)   {$\ddots$};
    \sq{4}{0};    \node at (4.5,-0.5){\scalebox{0.8}{$k'$}};
\end{tikzpicture} &
\begin{tikzpicture}[scale=0.95]
    \sq{2.5}{3};  \node at (3,2.5)   {\scalebox{0.8}{$k'+2$}};
    \sq{3}{2};    \node at (3.5,1.5) {\scalebox{0.8}{$k'+3$}};
                  \node at (4,0.7)   {$\ddots$};
    \sq{4}{0};    \node at (4.5,-0.5){\scalebox{0.8}{$k''$}};
\end{tikzpicture} & $\cdots$ &
\begin{tikzpicture}[scale=0.95]
    \sq{2.5}{3};  \node at (3,2.5)   {\scalebox{0.8}{$h+2$}};
    \sq{3}{2};    \node at (3.5,1.5) {\scalebox{0.8}{$h+3$}};
                  \node at (4,0.7)   {$\ddots$};
    \sq{4}{0};    \node at (4.5,-0.5){\scalebox{0.8}{$h'$}};
\end{tikzpicture}
\end{tabular}
\caption{A simple heap.}\label{fig:simple}
\end{figure}
\end{center}
\end{definition}

\begin{remark} Note that if $w \in \CFC(A_n)$, then there exists some $y \in \CFC(A_n)$ having a simple heap such that $\hat{H}(w)$ is ring equivalent to $\hat{H}(y)$.
\end{remark}

\section{Conjugacy classes of CFC elements in $W(A_n)$}\label{sec:conjugacyofCFC}
    In this section, we give a constructive description of conjugate CFC elements. The goal of this section is to prove the following theorem.
    
\begin{theorem}\label{thm:conjiffring} Let $w,y \in \CFC(A_n)$. Then $w$ and $y$ are conjugate if and only if $\hat{H}(w)$ and $\hat{H}(y)$ are ring equivalent.
\end{theorem}

    The following example motivates the proof of Lemma~\ref{lem:slide}.

\begin{example}\label{ex:slide} Let $w,y \in W(A_7)$ have reduced expressions $\w = 3456$ and $\y = 4567$, respectively. Note that $w$ and $y$ are both CFC, so there is a unique heap for each, as shown in Figures~\ref{fig:slideex0.1} and~\ref{fig:slideex0.2}, respectively. Notice that $\hat{H}(w)$ and $\hat{H}(y)$ are ring equivalent.

\begin{center} \begin{figure}[H] \centering
\begin{subfigure}{0.3\textwidth} \centering
\begin{tikzpicture}[scale=0.85]
    \sq{0}{10};       \node at (0.5,9.5) {\footnotesize $3$};
    \sq{0.5}{9};      \node at (1,8.5)   {\footnotesize $4$};
    \sq{1}{8};        \node at (1.5,7.5) {\footnotesize $5$};
    \sq{1.5}{7};      \node at (2,6.5)   {\footnotesize $6$};
\end{tikzpicture}
\caption{}\label{fig:slideex0.1}
\end{subfigure}
\begin{subfigure}{0.3\textwidth} \centering
\begin{tikzpicture}[scale=0.85]
    \sq{0}{10};       \node at (0.5,9.5) {\footnotesize $4$};
    \sq{0.5}{9};      \node at (1,8.5)   {\footnotesize $5$};
    \sq{1}{8};        \node at (1.5,7.5) {\footnotesize $6$};
    \sq{1.5}{7};      \node at (2,6.5)   {\footnotesize $7$};
\end{tikzpicture}
\caption{}\label{fig:slideex0.2}
\end{subfigure}
\caption{The heaps for Example~\ref{ex:slide}.}\label{fig:slideex0}
\end{figure}
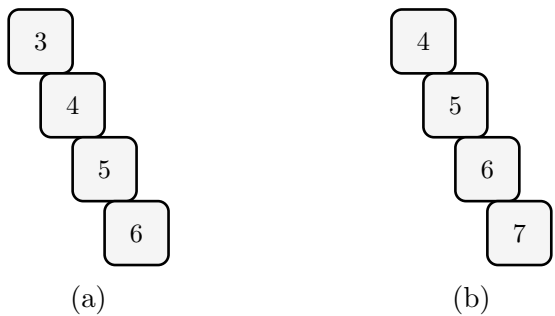 \end{center}

    We claim that we can obtain $y$ by conjugating $w$ by $x$, where $x$ has reduced expression $\x = 34567$.
    Then the heap of ${\x}{\w}{\x^{-1}}$ is shown in Figure~\ref{fig:slideexheap1}, where the \textcolor{orange}{orange} blocks correspond to the heap of $x$, the \textcolor{blue}{blue} blocks correspond to the heap of $w$, and the \textcolor{ggreen}{green} blocks correspond to the heap of $x^{-1}$.
    Then applying Lemma~\ref{lem:stst} to the extra long $76$-chain, denoted in the heap in Figure~\ref{fig:slideexheap1} by the hatched blocks, we get the heap shown in Figure~\ref{fig:slideexheap2}.
    Now we can apply Lemma~\ref{lem:stst} to the extra long $65$-chain to get the heap in Figure~\ref{fig:slideexheap3}.
    Continuing in this manner, applying Lemma~\ref{lem:stst} to the new extra long $s_is_j$-chains created, we get the heap in Figure~\ref{fig:slideexheap4.1}.
    Applying Lemma~\ref{lem:stst} to the extra long $34$-chain, denoted in the heap by the hatched blocks, we get the heap in Figure~\ref{fig:slideexheap4.2}.
    Then, canceling the two adjacent $3$ blocks, denoted in the heap by the checked blocks, the result follows, as shown in Figure~\ref{fig:slideexheap4.3}.
    
\begin{center} \begin{figure}[H] \centering
\begin{subfigure}{0.3\textwidth} \centering
\begin{tikzpicture}
    \sqor{0}{10};       \node at (0.5,9.5) {\footnotesize $3$};
    \sqor{0.5}{9};      \node at (1,8.5)   {\footnotesize $4$};
    \sqor{1}{8};        \node at (1.5,7.5) {\footnotesize $5$};
    \sqor{1.5}{7};      \node at (2,6.5)   {\footnotesize $6$};
    \sqorhash{2}{6};    \node at (2.5,5.5) {\footnotesize $7$};
    \sqbl{0}{8};        \node at (0.5,7.5) {\footnotesize $3$};
    \sqbl{0.5}{7};      \node at (1,6.5)   {\footnotesize $4$};
    \sqbl{1}{6};        \node at (1.5,5.5) {\footnotesize $5$};
    \sqblhash{1.5}{5};  \node at (2,4.5)   {\footnotesize $6$};
    \sqgrhash{2}{4};    \node at (2.5,3.5) {\footnotesize $7$};
    \sqgrhash{1.5}{3};  \node at (2,2.5)   {\footnotesize $6$};
    \sqgr{1}{2};        \node at (1.5,1.5) {\footnotesize $5$};
    \sqgr{0.5}{1};      \node at (1,0.5)   {\footnotesize $4$};
    \sqgr{0}{0};        \node at (0.5,-0.5){\footnotesize $3$};
\end{tikzpicture}
\caption{}\label{fig:slideexheap1}
\end{subfigure}
\begin{subfigure}{0.3\textwidth} \centering \vspace{56pt}
\begin{tikzpicture}
    \sqor{0}{10};       \node at (0.5,9.5) {\footnotesize $3$};
    \sqor{0.5}{9};      \node at (1,8.5)   {\footnotesize $4$};
    \sqor{1}{8};        \node at (1.5,7.5) {\footnotesize $5$};
    \sqorhash{1.5}{7};  \node at (2,6.5)   {\footnotesize $6$};
    \sqbl{0}{8};        \node at (0.5,7.5) {\footnotesize $3$};
    \sqbl{0.5}{7};      \node at (1,6.5)   {\footnotesize $4$};
    \sqblhash{1}{6};    \node at (1.5,5.5) {\footnotesize $5$};
    \sqblhash{1.5}{5};  \node at (2,4.5)   {\footnotesize $6$};
    \sqgr{2}{4};        \node at (2.5,3.5) {\footnotesize $7$};
    \sqgrhash{1}{4};    \node at (1.5,3.5) {\footnotesize $5$};
    \sqgr{0.5}{3};      \node at (1,2.5)   {\footnotesize $4$};
    \sqgr{0}{2};        \node at (0.5,1.5) {\footnotesize $3$};
\end{tikzpicture}
\caption{}\label{fig:slideexheap2}
\end{subfigure}
\begin{subfigure}{0.3\textwidth} \centering \vspace{112pt}
\begin{tikzpicture}
    \sqor{0}{10};       \node at (0.5,9.5) {\footnotesize $3$};
    \sqor{0.5}{9};      \node at (1,8.5)   {\footnotesize $4$};
    \sqorhash{1}{8};    \node at (1.5,7.5) {\footnotesize $5$};
    \sqbl{0}{8};        \node at (0.5,7.5) {\footnotesize $3$};
    \sqblhash{0.5}{7};  \node at (1,6.5)   {\footnotesize $4$};
    \sqblhash{1}{6};    \node at (1.5,5.5) {\footnotesize $5$};
    \sqbl{1.5}{5};      \node at (2,4.5)   {\footnotesize $6$};
    \sqgr{2}{4};        \node at (2.5,3.5) {\footnotesize $7$};
    \sqgrhash{0.5}{5};  \node at (1,4.5)   {\footnotesize $4$};
    \sqgr{0}{4};        \node at (0.5,3.5) {\footnotesize $3$};
\end{tikzpicture}
\caption{}\label{fig:slideexheap3}
\end{subfigure}
\begin{subfigure}{0.3\textwidth} \centering
\begin{tikzpicture}
    \sqorhash{0}{10};   \node at (0.5,9.5) {\footnotesize $3$};
    \sqorhash{0.5}{9};  \node at (1,8.5)   {\footnotesize $4$};
    \sqblhash{0}{8};    \node at (0.5,7.5) {\footnotesize $3$};
    \sqblhash{0.5}{7};  \node at (1,6.5)   {\footnotesize $4$};
    \sqbl{1}{6};        \node at (1.5,5.5) {\footnotesize $5$};
    \sqbl{1.5}{5};      \node at (2,4.5)   {\footnotesize $6$};
    \sqgr{2}{4};        \node at (2.5,3.5) {\footnotesize $7$};
    \sqgr{0}{6};        \node at (0.5,5.5) {\footnotesize $3$};
\end{tikzpicture}
\caption{}\label{fig:slideexheap4.1}
\end{subfigure}
\begin{subfigure}{0.3\textwidth} \centering \vspace{55pt}
\begin{tikzpicture}
    \sqor{0.5}{8};      \node at (1,7.5)   {\footnotesize $4$};
    \sqblcheck{0}{7};   \node at (0.5,6.5) {\footnotesize $3$};
    \sqbl{1}{6};        \node at (1.5,5.5) {\footnotesize $5$};
    \sqbl{1.5}{5};      \node at (2,4.5)   {\footnotesize $6$};
    \sqgr{2}{4};        \node at (2.5,3.5) {\footnotesize $7$};
    \sqgrcheck{0}{6};   \node at (0.5,5.5) {\footnotesize $3$};
\end{tikzpicture}
\caption{}\label{fig:slideexheap4.2}
\end{subfigure}
\begin{subfigure}{0.3\textwidth} \centering \vspace{80pt}
\begin{tikzpicture}
    \sqbl{0.5}{7};      \node at (1,6.5)   {\footnotesize $4$};
    \sqbl{1}{6};        \node at (1.5,5.5) {\footnotesize $5$};
    \sqbl{1.5}{5};      \node at (2,4.5)   {\footnotesize $6$};
    \sqgr{2}{4};        \node at (2.5,3.5) {\footnotesize $7$};
\end{tikzpicture}
\caption{}\label{fig:slideexheap4.3}
\end{subfigure}
\caption{The heaps for Example~\ref{ex:slide}.}\label{fig:slideex}
\end{figure}
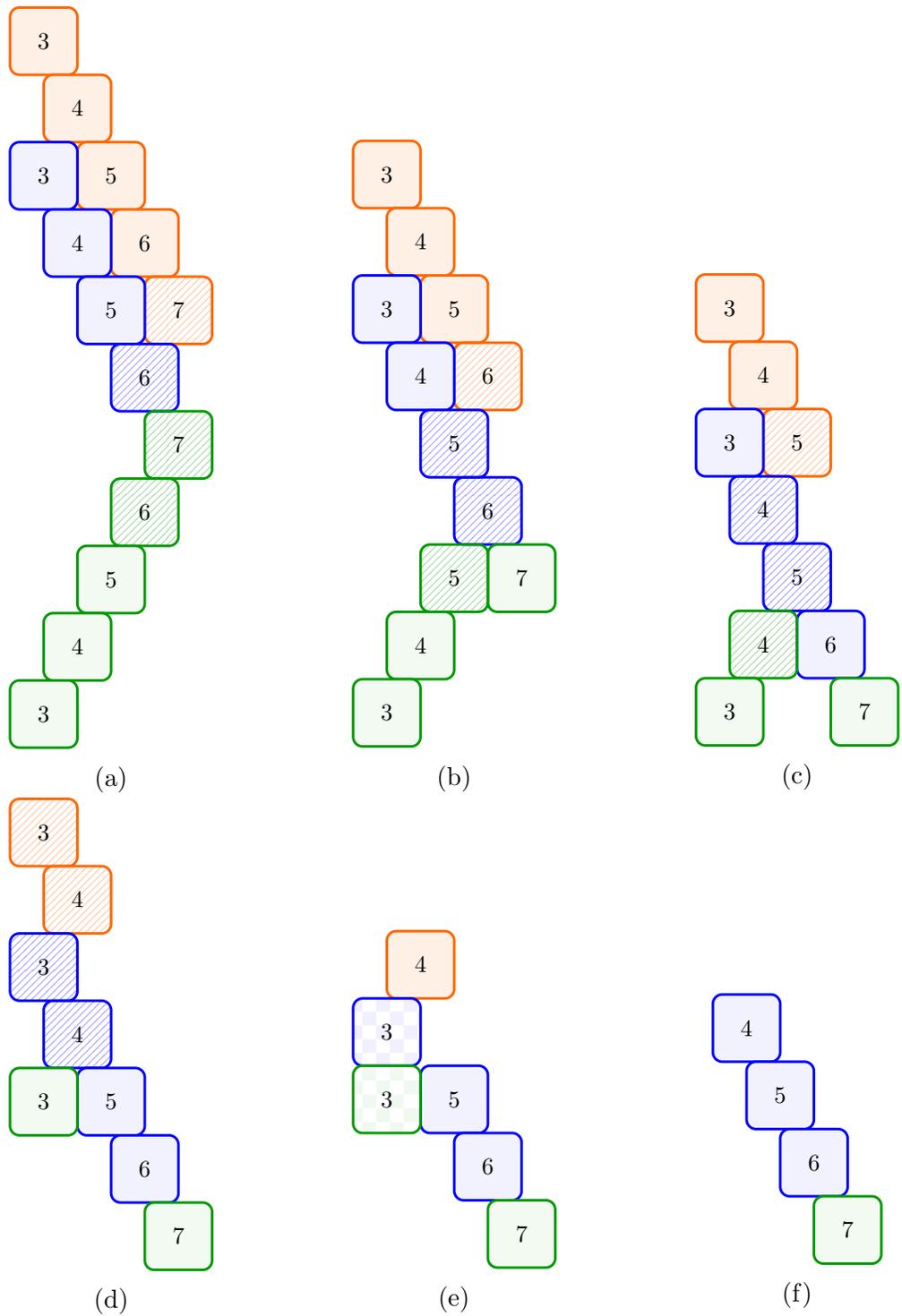
\end{center}
\end{example}

    We can generalize the technique used in the previous example to prove the following lemma, which shows that two CFC elements are conjugate if their cylindrical heaps are slide equivalent by one space.
    
\begin{lemma}\label{lem:slide}
Let $w \in \CFC(A_n)$. Suppose $H(w)$ has a diagonal chunk $C$ where the largest label of $C$ is $k' \leq n-1$.
    Let $y \in \CFC(A_n)$ such that $\hat{H}(y)$ is slide equivalent to $\hat{H}(w)$ by sliding $\hat{C}$, the ring for the chunk $C$, one space to the right. Then $y$ and $w$ are conjugate.
    That is, we can slide the ring in Figure~\ref{fig:slide1} to the ring in Figure~\ref{fig:slide2}, where $k' = k+m$ for some $m$ via conjugation.
\end{lemma}

\begin{center}
\begin{figure}[H] \centering
\begin{subfigure}{0.3\textwidth} \centering
\begin{tikzpicture}[scale=0.9]
\draw[line width=1.5pt,->] (-0.5,2)--(3,2); \draw[line width=1.5pt,->] (-0.5,-2)--(3,-2);
    \sq{0}{2};    \node at (0.5,1.5)   {\scalebox{0.75}{$k$}};
    \sq{0.5}{1};  \node at (1,0.5)     {\scalebox{0.75}{$k+1$}};
                  \node at (1.5,-0.35) {$\ddots$};
    \sq{1.5}{-1}; \node at (2,-1.5)    {\scalebox{0.75}{$k'$}};
\end{tikzpicture}
\caption{}\label{fig:slide1}
\end{subfigure}
\begin{subfigure}{0.3\textwidth} \centering
\begin{tikzpicture}[scale=0.9]
\draw[line width=1.5pt,->] (-0.5,2)--(3,2); \draw[line width=1.5pt,->] (-0.5,-2)--(3,-2);
    \sq{0}{2};    \node at (0.5,1.5)   {\scalebox{0.75}{$k+1$}};
    \sq{0.5}{1};  \node at (1,0.5)     {\scalebox{0.75}{$k+2$}};
                  \node at (1.5,-0.35) {$\ddots$};
    \sq{1.5}{-1}; \node at (2,-1.5)    {\scalebox{0.75}{$k'+1$}};
\end{tikzpicture}
\caption{}\label{fig:slide2}
\end{subfigure}
\caption{The rings for Lemma~\ref{lem:slide}.}\label{fig:slidelemma}
\end{figure}
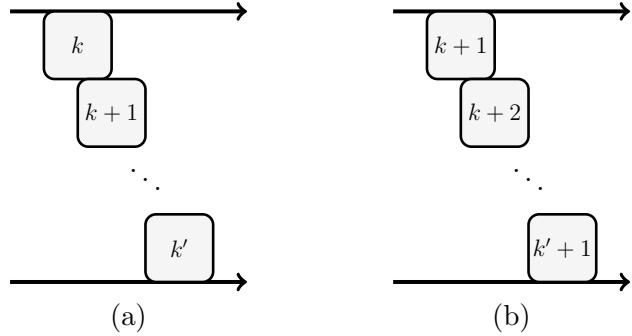
\end{center}

\begin{proof}
Without loss of generality, assume $H(w)$ consists of a diagonal single chunk, where $\w = (k)(k+1) \cdots (k')$ is a reduced expression for $w$.
    Notice that this implies there is no block in $H(w)$ labeled by $k'+1$.

    Let $x$ have reduced expression $\x = (k)(k+1) \cdots (k')(k'+1)$. Then a heap of $w$ conjugated by $x$, namely, $H({\x}{\w}{\x^{-1}})$, is shown in Figure~\ref{fig:slideheap1}, where the \textcolor{orange}{orange} blocks correspond to the heap of $x$, the \textcolor{blue}{blue} blocks correspond to the heap of $w$, and the \textcolor{ggreen}{green} blocks correspond to the heap of $x^{-1}$.
    Then, we have an extra long $(k'+1)(k')$-chain, denoted in the heap by the hatched blocks.
    By Lemma~\ref{lem:stst}, we get the heap in Figure~\ref{fig:slideheap2} since the \textcolor{orange}{orange} $k'+1$ block and the \textcolor{ggreen}{green} $k'$ block are eliminated.
    Now, in the new heap in Figure~\ref{fig:slideheap2}, we have an extra long $(k')(k'-1)$-chain, denoted by the hatched blocks, so applying Lemma~\ref{lem:stst} again, we eliminate the \textcolor{orange}{orange} $k'$ block and the \textcolor{ggreen}{green} $k'-1$ block.
    
    Continuing in this manner, applying $m$ iterations of Lemma~\ref{lem:stst} to $m$ extra long $s_is_j$-chains, we get the heap shown in Figure~\ref{fig:slideheap3.1}.
    After applying Lemma~\ref{lem:stst} to the extra long $(k)(k+1)$-chain, denoted in the heap by the hatched blocks, we get the heap in Figure~\ref{fig:slideheap3.2}.
    Then, canceling the two adjacent $k$ blocks, denoted in the heap by the checked blocks, the result follows, as shown in Figure~\ref{fig:slideheap3.3}.
    Note that the last application of Lemma~\ref{lem:stst} was to an extra long $(k)(k+1)$-chain located at the top of the heap.
\end{proof}

\begin{center}
\begin{figure}[htpb] \centering
\begin{subfigure}{0.3\textwidth} \centering
\begin{tikzpicture}[scale=0.9]
    \sqor{0}{12};       \node at (0.5,11.5){\scalebox{0.8}{$k$}};
    \sqor{0.5}{11};     \node at (1,10.5)  {\scalebox{0.8}{$k+1$}};
    \sqor{1}{10};       \node at (1.5,9.5) {\scalebox{0.8}{$k+2$}};
                        \node at (2.55,8.7){$\ddots$};
    \sqor{2.5}{8};      \node at (3,7.5)   {\scalebox{0.8}{$k'$}};
    \sqorhash{3}{7};    \node at (3.5,6.5) {\scalebox{0.8}{$k'+1$}};
    \sqbl{0}{10};       \node at (0.5,9.5) {\scalebox{0.8}{$k$}};
    \sqbl{0.5}{9};      \node at (1,8.5)   {\scalebox{0.8}{$k+1$}};
                        \node at (1.5,7.6) {$\ddots$};
    \sqbl{2}{7};        \node at (2.5,6.5) {\scalebox{0.8}{$k'-1$}};
    \sqblhash{2.5}{6};  \node at (3,5.5)   {\scalebox{0.8}{$k'$}};
    \sqgrhash{3}{5};    \node at (3.5,4.5) {\scalebox{0.8}{$k'+1$}};
    \sqgrhash{2.5}{4};  \node at (3,3.5)   {\scalebox{0.8}{$k'$}};
    \sqgr{2}{3};        \node at (2.5,2.5) {\scalebox{0.8}{$k'-1$}};
                        \node at (1.9,1.6) {$\iddots$};
    \sqgr{0.5}{1};      \node at (1,0.5)   {\scalebox{0.8}{$k+1$}};
    \sqgr{0}{0};        \node at (0.5,-0.5){\scalebox{0.8}{$k$}};
\end{tikzpicture}
\caption{}\label{fig:slideheap1}
\end{subfigure}
\begin{subfigure}{0.3\textwidth} \centering \vspace{49pt}
\begin{tikzpicture}[scale=0.9]
    \sqor{0}{12};   \node at (0.5,11.5){\scalebox{0.8}{$k$}};
    \sqor{0.5}{11}; \node at (1,10.5)  {\scalebox{0.8}{$k+1$}};
    \sqor{1}{10};   \node at (1.5,9.5) {\scalebox{0.8}{$k+2$}};
                    \node at (2.55,8.7){$\ddots$};
    \sqorhash{2.5}{8};\node at (3,7.5) {\scalebox{0.8}{$k'$}};
    
    \sqbl{0}{10};   \node at (0.5,9.5) {\scalebox{0.8}{$k$}};
    \sqbl{0.5}{9};  \node at (1,8.5)   {\scalebox{0.8}{$k+1$}};
                    \node at (1.5,7.6) {$\ddots$};
    \sqblhash{2}{7};\node at (2.5,6.5) {\scalebox{0.8}{$k'-1$}};
    \sqblhash{2.5}{6};\node at (3,5.5) {\scalebox{0.8}{$k'$}};
    
    \sqgr{3}{5};    \node at (3.5,4.5) {\scalebox{0.8}{$k'+1$}};
    \sqgrhash{2}{5};\node at (2.5,4.5) {\scalebox{0.8}{$k'-1$}};
                    \node at (2,3.5)   {$\iddots$};
    \sqgr{0.5}{3};  \node at (1,2.5)   {\scalebox{0.8}{$k+1$}};
    \sqgr{0}{2};    \node at (0.5,1.5) {\scalebox{0.8}{$k$}};
\end{tikzpicture}
\caption{}\label{fig:slideheap2}
\end{subfigure}
\begin{subfigure}{0.3\textwidth} \centering \vspace{102pt}
\begin{tikzpicture}
    \sqorhash{0}{7};    \node at (0.5,6.5) {\scalebox{0.8}{$k$}};
    \sqorhash{0.5}{6};  \node at (1,5.5)   {\scalebox{0.8}{$k+1$}};
    \sqgr{0}{3};        \node at (0.5,2.5) {\scalebox{0.8}{$k$}};
    \sqblhash{0}{5};    \node at (0.5,4.5) {\scalebox{0.8}{$k$}};
    \sqblhash{0.5}{4};  \node at (1,3.5)   {\scalebox{0.8}{$k+1$}};
    \sqbl{1}{3};        \node at (1.5,2.5) {\scalebox{0.8}{$k+2$}};
                        \node at (2.4,1.7) {$\ddots$};
    \sqbl{2.5}{1};      \node at (3,0.5)   {\scalebox{0.8}{$k'$}};
    \sqgr{3}{0};        \node at (3.5,-0.5){\scalebox{0.8}{$k'+1$}};
\end{tikzpicture}
\caption{}\label{fig:slideheap3.1}
\end{subfigure}
\begin{subfigure}{0.3\textwidth} \centering \vspace{20pt}
\begin{tikzpicture}
    \sqor{0.5}{5};      \node at (1,4.5)   {\scalebox{0.8}{$k+1$}};
    \sqgrcheck{0}{3};   \node at (0.5,2.5) {\scalebox{0.8}{$k$}};
    \sqblcheck{0}{4};   \node at (0.5,3.5) {\scalebox{0.8}{$k$}};
    \sqbl{1}{3};        \node at (1.5,2.5) {\scalebox{0.8}{$k+2$}};
                        \node at (2.4,1.7) {$\ddots$};
    \sqbl{2.5}{1};      \node at (3,0.5)   {\scalebox{0.8}{$k'$}};
    \sqgr{3}{0};        \node at (3.5,-0.5){\scalebox{0.8}{$k'+1$}};
\end{tikzpicture}\caption{}\label{fig:slideheap3.2}
\end{subfigure}
\begin{subfigure}{0.3\textwidth} \centering \vspace{48pt}
\begin{tikzpicture}
    \sqor{0.5}{4};      \node at (1,3.5)   {\scalebox{0.8}{$k+1$}};
    \sqbl{1}{3};        \node at (1.5,2.5) {\scalebox{0.8}{$k+2$}};
                        \node at (2.4,1.7) {$\ddots$};
    \sqbl{2.5}{1};      \node at (3,0.5)   {\scalebox{0.8}{$k'$}};
    \sqgr{3}{0};        \node at (3.5,-0.5){\scalebox{0.8}{$k'+1$}};
\end{tikzpicture}
\caption{}\label{fig:slideheap3.3}
\end{subfigure}
\caption{The heaps for Lemma~\ref{lem:slide}.}\label{fig:slideheaps}
\end{figure}
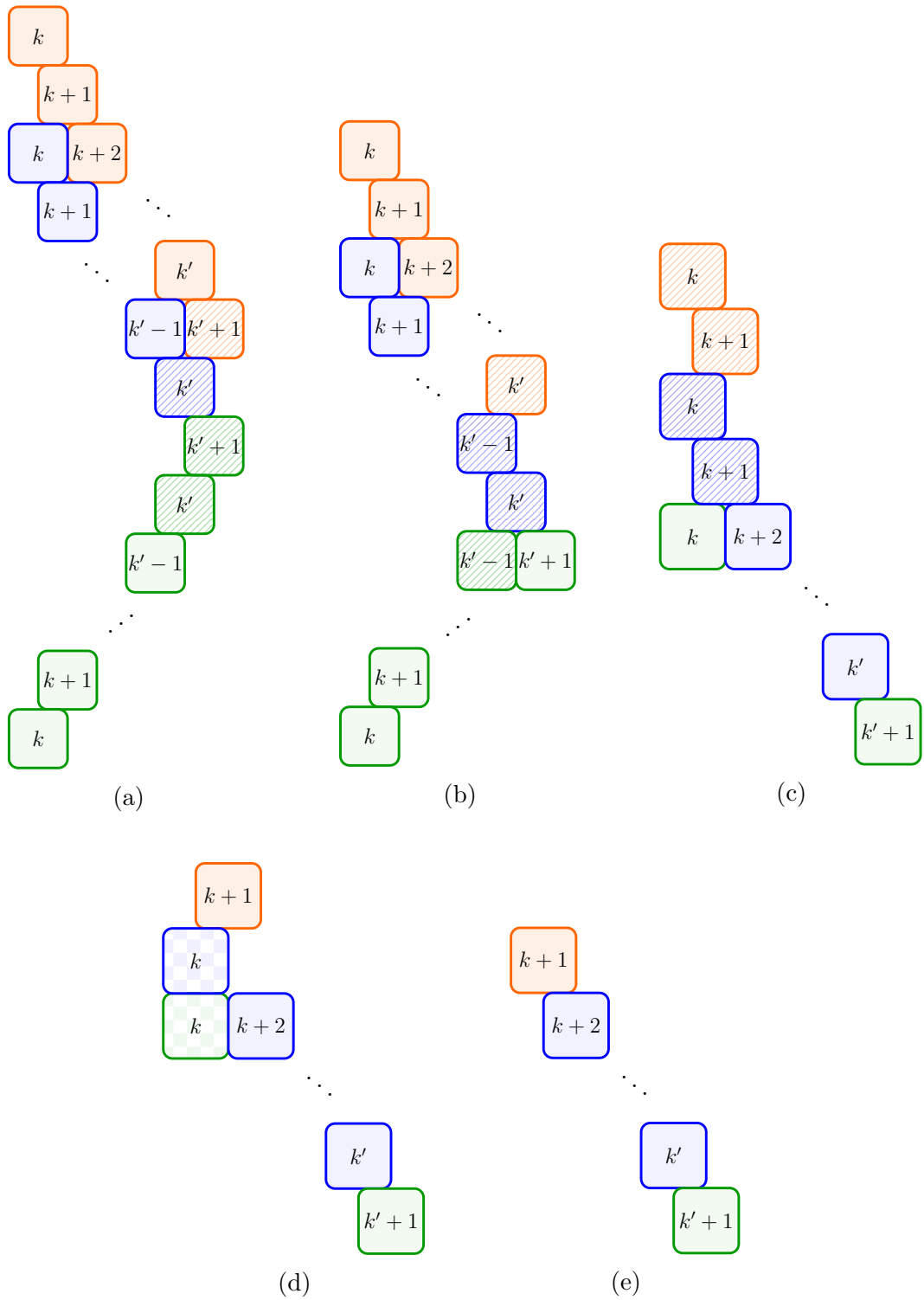
\end{center}

\begin{remark}\label{rem:slide}
Suppose $\hat{H}(w)$ and $\hat{H}(y)$ are slide equivalent by sliding rings of $\hat{H}(w)$ one space to the right to obtain $\hat{H}(y)$.
    Equivalently, we can obtain $\hat{H}(w)$ by sliding rings of $\hat{H}(y)$ one space to the left. In this case, we can obtain $w$ by conjugating $y$ by $x^{-1}$, where $x$ is as given in the proof of Lemma~\ref{lem:slide}.
    It follows that if $w,y \in \CFC(A_n)$ such that $\hat{H}(w)$ is slide equivalent to $\hat{H}(y)$, then $w$ is conjugate to $y$ since we can cyclically shift chunks of $H(w)$ to obtain diagonal chunks and then use Lemma~\ref{lem:slide} as necessary to shift the diagonal chunks. We can then obtain the chunks of $H(y)$ by doing the appropriate cyclic shifts.
\end{remark}

    We now state a lemma that will be useful in the proof of the lemma that follows.

\begin{lemma} \label{lem:boomerang}
Let $w \in W(A_n)$ have reduced expression $\w$. If $H(\w)$ has the heap shown in Figure~\ref{fig:boomerang1} as a convex subheap, then we can replace it with the convex subheap shown in Figure~\ref{fig:boomerang2} to obtain another heap for $w$.
\end{lemma}
\begin{proof} It follows from the relations in $W(A_n)$ that the subword \begin{equation} \label{eq:boom1} (k)(k+1) \cdots (k')(k'+1)(k') \cdots (k+1)(k) \end{equation} can be transformed into \begin{equation} \label{eq:boom2} (k'+1)(k') \cdots (k+1)(k)(k+1) \cdots (k')(k'+1), \end{equation} where $k'=k+m$, by performing a sequence of braid moves.
    In the heap, the subword in (\ref{eq:boom1}) corresponds to the convex subheap shown in Figure~\ref{fig:boomerang1} and the subword in (\ref{eq:boom2}) corresponds to the subheap shown in Figure~\ref{fig:boomerang2}.
    Suppose the convex subheap in Figure~\ref{fig:boomerang1} appears in $H(\w)$. Then applying the braid moves to the blocks corresponding to the subword in (\ref{eq:boom1}) to obtain blocks corresponding to the subword in (\ref{eq:boom2}), we get the convex subheap in Figure~\ref{fig:boomerang2}.
\end{proof}

\begin{center}
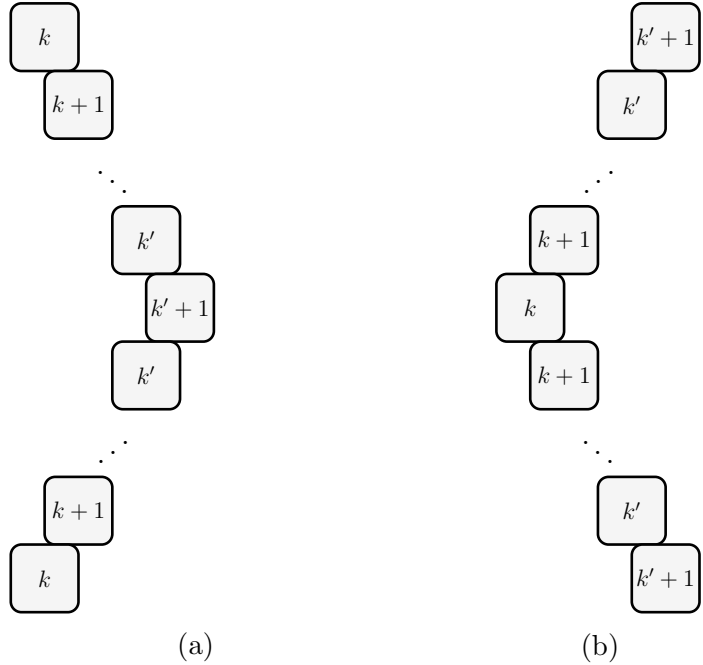
\begin{figure}[htpb] \centering
\begin{tabular}{cc}
\begin{subfigure}{0.3\textwidth}
\begin{tikzpicture}[scale=0.9]
\sq{0}{9};    \node at (0.5,8.5) {\scalebox{0.75}{$k$}};
\sq{0.5}{8};  \node at (1,7.5)   {\scalebox{0.75}{$k+1$}};
              \node at (1.5,6.5) {$\ddots$};
\sq{1.5}{6};  \node at (2,5.5)   {\scalebox{0.75}{$k'$}};
\sq{2}{5};    \node at (2.5,4.5) {\scalebox{0.75}{$k'+1$}};
\sq{1.5}{4};  \node at (2,3.5)   {\scalebox{0.75}{$k'$}};
              \node at (1.5,2.5) {$\iddots$};
\sq{0.5}{2};  \node at (1,1.5)   {\scalebox{0.75}{$k+1$}};
\sq{0}{1};    \node at (0.5,0.5) {\scalebox{0.75}{$k$}};
\end{tikzpicture}
\caption{}\label{fig:boomerang1}
\end{subfigure} &
\begin{subfigure}{0.3\textwidth} \centering
\begin{tikzpicture}[scale=0.9]
\sq{7}{9};    \node at (7.5,8.5) {\scalebox{0.75}{$k'+1$}};
\sq{6.5}{8};  \node at (7,7.5)   {\scalebox{0.75}{$k'$}};
              \node at (6.5,6.5) {$\iddots$};
\sq{5.5}{6};  \node at (6,5.5)   {\scalebox{0.75}{$k+1$}};
\sq{5}{5};    \node at (5.5,4.5) {\scalebox{0.75}{$k$}};
\sq{5.5}{4};  \node at (6,3.5)   {\scalebox{0.75}{$k+1$}};
              \node at (6.5,2.5) {$\ddots$};
\sq{6.5}{2};  \node at (7,1.5)   {\scalebox{0.75}{$k'$}};
\sq{7}{1};    \node at (7.5,0.5) {\scalebox{0.75}{$k'+1$}};
\end{tikzpicture}
\caption{}\label{fig:boomerang2}
\end{subfigure}
\end{tabular}
\caption{The equivalent convex subheaps for Lemma~\ref{lem:boomerang}, where $k'=k+m$.} \label{fig:boomerang}
\end{figure} \end{center}

    Note that any blocks that occur above or below the convex subheap in Figure~\ref{fig:boomerang1} must be shifted vertically as necessary when replacing the first convex subheap with the one in Figure~\ref{fig:boomerang2}.

\begin{example}\label{ex:boomerang} Let $w \in W(A_5)$ have expression $\w = 3524343213$. Then a heap for $\w$ is shown in Figure~\ref{fig:boomerangex1} and contains a convex subheap as in Lemma~\ref{lem:boomerang}, highlighted in \textcolor{ggreen}{green}. Applying Lemma~\ref{lem:boomerang} to $H(\w)$, we get the heap in Figure~\ref{fig:boomerangex2}.
\begin{center} \begin{figure}[H] \centering
\begin{subfigure}{0.4\textwidth} \centering
\begin{tikzpicture}[scale=0.85]
    \sq{0}{8};       \node at (0.5,7.5) {\footnotesize $1$};
    \sq{1}{8};       \node at (1.5,7.5) {\footnotesize $3$};
    \sq{2}{6};       \node at (2.5,5.5) {\footnotesize $5$};
    \sqgr{0.5}{7};   \node at (1,6.5)   {\footnotesize $2$};
    \sqgr{1}{6};     \node at (1.5,5.5) {\footnotesize $3$};
    \sqgr{1.5}{5};   \node at (2,4.5)   {\footnotesize $4$};
    \sqgr{1}{4};     \node at (1.5,3.5) {\footnotesize $3$};
    \sqgr{0.5}{3};   \node at (1,2.5)   {\footnotesize $2$};
    \sq{1}{2};       \node at (1.5,1.5) {\footnotesize $3$};
\end{tikzpicture}
\caption{}\label{fig:boomerangex1}
\end{subfigure}
\begin{subfigure}{0.4\textwidth} \centering
\begin{tikzpicture}[scale=0.85]
    \sq{0}{6};       \node at (0.5,5.5) {\footnotesize $1$};
    \sq{1}{8};       \node at (1.5,7.5) {\footnotesize $3$};
    \sq{2}{8};       \node at (2.5,7.5) {\footnotesize $5$};
    \sqgr{1.5}{7};   \node at (2,6.5)   {\footnotesize $4$};
    \sqgr{1}{6};     \node at (1.5,5.5) {\footnotesize $3$};
    \sqgr{0.5}{5};   \node at (1,4.5)   {\footnotesize $2$};
    \sqgr{1}{4};     \node at (1.5,3.5) {\footnotesize $3$};
    \sqgr{1.5}{3};   \node at (2,2.5)   {\footnotesize $4$};
    \sq{1}{2};       \node at (1.5,1.5) {\footnotesize $3$};
\end{tikzpicture}
\caption{}\label{fig:boomerangex2}
\end{subfigure}
\caption{The heaps for Example~\ref{ex:boomerang}.}\label{}
\end{figure}
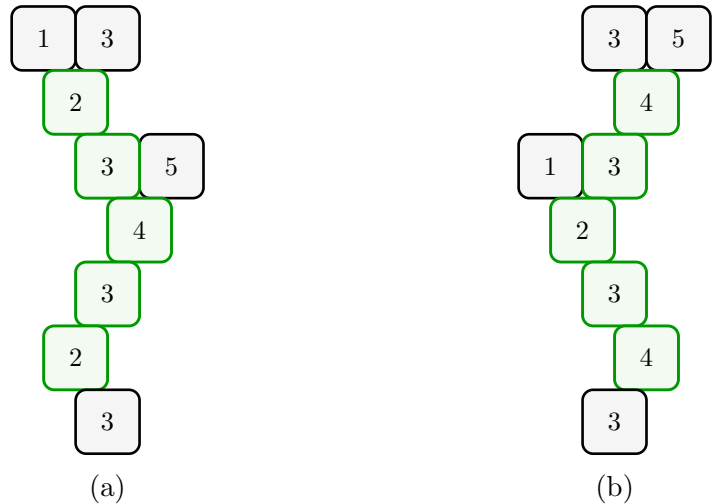 \end{center}
\end{example}

    The following example motivates the proof of Lemma~\ref{lem:swap}.

\begin{example}\label{ex:swap} Let $w,y \in W(A_6)$ have reduced expressions $\w = 12356$ and $\y = 12456$. Then, $w$ and $y$ are both CFC, so there is a unique heap for each.
    Then $\hat{H}(w)$ and $\hat{H}(y)$ are shown in Figures~\ref{fig:permexheap0.1} and~\ref{fig:permexheap0.2}, respectively. Notice that each has two rings, and, moreover, $\hat{H}(w)$ is ring equivalent to $\hat{H}(y)$.

\begin{center}
\begin{figure}[htpb] \centering
\begin{subfigure}{0.3\textwidth} \centering
\begin{tabular}{cc}
\begin{tikzpicture}[scale=0.85]
\draw[line width=1.5pt,->] (-0.5,3)--(3.75,3);
    \sq{0}{3};   \node at (0.5,2.5) {\footnotesize $1$};
    \sq{0.5}{2}; \node at (1,1.5)   {\footnotesize $2$};
    \sq{1}{1};   \node at (1.5,0.5) {\footnotesize $3$};
    \sq{2}{1};   \node at (2.5,0.5) {\footnotesize $5$};
    \sq{2.5}{0}; \node at (3,-0.5)  {\footnotesize $6$};
\draw[line width=1.5pt,->] (-0.5,-1)--(3.75,-1);
\end{tikzpicture}
\end{tabular}
\caption{}\label{fig:permexheap0.1}
\end{subfigure}
\begin{subfigure}{0.3\textwidth} \centering
\begin{tabular}{cc}
\begin{tikzpicture}[scale=0.85]
\draw[line width=1.5pt,->] (-0.5,3)--(3.75,3);
    \sq{0}{3};   \node at (0.5,2.5) {\footnotesize $1$};
    \sq{0.5}{2}; \node at (1,1.5)   {\footnotesize $2$};
    \sq{1.5}{2}; \node at (2,1.5)   {\footnotesize $4$};
    \sq{2}{1};   \node at (2.5,0.5) {\footnotesize $5$};
    \sq{2.5}{0}; \node at (3,-0.5)  {\footnotesize $6$};
\draw[line width=1.5pt,->] (-0.5,-1)--(3.75,-1);
\end{tikzpicture}
\end{tabular}
\caption{}\label{fig:permexheap0.2}
\end{subfigure}
\caption{The cylindrical heaps for Example~\ref{ex:swap}.}
\end{figure}
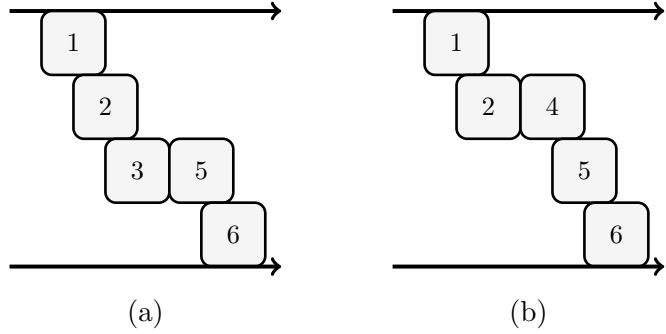 \end{center}
    
    We claim that $w$ and $y$ are conjugate.
    Let $x$ have reduced expression \begin{equation} \x = 345623451234 = (3456)(2345)(1234). \end{equation}
    Then the heap $H({\x}{\w}{\x^{-1}})$ is shown in Figure~\ref{fig:permexheap1.1}. 
    By applying Lemma~\ref{lem:boomerang} to $\textcolor{blue}{123} \textcolor{ggreen}{4321}$, denoted in $H({\x}{\w}{\x^{-1}})$ by hatched blocks, we obtain the heap in Figure~\ref{fig:permexheap1.2}.
    Now, we apply Lemma~\ref{lem:stst} to the extra long $54$-chain in the heap in Figure~\ref{fig:permexheap1.2}, denoted by hatched blocks.
    Then, we have an extra long $43$-chain to which we can apply Lemma~\ref{lem:stst}. Continuing in this manner, we get the heap shown in Figure~\ref{fig:permexheap2}.
    
    We can apply Lemma~\ref{lem:boomerang} to $\textcolor{ggreen}{2345432}$ (hatched) to get the heap shown in Figure~\ref{fig:permexheap3}.
    Now we apply Lemma~\ref{lem:stst} to the extra long $65$-chain (hatched). Then, we have an extra long $54$-chain to which we can apply Lemma~\ref{lem:stst}.
    Continuing, we apply Lemma~\ref{lem:stst} to the extra long $43$-chain and the extra long $32$-chain and we get the heap shown in Figure~\ref{fig:permexheap4}.
    Finally, we apply Lemma~\ref{lem:boomerang} to $\textcolor{ggreen}{3456543}$ (hatched) and get the heap shown in Figure~\ref{fig:permexheap5}.
    We can cancel the adjacent 6 blocks (checked), followed by adjacent 5 blocks, adjacent 4 blocks, and adjacent 3 blocks. After all the cancellation, we get the heap shown in Figure~\ref{fig:permexheap6}, which yields the desired result.
\end{example}

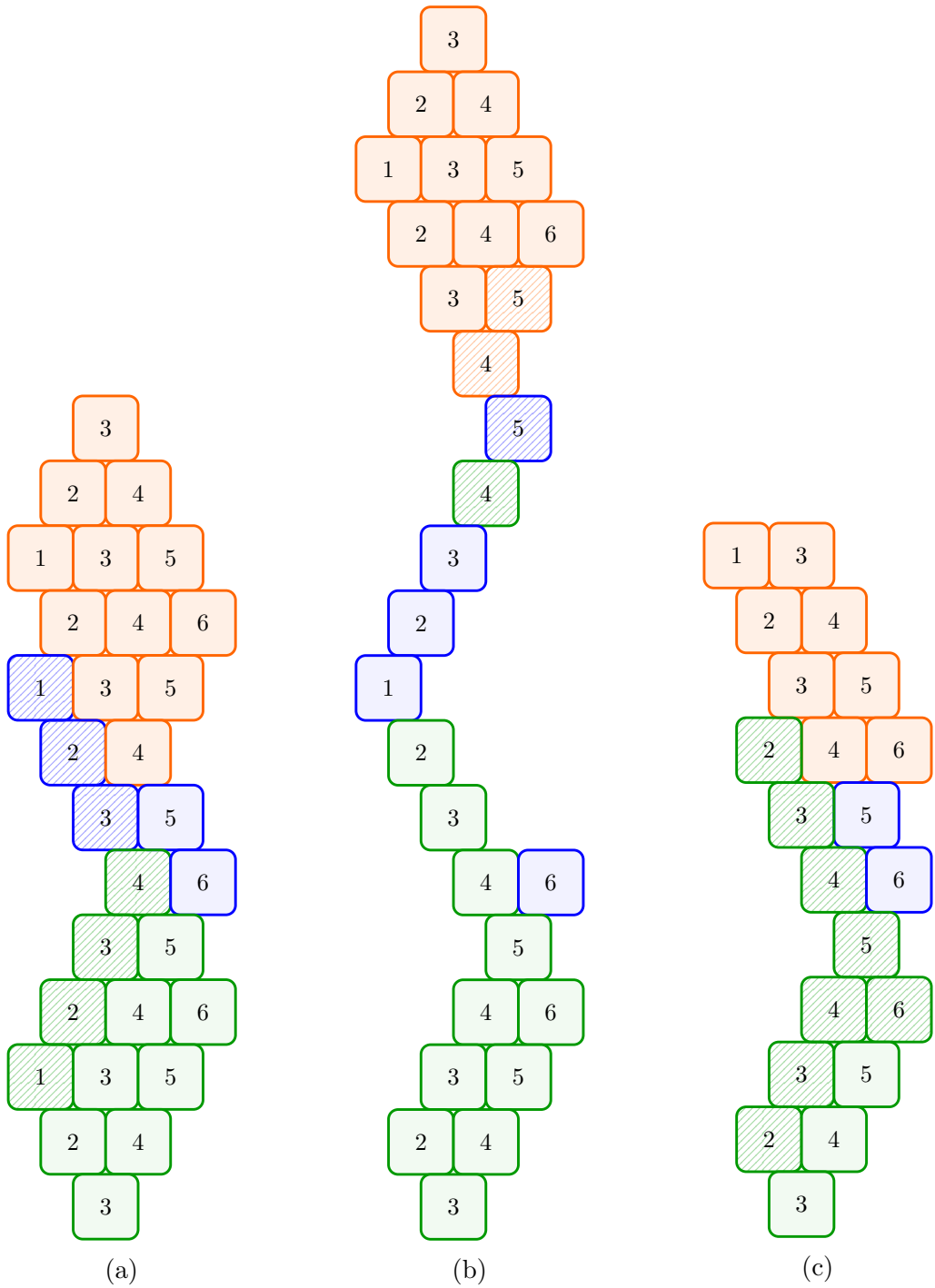
\begin{figure}[htpb] \centering
\begin{subfigure}{0.3\textwidth} \centering \vspace{162pt}
\begin{tikzpicture}[scale=0.95]
    \sqor{1}{12};    \node at (1.5,11.5) {\footnotesize $3$};
    \sqor{0.5}{11};  \node at (1,10.5)   {\footnotesize $2$};
    \sqor{1.5}{11};  \node at (2,10.5)   {\footnotesize $4$};
    \sqor{0}{10};    \node at (0.5,9.5)  {\footnotesize $1$};
    \sqor{1}{10};    \node at (1.5,9.5)  {\footnotesize $3$};
    \sqor{2}{10};    \node at (2.5,9.5)  {\footnotesize $5$};
    \sqor{0.5}{9};   \node at (1,8.5)    {\footnotesize $2$};
    \sqor{1.5}{9};   \node at (2,8.5)    {\footnotesize $4$};
    \sqor{2.5}{9};   \node at (3,8.5)    {\footnotesize $6$};
    \sqblhash{0}{8}; \node at (0.5,7.5)  {\footnotesize $1$};
    \sqor{1}{8};     \node at (1.5,7.5)  {\footnotesize $3$};
    \sqor{2}{8};     \node at (2.5,7.5)  {\footnotesize $5$};
    \sqblhash{0.5}{7};\node at (1,6.5)   {\footnotesize $2$};
    \sqor{1.5}{7};   \node at (2,6.5)    {\footnotesize $4$};
    \sqblhash{1}{6}; \node at (1.5,5.5)  {\footnotesize $3$};
    \sqbl{2}{6};     \node at (2.5,5.5)  {\footnotesize $5$};
    \sqgrhash{1.5}{5};\node at (2,4.5)   {\footnotesize $4$};
    \sqbl{2.5}{5};   \node at (3,4.5)    {\footnotesize $6$};
    \sqgrhash{1}{4}; \node at (1.5,3.5)  {\footnotesize $3$};
    \sqgr{2}{4};     \node at (2.5,3.5)  {\footnotesize $5$};
    \sqgrhash{0.5}{3};\node at (1,2.5)   {\footnotesize $2$};
    \sqgr{1.5}{3};   \node at (2,2.5)    {\footnotesize $4$};
    \sqgr{2.5}{3};   \node at (3,2.5)    {\footnotesize $6$};
    \sqgrhash{0}{2}; \node at (0.5,1.5)  {\footnotesize $1$};
    \sqgr{1}{2};     \node at (1.5,1.5)  {\footnotesize $3$};
    \sqgr{2}{2};     \node at (2.5,1.5)  {\footnotesize $5$};
    \sqgr{0.5}{1};   \node at (1,0.5)    {\footnotesize $2$};
    \sqgr{1.5}{1};   \node at (2,0.5)    {\footnotesize $4$};
    \sqgr{1}{0};     \node at (1.5,-0.5) {\footnotesize $3$};
\end{tikzpicture}
\caption{}\label{fig:permexheap1.1}
\end{subfigure}
\begin{subfigure}{0.3\textwidth} \centering
\begin{tikzpicture}[scale=0.95]
    \sqor{1}{12};    \node at (1.5,11.5) {\footnotesize $3$};
    \sqor{0.5}{11};  \node at (1,10.5)   {\footnotesize $2$};
    \sqor{1.5}{11};  \node at (2,10.5)   {\footnotesize $4$};
    \sqor{0}{10};    \node at (0.5,9.5)  {\footnotesize $1$};
    \sqor{1}{10};    \node at (1.5,9.5)  {\footnotesize $3$};
    \sqor{2}{10};    \node at (2.5,9.5)  {\footnotesize $5$};
    \sqor{0.5}{9};   \node at (1,8.5)    {\footnotesize $2$};
    \sqor{1.5}{9};   \node at (2,8.5)    {\footnotesize $4$};
    \sqor{2.5}{9};   \node at (3,8.5)    {\footnotesize $6$};
    \sqor{1}{8};     \node at (1.5,7.5)  {\footnotesize $3$};
    \sqorhash{2}{8}; \node at (2.5,7.5)  {\footnotesize $5$};
    \sqorhash{1.5}{7};\node at (2,6.5)   {\footnotesize $4$};
    \sqblhash{2}{6}; \node at (2.5,5.5)  {\footnotesize $5$};
    \sqgrhash{1.5}{5};\node at (2,4.5)   {\footnotesize $4$};
    \sqbl{1}{4};     \node at (1.5,3.5)  {\footnotesize $3$};
    \sqbl{0.5}{3};   \node at (1,2.5)    {\footnotesize $2$};
    \sqbl{0}{2};     \node at (0.5,1.5)  {\footnotesize $1$};
    \sqgr{0.5}{1};   \node at (1,0.5)    {\footnotesize $2$};
    \sqgr{1}{0};     \node at (1.5,-0.5) {\footnotesize $3$};
    \sqgr{1.5}{-1};  \node at (2,-1.5)   {\footnotesize $4$};
    \sqbl{2.5}{-1};  \node at (3,-1.5)   {\footnotesize $6$};
    \sqgr{2}{-2};    \node at (2.5,-2.5) {\footnotesize $5$};
    \sqgr{1.5}{-3};  \node at (2,-3.5)   {\footnotesize $4$};
    \sqgr{2.5}{-3};  \node at (3,-3.5)   {\footnotesize $6$};
    \sqgr{1}{-4};    \node at (1.5,-4.5) {\footnotesize $3$};
    \sqgr{2}{-4};    \node at (2.5,-4.5) {\footnotesize $5$};
    \sqgr{0.5}{-5};  \node at (1,-5.5)   {\footnotesize $2$};
    \sqgr{1.5}{-5};  \node at (2,-5.5)   {\footnotesize $4$};
    \sqgr{1}{-6};    \node at (1.5,-6.5) {\footnotesize $3$};
\end{tikzpicture}
\caption{}\label{fig:permexheap1.2}
\end{subfigure}
\begin{subfigure}{0.3\textwidth} \centering \vspace{214pt}
\begin{tikzpicture}[scale=0.95]
    \sqor{1}{12};    \node at (1.5,11.5) {\footnotesize $3$};
    \sqor{1.5}{11};  \node at (2,10.5)   {\footnotesize $4$};
    \sqor{0}{12};    \node at (0.5,11.5) {\footnotesize $1$};
    \sqor{2}{10};    \node at (2.5,9.5)  {\footnotesize $5$};
    \sqor{0.5}{11};  \node at (1,10.5)   {\footnotesize $2$};
    \sqor{2.5}{9};   \node at (3,8.5)    {\footnotesize $6$};
    \sqor{1}{10};    \node at (1.5,9.5)  {\footnotesize $3$};
    \sqor{1.5}{9};   \node at (2,8.5)    {\footnotesize $4$};
    \sqbl{2}{8};     \node at (2.5,7.5)  {\footnotesize $5$};
    \sqgrhash{0.5}{9};   \node at (1,8.5)    {\footnotesize $2$};
    \sqgrhash{1}{8};     \node at (1.5,7.5)  {\footnotesize $3$};
    \sqgrhash{1.5}{7};   \node at (2,6.5)    {\footnotesize $4$};
    \sqbl{2.5}{7};       \node at (3,6.5)    {\footnotesize $6$};
    \sqgrhash{2}{6};     \node at (2.5,5.5)  {\footnotesize $5$};
    \sqgrhash{1.5}{5};   \node at (2,4.5)    {\footnotesize $4$};
    \sqgrhash{2.5}{5};   \node at (3,4.5)    {\footnotesize $6$};
    \sqgrhash{1}{4};     \node at (1.5,3.5)  {\footnotesize $3$};
    \sqgr{2}{4};         \node at (2.5,3.5)  {\footnotesize $5$};
    \sqgrhash{0.5}{3};   \node at (1,2.5)    {\footnotesize $2$};
    \sqgr{1.5}{3};       \node at (2,2.5)    {\footnotesize $4$};
    \sqgr{1}{2};         \node at (1.5,1.5)  {\footnotesize $3$};
\end{tikzpicture}
\caption{}\label{fig:permexheap2}
\end{subfigure}
\caption{The heaps for Example~\ref{ex:swap}.}\label{}
\end{figure}
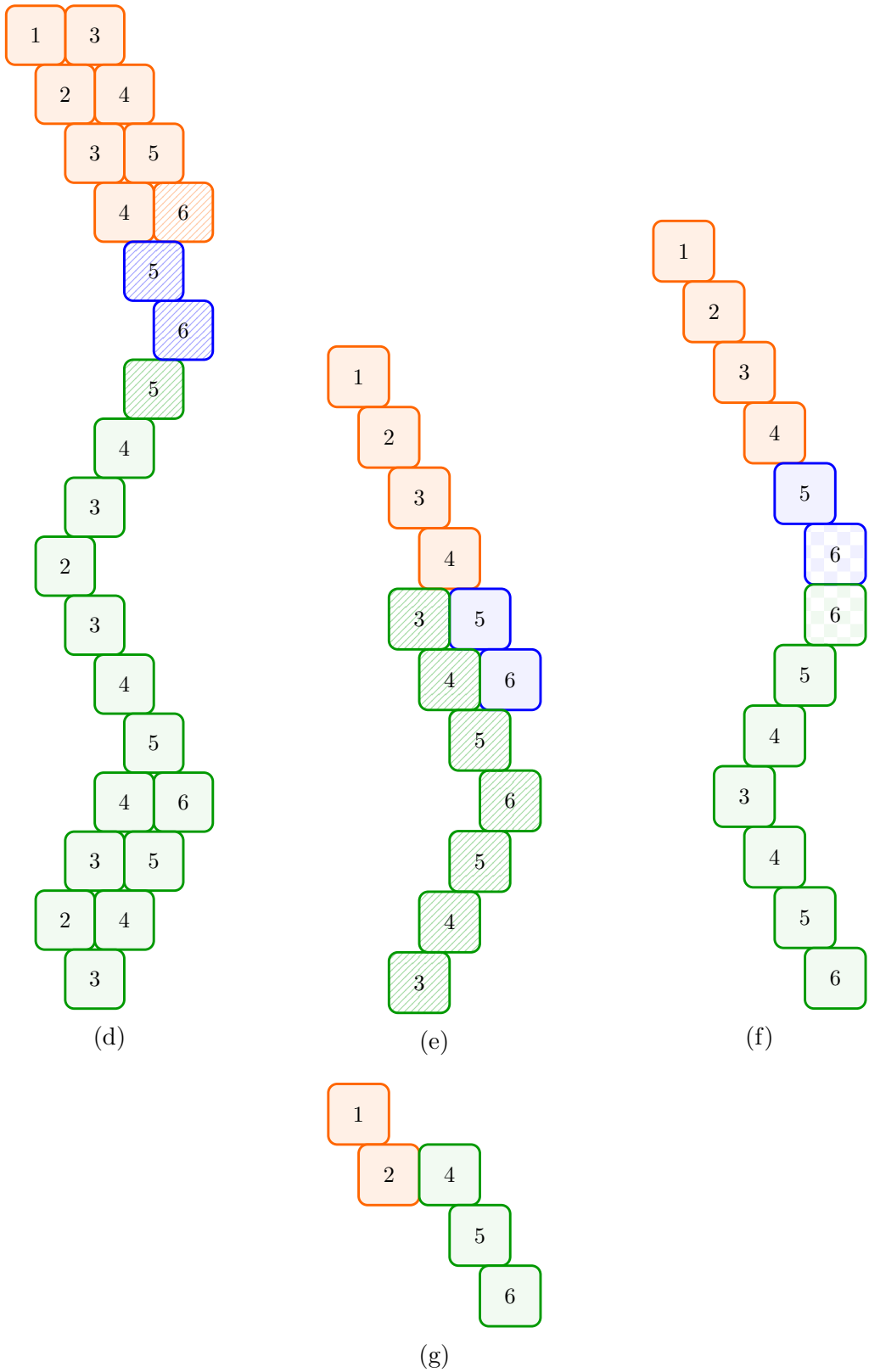
\begin{figure}[htpb] \centering
\ContinuedFloat
\begin{subfigure}{0.3\textwidth} \centering \vspace{-46pt}
\begin{tikzpicture}[scale=0.925]
    \sqor{1}{12};     \node at (1.5,11.5) {\footnotesize $3$};
    \sqor{1.5}{11};   \node at (2,10.5)   {\footnotesize $4$};
    \sqor{0}{12};     \node at (0.5,11.5) {\footnotesize $1$};
    \sqor{2}{10};     \node at (2.5,9.5)  {\footnotesize $5$};
    \sqor{0.5}{11};   \node at (1,10.5)   {\footnotesize $2$};
    \sqorhash{2.5}{9};\node at (3,8.5)    {\footnotesize $6$};
    \sqor{1}{10};     \node at (1.5,9.5)  {\footnotesize $3$};
    \sqor{1.5}{9};    \node at (2,8.5)    {\footnotesize $4$};
    \sqblhash{2}{8};  \node at (2.5,7.5)  {\footnotesize $5$};
    
    \sqgr{0.5}{3};    \node at (1,2.5)    {\footnotesize $2$};
    \sqgr{1}{4};      \node at (1.5,3.5)  {\footnotesize $3$};
    \sqgr{1.5}{5};    \node at (2,4.5)    {\footnotesize $4$};
    \sqgrhash{2}{6};  \node at (2.5,5.5)  {\footnotesize $5$};
    \sqblhash{2.5}{7};\node at (3,6.5)    {\footnotesize $6$};
    \sqgr{1}{2};      \node at (1.5,1.5)  {\footnotesize $3$};
    \sqgr{1.5}{1};    \node at (2,0.5)    {\footnotesize $4$};
    \sqgr{2}{0};      \node at (2.5,-0.5) {\footnotesize $5$};
    \sqgr{2.5}{-1};   \node at (3,-1.5)   {\footnotesize $6$};
    \sqgr{1.5}{-1};   \node at (2,-1.5)   {\footnotesize $4$};
    \sqgr{1}{-2};     \node at (1.5,-2.5) {\footnotesize $3$};
    \sqgr{2}{-2};     \node at (2.5,-2.5) {\footnotesize $5$};
    \sqgr{0.5}{-3};   \node at (1,-3.5)   {\footnotesize $2$};
    \sqgr{1.5}{-3};   \node at (2,-3.5)   {\footnotesize $4$};
    \sqgr{1}{-4};     \node at (1.5,-4.5) {\footnotesize $3$};
\end{tikzpicture}
\caption{}\label{fig:permexheap3}
\end{subfigure}
\begin{subfigure}{0.3\textwidth} \centering \vspace{108pt}
\begin{tikzpicture}[scale=0.95]
    \sqor{0}{10};       \node at (0.5,9.5)  {\footnotesize $1$};
    \sqor{0.5}{9};      \node at (1,8.5)    {\footnotesize $2$};
    \sqor{1}{8};        \node at (1.5,7.5)  {\footnotesize $3$};
    \sqor{1.5}{7};      \node at (2,6.5)    {\footnotesize $4$};
    \sqbl{2}{6};        \node at (2.5,5.5)  {\footnotesize $5$};
    \sqbl{2.5}{5};      \node at (3,4.5)    {\footnotesize $6$};
    \sqgrhash{1}{6};    \node at (1.5,5.5)  {\footnotesize $3$};
    \sqgrhash{1.5}{5};  \node at (2,4.5)    {\footnotesize $4$};
    \sqgrhash{2}{4};    \node at (2.5,3.5)  {\footnotesize $5$};
    \sqgrhash{2.5}{3};  \node at (3,2.5)    {\footnotesize $6$};
    \sqgrhash{2}{2};    \node at (2.5,1.5)  {\footnotesize $5$};
    \sqgrhash{1.5}{1};  \node at (2,0.5)    {\footnotesize $4$};
    \sqgrhash{1}{0};    \node at (1.5,-0.5) {\footnotesize $3$};
\end{tikzpicture}
\caption{}\label{fig:permexheap4}
\end{subfigure}
\begin{subfigure}{0.3\textwidth} \centering \vspace{50pt}
\begin{tikzpicture}[scale=0.95]
    \sqor{0}{10};       \node at (0.5,9.5)  {\footnotesize $1$};
    \sqor{0.5}{9};      \node at (1,8.5)    {\footnotesize $2$};
    \sqor{1}{8};        \node at (1.5,7.5)  {\footnotesize $3$};
    \sqor{1.5}{7};      \node at (2,6.5)    {\footnotesize $4$};
    \sqbl{2}{6};        \node at (2.5,5.5)  {\footnotesize $5$};
    \sqblcheck{2.5}{5}; \node at (3,4.5)    {\footnotesize $6$};
    \sqgr{1}{1};        \node at (1.5,0.5)  {\footnotesize $3$};
    \sqgr{1.5}{2};      \node at (2,1.5)    {\footnotesize $4$};
    \sqgr{2}{3};        \node at (2.5,2.5)  {\footnotesize $5$};
    \sqgrcheck{2.5}{4}; \node at (3,3.5)    {\footnotesize $6$};
    \sqgr{2}{-1};       \node at (2.5,-1.5) {\footnotesize $5$};
    \sqgr{1.5}{0};      \node at (2,-0.5)   {\footnotesize $4$};
    \sqgr{2.5}{-2};     \node at (3,-2.5)   {\footnotesize $6$};
\end{tikzpicture}
\caption{}\label{fig:permexheap5}
\end{subfigure}
\begin{subfigure}{\textwidth} \centering \vspace{10pt}
\begin{tikzpicture}[scale=0.95]
    \sqor{0}{6};       \node at (0.5,5.5)  {\footnotesize $1$};
    \sqor{0.5}{5};     \node at (1,4.5)    {\footnotesize $2$};
    \sqgr{1.5}{5};     \node at (2,4.5)    {\footnotesize $4$};
    \sqgr{2}{4};       \node at (2.5,3.5)  {\footnotesize $5$};
    \sqgr{2.5}{3};     \node at (3,2.5)    {\footnotesize $6$};
\end{tikzpicture}
\caption{}\label{fig:permexheap6}
\end{subfigure}
\caption{The heaps for Example~\ref{ex:swap} (continued).}\label{fig:permexheaps}
\end{figure}

    We can generalize the technique used in the previous example to prove the following lemma, which allows us to permute two adjacent diagonal chunks.

\begin{lemma}\label{lem:swap}
Let $w,y \in \CFC(A_n)$ such that $H(w)$ and $H(y)$ are simple, each consisting of two chunks as in Figure~\ref{fig:swap1}, where the chunk starting at 1 has size $k$ and the adjacent chunk has size $m$ and $k'=k+m$. Then $w$ and $y$ are conjugate.
\end{lemma}

\begin{center} \begin{figure}[H] \centering
\begin{subfigure}{0.45\textwidth} \centering
\begin{tabular}{@{}c@{}c}
\begin{tikzpicture}
    \sq{0}{3};    \node at (0.5,2.5)   {\scalebox{0.8}{$1$}};
    \sq{0.5}{2};  \node at (1,1.5)     {\scalebox{0.8}{$2$}};
                  \node at (1.5,0.7)   {$\ddots$};
    \sq{1.5}{0};  \node at (2,-0.5)    {\scalebox{0.8}{$k$}};
\end{tikzpicture} &
\begin{tikzpicture}
    \sq{2.5}{3};  \node at (3,2.5)     {\scalebox{0.8}{$k+2$}};
    \sq{3}{2};    \node at (3.5,1.5)   {\scalebox{0.8}{$k+3$}};
                  \node at (4,0.7)     {$\ddots$};
    \sq{4}{0};    \node at (4.5,-0.5)  {\scalebox{0.8}{$k'+1$}};
\end{tikzpicture}
\end{tabular}
\caption{}\label{fig:swap1}
\end{subfigure}
\begin{subfigure}{0.45\textwidth} \centering
\begin{tabular}{@{}c@{}c} \centering
\begin{tikzpicture}
    \sq{0}{3};    \node at (0.5,2.5) {\scalebox{0.8}{$1$}};
    \sq{0.5}{2};  \node at (1,1.5)   {\scalebox{0.8}{$2$}};
                  \node at (1.5,0.7) {$\ddots$};
    \sq{1.5}{0};  \node at (2,-0.5)  {\scalebox{0.8}{$m$}};
\end{tikzpicture} &
\begin{tikzpicture}
    \sq{2.5}{3};  \node at (3,2.5)   {\scalebox{0.8}{$m+2$}};
    \sq{3}{2};    \node at (3.5,1.5) {\scalebox{0.8}{$m+3$}};
                  \node at (4,0.7)   {$\ddots$};
    \sq{4}{0};    \node at (4.5,-0.5){\scalebox{0.8}{$k'+1$}};
\end{tikzpicture} \end{tabular}
\caption{}\label{fig:swap2}
\end{subfigure}
\caption{The heaps for Lemma~\ref{lem:swap}.}\label{}
\end{figure}
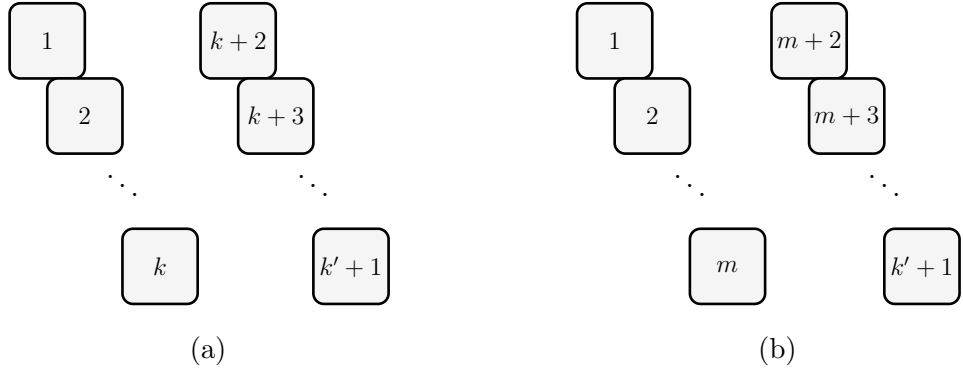 \end{center}

\begin{proof} We first consider the case $k>m$.
    Let $x \in W(A_n)$ have a reduced expression $\x$ that consists of $m+1$ ascending subwords of $k+1$ generators each, starting with $(m+1)(m+2) \cdots (k'+1)$ and being such that the sequence of first generators of each subword descends to 1 (as in Example~\ref{ex:swap}).
    That is, $$\x = \underbrace{(m+1)(m+2) \cdots (k'+1)}_{1} \underbrace{(m)(m+1) \cdots (k')}_{2} \cdots \underbrace{(2)(3) \cdots (k)}_{m} \underbrace{(1)(2) \cdots (k+1)}_{m+1}.$$
    Conjugate $w$ by $x$, and consider the heap $H({\x}{\w}{\x^{-1}})$, shown in Figure~\ref{fig:permheap1}, where $k'=k+m$ and \textcolor{orange}{orange} blocks correspond to the heap of $\x$, \textcolor{blue}{blue} blocks correspond to the heap of $w$, and \textcolor{ggreen}{green} blocks correspond to the heap of $\x^{-1}$.
    Now, to the heap in Figure~\ref{fig:permheap1}, we apply Lemma~\ref{lem:boomerang} to $$\textcolor{blue}{(1)(2) \cdots (k)} \textcolor{ggreen}{(k+1)(k) \cdots (2)(1)},$$ denoted by hatched blocks, to get the heap shown in Figure~\ref{fig:permheap2}.
    
    Then, as in the proof of Lemma~\ref{lem:slide}, apply Lemma~\ref{lem:stst} to the extra long $(k+2)(k+1)$-chain, denoted in the heap in Figure~\ref{fig:permheap2} as hatched blocks.
    This creates an extra long $(k+1)(k)$-chain.
    Continuing this process $k$ times, we get the heap shown in Figure~\ref{fig:permheap3} since the \textcolor{blue}{blue} $(1)(2)\cdots(k)(k+2)$, \textcolor{ggreen}{green} $(k+1)$, and \textcolor{orange}{orange} $(2)(3) \cdots (k+2)$ blocks in Figure~\ref{fig:permheap2} cancel via the iterations of Lemma~\ref{lem:stst} with extra long $s_is_j$-chains.
    
    Then, after $m$ steps as above, we get the heap shown in Figure~\ref{fig:permheap4}.
    Finally, we apply Lemma~\ref{lem:boomerang} to the blocks labeled $$\textcolor{ggreen}{(m+1)(m+2) \cdots (k')(k'+1)(k') \cdots (m+2)(m+1)},$$ denoted in the heap in Figure~\ref{fig:permheap4} by hatched blocks, to get the heap shown in Figure~\ref{fig:permheap5}.
    There are adjacent $k'+1$ blocks that cancel, denoted in the heap by checked blocks, followed by adjacent $k'$ blocks, and so on. After the cancellation, we get the heap shown in Figure~\ref{fig:permheap6} and the result follows.
    
    In the case where $k<m$, conjugate $w$ by $x^{-1}$, as given above, to obtain $y$.
\end{proof}

\begin{remark}\label{rem:permute} If the heap of $w$ is not simple, we can perform a sequence of cyclic shifts and applications of Lemma~\ref{lem:slide} and Remark~\ref{rem:slide} to obtain a simple heap, and, after applying Lemma~\ref{lem:swap}, we can reverse the cyclic shifts and applications of Lemma~\ref{lem:slide} and Remark~\ref{rem:slide}.
\end{remark}

\begin{center} \begin{figure}[htpb] \centering
\begin{tabular}{cc}
\begin{subfigure}{0.45\textwidth} \centering \vspace{124pt}
\begin{tikzpicture}[scale=0.82]
    \sqblhash{0}{9};   \node at (0.5,8.5)   {\scalebox{0.75}{$1$}};
    \sqblhash{0.5}{8}; \node at (1,7.5)     {\scalebox{0.75}{$2$}};
                       \node at (1.8,6.6)   {$\ddots$};
    \sqblhash{2}{6};   \node at (2.5,5.5)   {\scalebox{0.75}{$k$}};
    \sqbl{3}{6};       \node at (3.5,5.5)   {\scalebox{0.75}{$k+2$}};
    \sqbl{3.5}{5};     \node at (4,4.5)     {\scalebox{0.75}{$k+3$}};
                       \node at (4.8,3.65)  {$\ddots$};
    \sqbl{5}{3};       \node at (5.5,2.5)   {\scalebox{0.75}{$k'$}};
    \sqbl{5.5}{2};     \node at (6,1.5)     {\scalebox{0.75}{$k'+1$}};

    \sqor{2.5}{7};     \node at (3,6.5)     {\scalebox{0.75}{$k+1$}};
                       \node at (2.25,7.6)  {$\ddots$};
    \sqor{1}{9};       \node at (1.5,8.5)   {\scalebox{0.75}{$3$}};
    \sqor{0.5}{10};    \node at (1,9.5)     {\scalebox{0.75}{$2$}};
    \sqor{0}{11};      \node at (0.5,10.5)  {\scalebox{0.75}{$1$}};
    \sqor{3}{8};       \node at (3.5,7.5)   {\scalebox{0.75}{$k+2$}};
                       \node at (2.75,8.6)  {$\ddots$};
    \sqor{1.5}{10};    \node at (2,9.5)     {\scalebox{0.75}{$4$}};
    \sqor{1}{11};      \node at (1.5,10.5)  {\scalebox{0.75}{$3$}};
    \sqor{0.5}{12};    \node at (1,11.5)    {\scalebox{0.75}{$2$}};
                       \node at (2,12.5)    {$\iddots$};
                       \node at (4.5,8.5)   {$\iddots$};
                       \node at (3,11.5)    {$\iddots$};
                       \node at (5.1,11.6)  {$\ddots$};
                       \node at (4,10.5)    {$\cdots$};
                       \node at (4,9.5)     {$\cdots$};
    \sqor{3}{14};      \node at (3.5,13.5)  {\scalebox{0.75}{$m+1$}};
    \sqor{3.5}{13};    \node at (4,12.5)    {\scalebox{0.75}{$m+2$}};
    \sqor{5.5}{11};    \node at (6,10.5)    {\scalebox{0.75}{$k'+1$}};
    \sqor{5}{10};      \node at (5.5,9.5)   {\scalebox{0.75}{$k'$}};
    
    \sqgrhash{2.5}{5}; \node at (3,4.5)     {\scalebox{0.75}{$k+1$}};
    \sqgrhash{2}{4};   \node at (2.5,3.5)   {\scalebox{0.75}{$k$}};
    \sqgr{3}{4};       \node at (3.5,3.5)   {\scalebox{0.75}{$k+2$}};
                   
    \sqgr{4.5}{2};     \node at (5,1.5)     {\scalebox{0.75}{$k'-1$}};
    \sqgr{5}{1};       \node at (5.5,0.5)   {\scalebox{0.75}{$k'$}};
    \sqgr{5.5}{0};     \node at (6,-0.5)    {\scalebox{0.75}{$k'+1$}};
                       \node at (5,-1.5)    {$\iddots$};
    \sqgrhash{0.5}{2}; \node at (1,1.5)     {\scalebox{0.75}{$2$}};
    \sqgrhash{0}{1};   \node at (0.5,0.5)   {\scalebox{0.75}{$1$}};
    \sqgr{1}{1};       \node at (1.5,0.5)   {\scalebox{0.75}{$3$}};
    \sqgr{0.5}{0};     \node at (1,-0.5)    {\scalebox{0.75}{$2$}};
    \sqgr{1.5}{2};     \node at (2,1.5)     {\scalebox{0.75}{$4$}};
                       \node at (2,-1.5)    {$\ddots$};
    \sqgr{3}{-3};      \node at (3.5,-3.5)  {\scalebox{0.75}{$m+1$}};
    \sqgr{2.5}{-2};    \node at (3,-2.5)    {\scalebox{0.75}{$m$}};
    \sqgr{3.5}{-2};    \node at (4,-2.5)    {\scalebox{0.75}{$m+2$}};
                       \node at (1.8,2.5)   {$\iddots$};
                       \node at (2.8,2.5)   {$\iddots$};
                       \node at (3.5,1.5)   {$\cdots$};
                       \node at (3.5,0.5)   {$\cdots$};
                       \node at (3.5,-0.5)  {$\cdots$};
                       \node at (4.3,2.5)   {$\ddots$};
\end{tikzpicture}
\caption{}\label{fig:permheap1}
\end{subfigure} &
\begin{subfigure}{0.45\textwidth} \centering \vspace{-45pt}
\begin{tikzpicture}[scale=0.8]
    \sqbl{0}{5};        \node at (0.5,4.5)   {\scalebox{0.7}{$1$}};
    \sqbl{0.5}{6};      \node at (1,5.5)     {\scalebox{0.7}{$2$}};
    \sqbl{2}{8};        \node at (2.5,7.5)   {\scalebox{0.7}{$k$}};
    \sqblhash{3}{10};   \node at (3.5,9.5)   {\scalebox{0.7}{$k+2$}};
    \sqbl{3.5}{1};      \node at (4,0.5)     {\scalebox{0.7}{$k+3$}};
    \sqbl{3}{0};        \node at (3.5,-0.5)  {\scalebox{0.7}{$k+2$}};
                        \node at (5.1,-0.35) {$\ddots$};
    \sqbl{5}{-1};       \node at (5.5,-1.5)  {\scalebox{0.7}{$k'$}};
    \sqbl{5.5}{-2};     \node at (6,-2.5)    {\scalebox{0.7}{$k'+1$}};

    \sqorhash{2.5}{11}; \node at (3,10.5)    {\scalebox{0.7}{$k+1$}};
                        \node at (2.25,11.6) {$\ddots$};
    \sqor{1}{13};       \node at (1.5,12.5)  {\scalebox{0.7}{$3$}};
    \sqor{0.5}{14};     \node at (1,13.5)    {\scalebox{0.7}{$2$}};
    \sqor{0}{15};       \node at (0.5,14.5)  {\scalebox{0.7}{$1$}};
    \sqorhash{3}{12};   \node at (3.5,11.5)  {\scalebox{0.7}{$k+2$}};
                        \node at (2.75,12.8) {$\ddots$};
    \sqor{1.5}{14};     \node at (2,13.5)    {\scalebox{0.7}{$4$}};
    \sqor{1}{15};       \node at (1.5,14.5)  {\scalebox{0.7}{$3$}};
    \sqor{0.5}{16};     \node at (1,15.5)    {\scalebox{0.7}{$2$}};
                        \node at (2,16.5)    {$\iddots$};
                        \node at (4.5,12.5)  {$\iddots$};
                        \node at (2.75,16)   {$\iddots$};
                        \node at (5.1,15.6)  {$\ddots$};
                        \node at (4,14.5)    {$\cdots$};
                        \node at (4,13.5)    {$\cdots$};
    \sqor{3}{18};       \node at (3.5,17.5)  {\scalebox{0.7}{$m+1$}};
    \sqor{3.5}{17};     \node at (4,16.5)    {\scalebox{0.7}{$m+2$}};
    \sqor{5.5}{15};     \node at (6,14.5)    {\scalebox{0.7}{$k'+1$}};
    \sqor{5}{14};       \node at (5.5,13.5)  {\scalebox{0.7}{$k'$}};
                    
    \sqgrhash{2.5}{9};  \node at (3,8.5)     {\scalebox{0.7}{$k+1$}};
    \sqgr{2.5}{1};      \node at (3,0.5)     {\scalebox{0.7}{$k+1$}};
    \sqgr{2}{2};        \node at (2.5,1.5)   {\scalebox{0.7}{$k$}};
                        \node at (1.75,6.6)  {$\iddots$};
                        \node at (1.75,2.5)  {$\ddots$};
    \sqgr{3}{0};        \node at (3.5,-0.5)  {\scalebox{0.7}{$k+2$}};
    \sqgr{4.5}{-4};     \node at (5,-4.5)    {\scalebox{0.7}{$k'-1$}};
    \sqgr{5}{-3};       \node at (5.5,-3.5)  {\scalebox{0.7}{$k'$}};
    \sqgr{4.5}{-2};     \node at (5,-2.5)    {\scalebox{0.7}{$k'-1$}};
    \sqgr{5.5}{-4};     \node at (6,-4.5)    {\scalebox{0.7}{$k'+1$}};
                        \node at (5,-5.5)    {$\iddots$};
    \sqgr{0.5}{4};      \node at (1,3.5)     {\scalebox{0.7}{$2$}};
    \sqgr{1}{-3};       \node at (1.5,-3.5)  {\scalebox{0.7}{$3$}};
    \sqgr{0.5}{-4};     \node at (1,-4.5)    {\scalebox{0.7}{$2$}};
    \sqgr{1.5}{-2};     \node at (2,-2.5)    {\scalebox{0.7}{$4$}};
                        \node at (2,-5.5)    {$\ddots$};
    \sqgr{3}{-7};       \node at (3.5,-7.5)  {\scalebox{0.7}{$m+1$}};
    \sqgr{2.5}{-6};     \node at (3,-6.5)    {\scalebox{0.7}{$m$}};
    \sqgr{3.5}{-6};     \node at (4,-6.5)    {\scalebox{0.7}{$m+2$}};
                        \node at (2.7,-1.4)  {$\iddots$};
                        \node at (3.5,-2.5)  {$\cdots$};
                        \node at (3.5,-3.5)  {$\cdots$};
                        \node at (3.5,-4.5)  {$\cdots$};
                        \node at (4.4,-1.4)  {$\ddots$};
\end{tikzpicture}
\caption{}\label{fig:permheap2}
\end{subfigure}
\end{tabular}
\caption{The heaps for Lemma~\ref{lem:swap}.}\label{}
\end{figure}
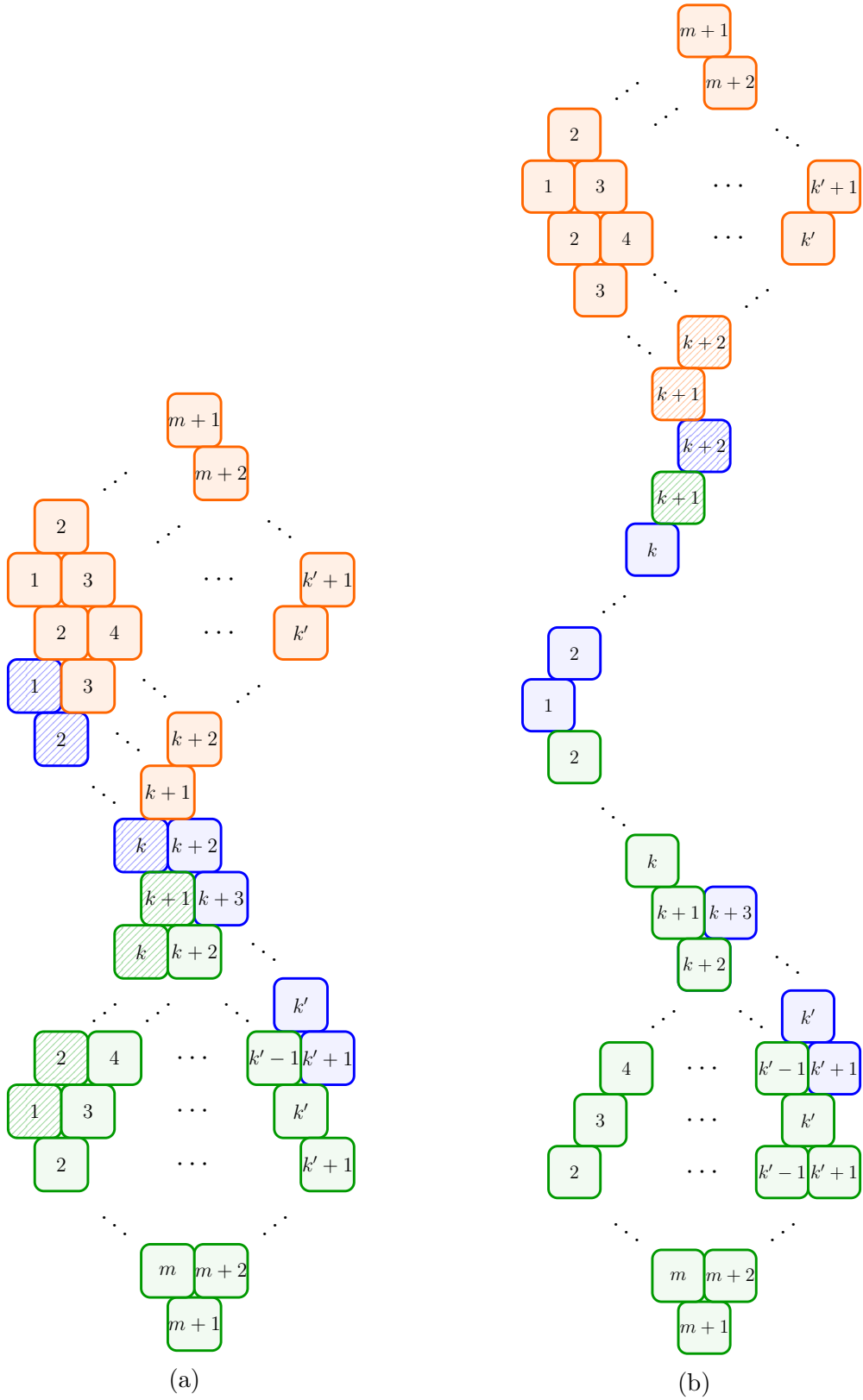
\begin{figure}[htpb] \centering
\ContinuedFloat
\begin{tabular}{cc}
\begin{subfigure}{0.45\textwidth} \centering
\begin{tikzpicture}
    \sqbl{0}{13};   \node at (0.5,12.5)  {\scalebox{0.85}{$1$}};
                    \node at (5,3.5)     {$\iddots$};
    \sqgr{0.5}{12}; \node at (1,11.5)    {\scalebox{0.85}{$2$}};
    
    \sqbl{3}{10};   \node at (3.5,9.5)   {\scalebox{0.85}{$k+2$}};
    \sqbl{3.5}{9};  \node at (4,8.5)     {\scalebox{0.85}{$k+3$}};
                    \node at (4.8,7.65)  {$\ddots$};
    \sqbl{5}{7};    \node at (5.5,6.5)   {\scalebox{0.85}{$k'$}};
    \sqbl{5.5}{6};  \node at (6,5.5)     {\scalebox{0.85}{$k'+1$}};

    \sqor{2.5}{11}; \node at (3,10.5)    {\scalebox{0.85}{$k+1$}};
    \sqor{3}{12};   \node at (3.5,11.5)  {\scalebox{0.85}{$k+2$}};
                    \node at (2.25,11.6) {$\ddots$};
                    \node at (3,12.7) {$\ddots$};
    \sqor{1}{13};   \node at (1.5,12.5)  {\scalebox{0.85}{$3$}};
    \sqor{0.5}{14}; \node at (1,13.5)    {\scalebox{0.85}{$2$}};
    \sqor{0}{15};   \node at (0.5,14.5)  {\scalebox{0.85}{$1$}};
    \sqor{1}{15};   \node at (1.5,14.5)  {\scalebox{0.85}{$3$}};
    \sqor{1.5}{14}; \node at (2,13.5) {\scalebox{0.85}{$4$}};
                    \node at (4.7,12.6)  {$\iddots$};
                    \node at (2.3,15.6)  {$\iddots$};
                    \node at (3,14.8)    {$\iddots$};
                    \node at (5.2,14.8)  {$\ddots$};
                    \node at (4,13.5)  {$\cdots$};
    \sqor{3}{17};   \node at (3.5,16.5)  {\scalebox{0.85}{$m+1$}};
    \sqor{3.5}{16}; \node at (4,15.5)    {\scalebox{0.85}{$m+2$}};
    \sqor{5.5}{14}; \node at (6,13.5)    {\scalebox{0.85}{$k'+1$}};
    \sqgr{2.5}{9};  \node at (3,8.5)     {\scalebox{0.85}{$k+1$}};
    \sqgr{2}{10};   \node at (2.5,9.5)   {\scalebox{0.85}{$k$}};
                    \node at (1.75,10.5) {$\ddots$};
    \sqgr{3}{8};    \node at (3.5,7.5)   {\scalebox{0.85}{$k+2$}};
                    
    \sqgr{4.5}{4};  \node at (5,3.5)     {\scalebox{0.85}{$k'-1$}};
    \sqgr{5}{5};    \node at (5.5,4.5)   {\scalebox{0.85}{$k'$}};
    \sqgr{4.5}{6};  \node at (5,5.5)     {\scalebox{0.85}{$k'-1$}};
    \sqgr{5.5}{4};  \node at (6,3.5)     {\scalebox{0.85}{$k'+1$}};

    \sqgr{1}{5};    \node at (1.5,4.5)   {\scalebox{0.85}{$3$}};
    \sqgr{0.5}{4};  \node at (1,3.5)     {\scalebox{0.85}{$2$}};
    \sqgr{1.5}{6};  \node at (2,5.5)     {\scalebox{0.85}{$4$}};
                    \node at (2,2.5)     {$\ddots$};
                    \node at (5,2.5)     {$\iddots$};
    \sqgr{3}{1};    \node at (3.5,0.5)   {\scalebox{0.85}{$m+1$}};
    \sqgr{2.5}{2};  \node at (3,1.5)     {\scalebox{0.85}{$m$}};
    \sqgr{3.5}{2};  \node at (4,1.5)     {\scalebox{0.85}{$m+2$}};
                    \node at (2.8,6.5)   {$\iddots$};
                    \node at (3.5,5.5)   {$\cdots$};
                    \node at (3.5,4.5)   {$\cdots$};
                    \node at (3.5,3.5)   {$\cdots$};
                    \node at (4.3,6.5)   {$\ddots$};
\end{tikzpicture}
\caption{}\label{fig:permheap3}
\end{subfigure} &
\begin{subfigure}{0.45\textwidth} \centering \vspace{30pt}
\begin{tikzpicture}
    \sqor{0}{13};   \node at (0.5,12.5)  {\scalebox{0.75}{$1$}};
    \sqor{0.5}{12}; \node at (1,11.5)    {\scalebox{0.75}{$2$}};
    \sqor{2.5}{9};  \node at (3,8.5)     {\scalebox{0.75}{$k+1$}};
    \sqor{2}{10};   \node at (2.5,9.5)   {\scalebox{0.75}{$k$}};
                    \node at (1.85,10.6) {$\ddots$};
    \sqbl{3}{8};    \node at (3.5,7.5)   {\scalebox{0.75}{$k+2$}};
                    \node at (4.25,6.75) {$\ddots$};
    \sqbl{5}{5};    \node at (5.5,4.5)   {\scalebox{0.75}{$k'$}};
    \sqbl{4.5}{6};  \node at (5,5.5)     {\scalebox{0.75}{$k'-1$}};
    \sqbl{5.5}{4};  \node at (6,3.5)     {\scalebox{0.75}{$k'+1$}};
    \sqgrhash{3}{6};    \node at (3.5,5.5)   {\scalebox{0.75}{$m+2$}};
    \sqgrhash{2.5}{7};  \node at (3,6.5)     {\scalebox{0.75}{$m+1$}};
                    \node at (4.25,4.75) {$\ddots$};
    \sqgrhash{5}{3};    \node at (5.5,2.5)   {\scalebox{0.75}{$k'$}};
    \sqgrhash{4.5}{4};  \node at (5,3.5)     {\scalebox{0.75}{$k'-1$}};
    \sqgrhash{5.5}{2};  \node at (6,1.5)     {\scalebox{0.75}{$k'+1$}};
    \sqgrhash{5}{1};    \node at (5.5,0.5)   {\scalebox{0.75}{$k'$}};
                    \node at (4.75,-0.5) {$\iddots$};
    \sqgrhash{3.5}{-1}; \node at (4,-1.5)    {\scalebox{0.75}{$m+2$}};
    \sqgrhash{3}{-2};   \node at (3.5,-2.5)  {\scalebox{0.75}{$m+1$}};
\end{tikzpicture}
\caption{}\label{fig:permheap4}
\end{subfigure}
\end{tabular}
\caption{The heaps for Lemma~\ref{lem:swap} (continued).}\label{}
\end{figure}
\begin{figure}[htpb] \centering
\ContinuedFloat
\begin{tabular}{cc} \centering
\begin{subfigure}{0.5\textwidth} \centering
\begin{tikzpicture}
    \sqor{0}{13};   \node at (0.5,12.5)  {\scalebox{0.75}{$1$}};
    \sqor{0.5}{12}; \node at (1,11.5)    {\scalebox{0.75}{$2$}};
    \sqor{2.5}{9};  \node at (3,8.5)     {\scalebox{0.75}{$k+1$}};
    \sqor{2}{10};   \node at (2.5,9.5)   {\scalebox{0.75}{$k$}};
                    \node at (1.85,10.6) {$\ddots$};
    \sqbl{3}{8};    \node at (3.5,7.5)   {\scalebox{0.75}{$k+2$}};
                    \node at (4.4,6.6) {$\ddots$};
    \sqbl{5}{6};    \node at (5.5,5.5)   {\scalebox{0.75}{$k'$}};
    \sqblcheck{5.5}{5};  \node at (6,4.5) {\scalebox{0.75}{$k'+1$}};
    \sqgrcheck{5.5}{4};  \node at (6,3.5) {\scalebox{0.75}{$k'+1$}};
    \sqgr{3}{1};    \node at (3.5,0.5)    {\scalebox{0.75}{$m+2$}};
    \sqgr{2.5}{0};  \node at (3,-0.5)     {\scalebox{0.75}{$m+1$}};
                    \node at (4.4,1.7)    {$\iddots$};
    \sqgr{5}{3};    \node at (5.5,2.5)    {\scalebox{0.75}{$k'$}};
    \sqgr{5.5}{-4}; \node at (6,-4.5)     {\scalebox{0.75}{$k'+1$}};
    \sqgr{5}{-3};   \node at (5.5,-3.5)   {\scalebox{0.75}{$k'$}};
                    \node at (4.4,-2.5)   {$\ddots$};
    \sqgr{3}{-1};   \node at (3.5,-1.5)   {\scalebox{0.75}{$m+2$}};
\end{tikzpicture}
\caption{}\label{fig:permheap5}
\end{subfigure} &
\begin{subfigure}{0.5\textwidth} \centering \vspace{4.3in}
\begin{tikzpicture}
    \sqor{0}{3};    \node at (0.5,2.5)   {\scalebox{0.85}{$1$}};
    \sqor{0.5}{2};  \node at (1,1.5)     {\scalebox{0.85}{$2$}};
                    \node at (1.5,0.7)   {$\ddots$};
    \sqor{1.5}{0};  \node at (2,-0.5)    {\scalebox{0.85}{$m$}};
    \sqgr{2.5}{0};  \node at (3,-0.5)    {\scalebox{0.85}{$m+2$}};
    \sqgr{3}{-1};   \node at (3.5,-1.5)  {\scalebox{0.85}{$m+3$}};
                    \node at (4.25,-2.25){$\ddots$};
    \sqgr{4}{-3};   \node at (4.5,-3.5)  {\scalebox{0.85}{$k'+1$}};
\end{tikzpicture}
\caption{}\label{fig:permheap6}
\end{subfigure}
\end{tabular}
\caption{The heaps for Lemma~\ref{lem:swap} (continued).}
\end{figure}
\end{center}

    We are now ready to prove Theorem~\ref{thm:conjiffring}.

\begin{proof}[Proof of Theorem~\ref{thm:conjiffring}] Suppose $w,y \in \CFC(A_n)$ are conjugate.
    Note that every chunk of size $\ell$ in $W(A_n)$ corresponds to a cycle of length $\ell+1$ with connected support in $S_{n+1}$.
    In particular, the chunk that corresponds to the group element $(k)(k+1) \cdots (k+m)$ corresponds to the $(m+2)$-cycle $(k~k+1~ \cdots ~ k+m~k+m+1)$.
    By assumption, as permutations, $w$ and $y$ in $S_{n+1}$ have the same cycle type. Suppose $w$ and $y$ each consist of products of disjoint cycles of lengths $k_{1}, k_{2}, \ldots, k_{s}$. In this case, it is not possible for $H(w)$ and $H(y)$ to have a different number of chunks. Furthermore, there are $n$ chunks of size $k$ in $H(w)$ if and only if there are $n$ chunks of size $k$ in $H(y)$. 
    Then both of $H(w)$ and $H(y)$ consist of chunks of sizes $k_{1}-1,k_{2}-1,\ldots,k_{s}-1$. That is, for every ring $R$ in $\hat{H}(w)$, there is a corresponding ring $R'$ in $\hat{H}(y)$. Then, we can permute and slide rings in $\hat{H}(w)$ as necessary to obtain $\hat{H}(y)$.
    Hence $\hat{H}(w)$ and $\hat{H}(y)$ are ring equivalent.

    Now, suppose $\hat{H}(w)$ and $\hat{H}(y)$ are ring equivalent. Then there exists some sequence of cyclic shifts, slides, and permutations of chunks that takes $\hat{H}(w)$ to $\hat{H}(y)$. We can perform these operations via conjugation, as in Lemmas~\ref{lem:slide} and~\ref{lem:swap} and Remarks~\ref{rem:slide} and~\ref{rem:permute}. Hence $w$ and $y$ are conjugate.
\end{proof}

    In the future, we hope to be able to generalize the notion of chunks and rings to CFC elements of Coxeter groups of types other than $A_n$ in order to have a result analogous to Theorem~\ref{thm:conjiffring}.
    We will need a different proof for an analogous theorem in Coxeter group of types other than $A_n$ since we used cycle type in the argument for the forward direction of the proof of Theorem~\ref{thm:conjiffring}.

\bibliographystyle{plain}
\bibliography{References}

\end{document}